\documentclass[12pt,leqno]{article}
\usepackage{amssymb,amsmath,amsfonts,amsbsy, amsthm, xspace, latexsym, amscd, fancyhdr}

\swapnumbers
\theoremstyle{definition}
\newtheorem{definition}{Definition}[section]
\newtheorem{remark}[definition]{Remark}
\newtheorem{example}[definition]{Example}
\newtheorem{examples}[definition]{Examples}
\newtheorem{notation}[definition]{Notation}

\theoremstyle{theorem}
\newtheorem{theorem}[definition]{Theorem}
\newtheorem{proposition}[definition]{Proposition}
\newtheorem{lemma}[definition]{Lemma}
\newtheorem{corollary}[definition]{Corollary}

\newcommand{\ifff}{{\em i\,f\,{f}\;  }}

\newcommand{\Proof}{\noindent \textbf{Proof:\;}}

\newcommand{\bra}{\ensuremath{\large\langle}}

\newcommand{\ket}{\ensuremath{\large\rangle}}

\newcommand{\supp}{\ensuremath{{\rm{supp}}}}

\newcommand{\card}{\ensuremath{{\rm{card}}}}

\newcommand{\st}{\ensuremath{{{\rm{st}}}}}

\newcommand{\eqdef}{\ensuremath{\overset{\text{def}}{=}}}

\newcommand{\embed}{\ensuremath{\hookrightarrow}}

\newcommand{\rest}{\ensuremath{\upharpoonright}}

\begin{document} 
\title{Lecture Notes: Non-Standard Approach to J.F. Colombeau's Theory of Generalized 
Functions$^*$\\
\small{University of Vienna, Austria, May 2006.}}

\author{Todor D. Todorov\\ 
                        Mathematics Department\\                
                        California Polytechnic State University\\
                        San Luis Obispo, California 93407, USA\\
																							(ttodorov@calpoly.edu)\\}
\date{}
\maketitle	
\begin{abstract} In these lecture notes we present an introduction to
non-standard analysis especially written for the community of mathematicians,
physicists and engineers who do research on J. F. Colombeau' theory of new generalized
functions and its applications. The main purpose of our non-standard approach to
Colombeau' theory is the improvement of the properties of the scalars of the varieties of spaces of
generalized functions: in our non-standard approach the sets of scalars of the functional spaces
always form algebraically closed non-archimedean Cantor complete fields. In contrast, the scalars
of the functional spaces in Colombeau's theory are rings with zero divisors. The improvement of the
scalars leads to other improvements and simplifications of Colombeau's theory such as reducing the
number of quantifiers and possibilities for an axiomatization of the theory. 
Some of the algebras we construct in these notes have already counterparts in Colombeau's theory,
other seems to be without counterpart. We present applications of the theory to PDE and
mathematical physics. Although our approach is directed mostly to Colombeau's community, the
readers who are already familiar with non-standard methods might also find a short
and comfortable way to learn about Colombeau's theory: a new branch of functional analysis which
naturally generalizes the Schwartz theory of distributions with numerous applications to partial
differential equations, differential geometry, relativity theory and other areas of mathematics
and physics.  
\end{abstract}

$^*$Work supported by START-project Y237 of the Austrian Science Fund\newline

MSC: Functional Analysis (46F30); Generalized Solutions of PDE (35D05).
\newpage
\section{Introduction}
This lecture notes are an extended version of the several lectures I gave at the
University of Vienna during my visit in the Spring of 2006. My audience consisted mostly of
colleagues, graduate and undergraduate students who do research on J.F. Colombeau's non-linear
theory of generalized functions (J.F. Colombeau's~(\cite{jCol84a}-\cite{jCol91}) and its
applications to ordinary and partial differential equations, differential geometry, relativity
theory and mathematical physics. With very few exceptions the colleagues attended my talks were not
familiar with nonstandard analysis.  This fact strongly influenced the nature of my lectures and
these lecture notes. I do not assume that the reader of these notes is necessarily familiar
neither with A. Robonson's non-standard analysis (A. Robonson~\cite{aRob66}) nor with  A.
Robonson's non-standard asymptotic analysis  (A. Robinson~\cite{aRob73} and A. Robonson and A.H.
Lightstone~~\cite{LiRob}). I have tried to downplay the role of mathematical logic as much as
possible. With examples from Colombeau's theory I tried to convince my colleagues that the
involvement of the non-standard methods in Colombeau theory has at least the following three
advantages: 
\begin{enumerate}
\item The scalars of the non-standard version of Colombeau's theory are algebraically closed
Cantor complete fields. Recall that in Colmbeau's theory 
the scalars of the functional spaces are rings with zero divisors.

\item The involvement of non-standard analysis in Colombeau's theory leads to simplification of
the theory by reducing the number of the quantifiers. This should be not of surprise because
non-standard analysis is famous with the so called reduction of quantifiers. For comparison, the
familiar definition of a limit of a function in standard analysis involves three (non-commuting)
quantifiers. In contrast, its non-standard characterization uses only one quantifier. Another
example gives the definition of a compact set in point set topology involves at least two
quantifiers. In contrast, there is a free of quantifiers non-standard characterization of the
compactness in terms of monads. Since Colombeau' theory is relatively heavy of
quantifiers, the reduction of the numbers of quantifiers makes the theory more attractive to
colleagues outside the Colombeau's community and in particular to theoretical physicists.

\item  In my lectures and in these notes I decided to follow
mostly the so called constructive version of the non-standard analysis where the non-standard real
number 
$a\in{^*\mathbb{R}}$ is equivalence class of families $(a_i)$ in the ultrapower
$\mathbb{R}^\mathcal{I}$ for some infinite set $\mathcal{I}$. Similarly, every non-standard
smooth function $f\in{^*\mathcal{E}}(\Omega)$ is defined as equivalence class of families $(f_i)$
in the ultrapower $\mathcal{E}(\Omega)^\mathcal{I}$. Here
$\mathcal{E}(\Omega)$ is a (short) notation for $\mathcal{C}^\infty(\Omega)$. The equivalence
relation in both $\mathbb{R}^\mathcal{I}$  and $\mathcal{E}(\Omega)^\mathcal{I}$ is defined in
terms of a free ultrafilter $\mathcal{U}$ on $\mathcal{I}$. In our approach the choice of the index
set $\mathcal{I}$ and the choice of the ultrafilter $\mathcal{U}$ are borrowed from Colombeau's
theory. This approach to non-standard analysis is more directly connected with the standard (real)
analysis and allow to involve the non-standard analysis in research with comparatively limited
knowledge in the non-standard theory. The non-standard analysis however has also axiomatic version
based on two axioms known a Saturation Principle and Transfer Principle. The involvement of
non-standard analysis, if based on these two principles, opens the opportunities for
axiomatization of Colombeau's theory. I have demonstrated this in the notes by presenting a couple
of proofs to several theorems: one using families (nets), and another using these two axioms. The
first might be more convincible for beginners to non-standard analysis but the second proofs are
more elegant and short because it does not involve the representatives of the generalized numbers
and generalized functions.
\end{enumerate} 

Let  $\mathcal{T}$ stand for
the usual topology on $\mathbb{R}^d$. J.F. Colombeau's non-linear theory of generalized functions is
based on varieties of families of differential commutative
rings $\mathcal{G}\eqdef\{\mathcal{G}(\Omega)\}_{\Omega\in\mathcal{T}}$ such that:
1) Each $\mathcal{G}$ is a {\bf sheaf} of differential rings (consequently, each
$f\in\mathcal{G}(\Omega)$ has a {\bf support} which is a closed set of $\Omega$). 2) Each
$\mathcal{G}(\Omega)$ is supplied with a chain of sheaf-preserving embeddings
$\mathcal{C}^\infty(\Omega)\subset\mathcal{D}^\prime(\Omega)\subset\mathcal{G}(\Omega)$,
where $\mathcal{C}^\infty(\Omega)$ is a {\bf differential subring} of
$\mathcal{G}(\Omega)$ and the space of L. Schwartz's distributions
$\mathcal{D}^\prime(\Omega)$ is a {\bf differential linear subspace} of $\mathcal{G}(\Omega)$. 3)
The ring of the scalars $\widetilde{\mathbb{C}}$ of the family $\mathcal{G}$  (defined as the set of
the functions in $\mathcal{G}(\mathbb{R}^d)$ with zero gradient) is a non-Archimedean ring with zero
devisors containing a copy of the complex numbers $\mathbb{C}$. Colombeau theory has numerous
applications to ordinary and partial differential equations, fluid mechanics, elasticity theory, quantum
field theory and more recently to general relativity. 

\newpage

\section{$\kappa$-Good Two Valued Measures}\label{S: kappa-Good Two Valued Measures}

	I follow the  philosophy that every non-standard real
number $a\in{^*\mathbb{R}}$ is, roughly speaking, a family $(a_i)$ in the ultrapower
$\mathbb{R}^\mathcal{I}$ for some infinite set $\mathcal{I}$. Similarly, every nonstandard
smooth function $f\in{^*\mathcal{E}}(\Omega)$ is again, roughly speaking, a  family $(f_i)$
in the ultrapower $\mathcal{E}(\Omega)^\mathcal{I}$. Here
$\mathcal{E}(\Omega)$ is a (short) notation for $\mathcal{C}^\infty(\Omega)$. 

\begin{definition}[$\kappa$-Good Two Valued Measures]\label{D: kappa-Good Two Valued Measures}
Let $\mathcal{I}$ be an infinite set of cardinality $\kappa$, i.e. $\card(\mathcal{I})=\kappa$. A
mapping  $p:\mathcal{P}(\mathcal{I})\to\{0, 1\}$ is a {\bf $\kappa$-good two-valued
(probability) measure} if
\begin{enumerate}
\item $p$ is finitely additive, i.e. $p(A\cup B)=p(A)+p(B)$ for disjoint $A$ and $B$.
\item $p(\mathcal{I})=1$.
\item $p(A)=0$ for finite $A$.
\item There exists a sequence of sets $(\mathcal{I}_n)$ such that
\begin{enumerate}
\item{\bf (a)} $\mathcal{I}\supset\mathcal{I}_1\supset\mathcal{I}_2\supset\dots$,
\item{\bf (b)} $\mathcal{I}_n\setminus\mathcal{I}_{n-1}\not=\varnothing$ for all $n$,
\item{\bf (c)} $\bigcap_{n=1}^\infty\mathcal{I}_n=\varnothing$, 
\item{\bf (d)} $p(\mathcal{I}_n)=1$ for all $n$.
\end{enumerate}
\item If $\mathcal{I}$ is uncountable, we impose one more property: $p$ should be
{\bf $\kappa$-good} in the sense that for every set $\Gamma\subseteq \mathcal{I}$, with
$\card(\Gamma)\leq\kappa$, and every {\bf reversal} $R: \mathcal{P}_\omega(\Gamma)\to\mathcal{U}$
there exists a {\bf strict reversal} 
$S: \mathcal{P}_\omega(\Gamma)\to\mathcal{U}$ such that $S(X)\subseteq R(X)$ for all 
$X\in\mathcal{P}_\omega(\Gamma)$. Here $\mathcal{P}_\omega(\Gamma)$ denotes the set of all finite
subsets of $\Gamma$ and $\mathcal{U}=\{A\in\mathcal{P}(\mathcal{I}) \mid P(A)=1\}$.
\end{enumerate}
\end{definition}

\begin{remark}[Reversals]\label{R: Reversals} Let $\Gamma\subseteq I$. A function
$R:\mathcal{P}_\omega(\Gamma)\to\mathcal{U}$ is called a {\bf reversal}  if $X\subseteq Y$ implies
$R(X)\supseteq R(Y)$ for every $X, Y\in\mathcal{P}_\omega(\Gamma)$. A function $S:
\mathcal{P}_\omega(\Gamma)\to\mathcal{U}$ is called a {\bf strict reversal}  if $S(X\cup Y)=
S(X)\cap S(Y)$ for every $X, Y\in\mathcal{P}_\omega(\Gamma)$. It is clear that every strict
reversal is a reversal (which justifies the terminology). 
\end{remark}
\newpage
\section{Existence of Two Valued $\kappa$-Good Measures}
\begin{theorem}[Existence of Two Valued $\kappa$-Good Measures]\label{T: Existence of Two Valued
kappa-Good Measures} Let $\mathcal{I}$ be an infinite set and let $(\mathcal{I}_n)$ be a
sequence of sets with the properties (a)-(c) (think of Colombeau's theory).
Then there exists a two valued $\kappa$-good measure $p:\mathcal{P}(\mathcal{I})\to\{0, 1\}$, where
$\kappa=\card(\mathcal{I})$, such that
$p(\mathcal{I}_n)=1$ for all
$n\in\mathbb{N}$. 
\end{theorem} 
\begin{remark} We should note that for every infinite set $\mathcal{I}$ there exists
a sequence $(\mathcal{I}_n)$ with the properties (a)-(c).
\end{remark}
\Proof {\em Step 1:} {\bf Define} $\mathcal{F}_0\subset\mathcal{P}(\mathcal{I})$ by
\[
\mathcal{F}_0=\{A\in\mathcal{P}(\mathcal{I}) \mid \mathcal{I}_n\subseteq A \text{\; for some\; }
n\}.
\]
It is easy to check that $\mathcal{F}_0$ {\bf is a free countably incomplete filter on
}$\mathcal{I}$ in the sense that
$\mathcal{F}_0$ has the following properties:
\begin{enumerate}
\item $\varnothing\notin\mathcal{F}_0$.
\item $\mathcal{F}_0$ is closed under finite intersections.
\item $\mathcal{F}_0\ni A\subseteq B\in\mathcal{P}(\mathcal{I})$ implies
$B\in\mathcal{F}_0$.
\item $\mathcal{I}_n\in\mathcal{F}_0$ for all $n\in\mathbb{N}$.
\end{enumerate}

	{\em Step 2:} {\bf We extend $\mathcal{F}_0$ to a ultrafilter
$\mathcal{U}$} on $\mathcal{I}$ by Zorn lemma: Let $\mathcal{L}$ denote the set of all free filter
$\mathcal{F}$ on $\mathcal{I}$ containing $\mathcal{I}_n$, i.e.
\[
\mathcal{L}= \{\mathcal{F}\subset\mathcal{P}(\mathcal{I})\mid \mathcal{F} \text{\, satisfies
(i)-(iv)}, \text{\, where\, } \mathcal{F}_0 \text{\, should be replaced by\,}
\mathcal{F}\}.
\]
We shall order
$\mathcal{L}$ by inclusion $\subset$. Observe that every chain $L$ in $\mathcal{L}$ is bounded from
above by $\bigcup_{A\in L}A$ and it is not difficult to show that
$\bigcup_{A\in L}A\in\mathcal{L}$. Thus $\mathcal{L}$ has maximal elements $\mathcal{U}$ by Zorn
lemma. In what follows we shall keep $\mathcal{U}$ fixed.

	{\em Step 3:} We shall prove now that {\bf $\mathcal{U}$ has the following (free ultrafilter)
properties:}
\begin{enumerate}
\item $\varnothing\notin\mathcal{U}$.
\item $\mathcal{U}$ is closed under finite intersections.
\item $\mathcal{U}\ni A\subseteq B\in\mathcal{P}(\mathcal{I})$ implies
$B\in\mathcal{U}$.
\item{\bf (4)} $\mathcal{I}_n\in\mathcal{U}$ for all $n\in\mathbb{N}$.
\item{\bf (5)} $A\cup B\in\mathcal{U}$ implies either $A\in\mathcal{U}$ 
or $B\in\mathcal{U}$. 
\end{enumerate}
	Indeed, $\mathcal{U}$ satisfies (1)-(4) by the choice of $\mathcal{U}$ since
$\mathcal{U}\in\mathcal{L}$. To show the property (5), suppose (on the contrary) that $A\cup
B\in\mathcal{U}$ and $A, B\notin\mathcal{U}$ for some subsets $A$ and $B$ of
$\mathcal{I}$. Next, we observe that $\mathcal{F}_A=\{X\in\mathcal{P}(\mathcal{I})\mid A\cup
X\in\mathcal{U}\}$ is also a free filter on $\mathcal{I}$ (i.e. $\mathcal{F}_A$ satisfies the
properties (1)-(4)). Next, we observe that $\mathcal{F}_A$ is a proper extension of
$\mathcal{U}$ since
$B\in\mathcal{F}_A\setminus\mathcal{U}$ by the assumption for $B$, contradicting 
the maximality of\, $\mathcal{U}$.

{\em Step 4:} {\bf Define} $p: \mathcal{P}(\mathcal{I})\to\{0, 1\}$ by
$p(A)=1$ whenever $A\in\mathcal{U}$ and $p(A)=0$ whenever $A\notin\mathcal{U}$. We have to show now
that $p$ is a $\kappa$-good two valued measure (Definition~\ref{D: kappa-Good Two Valued
Measures}). To check the finite additivity property (i) of
$p$, suppose that
$A\cap B=\varnothing$ for some 
$A, B\in\mathcal{P}(\mathcal{I})$. Suppose, first, that
$A\cup B\in\mathcal{U}$, so we have $p(A\cup B)=1$. On the other hand, by properties (1) and
(5), exactly one of the following two statements is true: either (a) $A\in\mathcal{U}$ and
$A\notin\mathcal{U}$ or (b) $A\notin\mathcal{U}$ and $A\in\mathcal{U}$. In either case we have
$p(A)+p(B)=1$, as required. Suppose, now, that $A\cup B\notin\mathcal{U}$, so we have $p(A\cup
B)=0$. In this case we have $A\notin\mathcal{U}$ and $B\notin\mathcal{U}$ by property (3). Thus
$p(A)+p(B)=0$. The property (ii): $p(\mathcal{I})=1$ holds since $\mathcal{I}\in\mathcal{U}$ by
properties (3) and (4) of $\mathcal{U}$. To prove the property (iii), suppose (on the contrary) that
$p(A)=1$ for some finite set
$A\subset\mathcal{I}$, i.e. $A\in\mathcal{U}$. It follows that there exists $i\in A$ such that
$\{i\}\in\mathcal{U}$ by property (5) of $\mathcal{U}$ since we have $\bigcup_{i\in A}\{i\}=A$.
Thus
$\{i\}\in\mathcal{I}_n$ for all
$n\in\mathcal{N}$ by properties (1), (2) and (4) of $\mathcal{U}$. It follows that 
$\{i\}\in\bigcap_{n\in\mathbb{N}}\mathcal{I}_n$ contradicting property (c) of the sequence
$(\mathcal{I}_n)$. The property (iv) holds by the choice of
$\mathcal{U}$ since
$\mathcal{I}_n\in\mathcal{F}_0\subset\mathcal{U}$ thus $p(\mathcal{I}_n)=1$. For the proof of the
property (v) of the measure
$p$ we shall refer to C. C. Chang and H. J. Keisler~\cite{CKeis} or to T. Lindstr\o m~\cite{tLin}.  
$\blacktriangle$
\section{A Non-Standard Analysis: The General Theory}\label{S: A Non-Standard Analysis: The General
Theory}
\begin{definition}[A Non-Standard Extension of a Set]\label{D: A Non-Standard Extension of a Set}
Let $S$ be a set and $\mathcal{I}$ be and infinite set, and $S^\mathcal{I}$ be the corresponding
ultrapower. 
\begin{enumerate}
\item We say that $(a_i)$ and
$(b_i)$ are equal almost everywhere in $\mathcal{I}$, in symbol $a_i=b_i$ a.e., if
$p(\{i\in\mathcal{I}\mid a_i=b_i\text{\, in\, } S\})=1$, or equivalently, if 
$\{i\in\mathcal{I}\mid a_i=b_i \text{\, in\, } S\}\in\mathcal{U}$, where
$\mathcal{U}=\{A\in\mathcal{P}(\mathcal{I})\mid p(A)=1\}$. We denote by $\sim$ the corresponding
equivalence relation, i.e. $(a_i)\sim(b_i)$ if $a_i=b_i$ a.e.. 
\item We denote by
$\bra a_i\ket$ the equivalence class determined by $(a_i)$. The set of all equivalence classes
$^*S=S^\mathcal{I}/_\sim$ is called a {\bf non-standard extension} of $S$.
\item Let $s\in S$. We define $^*s=\bra a_i \ket$, where $a_i=s$ for all
$i\in\mathcal{I}$. We define the canonical embedding $\sigma: S\to{^*S}$ by 
$\sigma(s)={^*s}$, and denote by $^\sigma S=\{^*s\mid s\in S\}$ the range of $\sigma$. We shall
sometimes treat this embedding as an inclusion,
$S\subseteq{^*S}$, by letting $s={^*s}$ for all $s\in S$. 
\item More generally, if $X\subseteq S$, we define $^*X\subseteq{^*S}$ by 
\[
^*X=\{\bra x_i\ket\in{^*S}\mid x_i\in X \text{\, a.e.}\}.
\]
 We have $X\subseteq{^*X}$ under the
embedding
$x\to{^*x}$. We say that $^*X$ is the {\bf non-standard extension} of $X$.
\end{enumerate}
\end{definition}
\begin{theorem}[Axiom 1. Extension Principle]\label{T: Extension Principle} Let $S$ be a set. Then
$S\subseteq{^*S}$ and 
$S={^*S}$
\ifff
$S$ is a finite set. 

\end{theorem}
\begin{proof} $S\subseteq{^*S}$ holds in the sense of the embedding $\sigma$. Suppose, first, that $S$ is a
finite set and let 
$\bra a_i\ket\in{^*S}$.  We observe that the finite collection of sets
$\{i\in\mathcal{I}\mid a_i=s\}$, $s\in S$, are mutually disjoint and $\bigcup_{s\in
S}\{i\in\mathcal{I}\mid a_i=s\}=\mathcal{I}$. Thus $\sum_{s\in S}p(\{i\in\mathcal{I}\mid
a_i=s\})=1$ by the finite additivity of the measure $p$. It follows that there exists a unique
$s_0\in S$ such that
$p(\{i\in\mathcal{I}\mid a_i=s_0\})=1$ (and $p(\{i\in\mathcal{I}\mid a_i=s_0\})=0$ for all $s\in S,
s\not= s_0$). Thus we have $\bra a_i\ket={^*s_0}\in{S}$, as required. Suppose now, that $S$ is an
infinite set. We have to show that $^*S\setminus S\not=\varnothing$. Indeed, by axiom of choice,
there exists a sequence $(s_n)$ in $S$ such that $s_m\not= s_n$ whenever $m\not=
n$. Next, we define
$(a_i)\in S^\mathcal{I}$ by $a_i= s_n$, where $n=\max\{m\in\mathbb{N}\mid i\in
\mathcal{I}_{m-1}\setminus\mathcal{I}_m\}$ and we have let also $\mathcal{I}_0=\mathcal{I}$. Let
$s\in S$. We  have to show that the set
$\{i\in\mathcal{I}\mid a_i\not=s\}$ is of measure 1. Indeed, if $s$ is not in the range of $(s_n)$,
then $\{i\in\mathcal{I}\mid a_i\not=s\}=\mathcal{I}$ and is of measure 1. If $s$ is in the range
of $(s_n)$, then $s=s_k$ for exactly one $k\in\mathbb{N}$. We observe that
$\mathcal{I}_k\subseteq\{i\in\mathcal{I}\mid a_i\not=s\}$. Now the set $\{i\in\mathcal{I}\mid
a_i\not=s\}$ is of measure 1 because $\mathcal{I}_k$ is of measure one, by property~(iv)-(c) of $p$.
The proof is complete. Thus $\bra s_i\ket\in{^*S}\setminus S$ as required. 

\end{proof}

	In what follows $(A_i)\in\mathcal{P}(S)^\mathcal{I}$ means that $A_i\subseteq S$ for all
$i\in\mathcal{I}$.
\begin{definition}[Internal Sets]\label{D: Internal Sets} Let $\mathcal{A}\subseteq{^*S}$. We say
that $\mathcal{A}$ is an {\bf internal set} of $^*S$ if there exists a family
$(A_i)\in\mathcal{P}(S)^\mathcal{I}$ of subsets of $S$ such that 
\[
\mathcal{A}=\{\bra s_i\ket\in{^*S}\mid s_i\in A_i \text{\, a.e.\,} \}.
\]
We say that the family $(A_i)$ generates $\mathcal{A}$ and we write $\mathcal{A}=\bra A_i\ket$.
Let, in the particular, $A_i=A$ for all $i\in\mathcal{I}$ and some $A\subseteq S$. We say that the
internal set $^*A=\bra A_i\ket$ is the {\bf non-standard extension} of $A$. We denote by
$^*\mathcal{P}(S)$ the set of the internal subsets of
$^* S$. The sets in
$^*\mathcal{P}(S)\setminus\mathcal{P}(S)$ are call {\bf external}.
\end{definition}

		If $X\subseteq S$, then $^*X$ is internal and $^*X$ is generated by the constant family $X_i=X$
for all $i\in\mathcal{I}$. In particular $^*S$ is an internal set. Let $\bra
s_i\ket\in{^*S}\setminus S$ be the element defined in the proof of Theorem~\ref{T: Extension
Principle}. Then the singleton
$\{\bra s_i\ket\}$ is an internal set which is not of the form $^*X$ for some
$X\subseteq S$. This internal set
is generated by the singletons
$\{s_i\}$, i.e. $\{\bra s_i\ket\}=\bra\{s_i\} \ket$. More generally, every finite subset of $^*S$
is an internal set. We shall give more examples of infinite internal sets of $^*\mathbb{R}$ and 
$^*\mathbb{C}$ in the next section. If $A\subseteq S$, then $A$ is an external set of $^*S$.

	In the next theorem we use for the first time the property (v) of the probability measure $p$
(Definition~\ref{D: kappa-Good Two Valued Measures}). Recall that $\kappa=\card(\mathcal{I})$."

\begin{theorem}[Axiom 2. Saturation Principle in $^*\mathbb{C}$]\label{T: Saturation Principle in
*C}
$^*\mathbb{C}$ is {\bf $\kappa$-saturated} in the sense that every family
$\{\mathcal{A}_\gamma\}_{\gamma\in\Gamma}$ of internal sets of $^*\mathbb{C}$ with the
finite intersection property, and an index set $\Gamma$ with $\card(\Gamma)\leq \kappa$, has a
non-empty intersection.
\end{theorem}

\Proof We have (by assumption) that
$\bigcap_{\gamma\in F}\,\mathcal{A}_{\gamma}\not=\varnothing$  for every finite subset $F$
of $\Gamma$. We have to show that
$\bigcap_{\gamma\in\Gamma}\,\mathcal{A}_\gamma\not=\varnothing$. The fact that
$\mathcal{A}_\gamma$ is an internal set means that
$\mathcal{A}_\gamma=\bra\mathbb{A}_{\gamma,i}\ket$ for some
$\mathbb{A}_{\gamma,i}\subseteq\mathbb{C}$. Hence, for every finite subset $F$
of $\Gamma$ we have $\{i\in\mathcal{I}:
\bigcap_{\gamma\in F}\,\mathbb{A}_{\gamma, i}\not=\varnothing\}\in\mathcal{U}$. Next, we
define the function $R: \mathcal{P}_\omega(\Gamma)\to\mathcal{U}$, by 
\[
R(F)=\mathcal{I}_{\card(F)}\cap\{i\in\mathcal{I}:
\bigcap_{\gamma\in F}\,\mathbb{A}_{\gamma,i}\not=\varnothing\},
\]
for every finite subset $F$ of $\Gamma$. It is clear that $R$ is a
reversal (Remark~\ref{R: Reversals}). Since $p$ is a
$\kappa$-good measure, it follows that there exists a
strict reversal $S: \mathcal{P}_\omega(\Gamma)\to\mathcal{U}$ which minorizes $R$, i.e. 
\[
S(F) \subseteq \mathcal{I}_{\card(F)}\cap\{i\in\mathcal{I}:
\bigcap_{\gamma\in F}\,\mathbb{A}_{\gamma, i}\not=\varnothing\},
\]
for every finite subset $F$ of $\Gamma$. For every
$i\in\mathcal{I}$ we define
\[
\Gamma_i=\{\gamma\in\Gamma\mid i\in S(\{\gamma\})\}.
\]
Notice that if $\card(\Gamma_i)=m$ for some $m\in\mathbb{N}$ and some
$i\in\mathcal{I}$, then $i\in\mathcal{I}_m$. Indeed,
$\card(\Gamma_i)=m$ means that $\Gamma_i=\{\gamma_1, \gamma_2,\dots, \gamma_m\}$ for
some distinct $\gamma_1, \gamma_2,\dots, \gamma_m\in\Gamma$ such that $i\in\bigcap_{n=1}^m\,
S(\gamma_n)$. Using the fact that $S$ is a strict reversal, we have
$\bigcap_{n=1}^m\, S(\gamma_n)=S(\{\gamma_1,
\gamma_2,\dots,\gamma_m\})\subseteq R(\{\gamma_1,
\gamma_2,\dots,\gamma_m\})\subseteq\mathcal{I}_m$, hence, $i\in\mathcal{I}_m$ follows. On the
other hand, $\bigcap_{m=1}^\infty\mathcal{I}_m=\varnothing$ implies that $\Gamma_i$ is a
finite set for every $i\in\mathcal{I}$. As a result,
$\bigcap_{\gamma\in\Gamma_i}\; \mathbb{A}_{\gamma, i}\not=\varnothing$ for all
$i\in\mathcal{I}$. By
Axiom of Choice, there exists $(A_i)\in\mathbb{C}^{\mathcal{I}}$ such that
$A_i\in\bigcap_{\gamma\in\Gamma_i}\mathbb{A}_{\gamma, i}$ for all
$i\in\mathcal{I}$. We intend to show that $\bra
A_i\ket\in\bigcap_{\gamma\in\Gamma}\,\mathcal{A}_\gamma $. Indeed, for every
$\gamma\in\Gamma$ we have 
\[
S(\{\gamma\})=\{i\mid\gamma\in\Gamma_i\}\subseteq\{i\mid A_i\in
\mathbb{A}_{\gamma, i}\}.
\]
Since $S(\{\gamma\})\in\mathcal{U}$, it follows that $\{i\in\mathcal{I}\mid A_i\in
\mathbb{A}_{\gamma, i}\}\in\mathcal{U}$. Hence $\bra A_i\ket\in\bra \mathbb{A}_{\gamma,
i}\ket=\mathcal{A}_{\gamma}$, as required.

$\blacktriangle$

	In what follows we use the notation $\mathbb{N}_0=\{0,1,2,\dots\}$. 

\begin{theorem}[Sequential Saturation]\label{T: Sequential Saturation} $^*S$
is sequentially saturated in the sense that every sequence
$\{\mathcal{A}_n\}_{n\in\mathbb{N}_0}$ of internal sets of $^*S$ with the finite
intersection property has a non-empty intersection.
\end{theorem}

\noindent{\bf Proof 1  (An Indirect Proof):} An immediate consequence of Theorem~\ref{T: Saturation
Principle in *C}  in the case of countable index set $\Gamma$.

\noindent{\bf Proof  2 (A Direct Proof):} We have $\bigcap_{n=0}^m\,\mathcal{A}_n\not=\varnothing$
for all
$m\in\mathbb{N}_0$, by assumption. We have to show that
$\bigcap_{n=0}^\infty\,\mathcal{A}_n\not=\varnothing$. The fact that
$\mathcal{A}_n$ are internal sets means that
$\mathcal{A}_n=\bra\mathbb{A}_{n,i}\ket$ for some
$\mathbb{A}_{n,i}\subseteq\mathbb{C}$, where $n\in\mathbb{N}_0,\;
i\in\mathcal{I}$. We have $\left<\, \bigcap_{n=0}^m\,\mathbb{A}_{n,i}\right>=
\bigcap_{n=0}^m\,\left<\mathbb{A}_{n, i}\right>=
\bigcap_{n=0}^m\,\mathcal{A}_n\not=\varnothing$. Thus for every $m\in\mathbb{N}_0$ we have
\begin{equation}\label{E: Assumption}
\Phi_m=\{i\in\mathcal{I} \mid \cap_{n=0}^m\, \mathbb{A}_{n,i}\not=\varnothing\}\in\mathcal{U}.
\end{equation}
Without loss of generality we can assume that $\mathbb{A}_{0,i}\not=\varnothing$ for all
$i\in\mathcal{I}$ (indeed, if $\Phi_0\not=\mathcal{I}$, we
can choose another representative of $\mathcal{A}_0$ by
$\mathbb{A}^\prime_{0,i}=\mathbb{A}_{0,i}$ for $i\in\Phi_0$ and by
$\mathbb{A}^\prime_{0,i}=\mathbb{C}$ for
$i\in\mathcal{I}\setminus\Phi_0$). Next, we define the function
$\mu:\mathcal{I}\to\mathbb{N}_0\cup\{\infty\}$, by
\[
\mu(i)=\max\{m\in\mathbb{N}_0 \mid \cap_{n=0}^m\,
\mathbb{A}_{n,i}\not=\varnothing\}.
\]
Notice that $\mu$ is well-defined because the set 
\[
\{m\in\mathbb{N}_0 \mid \cap_{n=0}^m\,
\mathbb{A}_{n,i}\not=\varnothing\},
\]
is non-empty for all $i\in\mathcal{I}$ due to our assumption for 
$\mathbb{A}_{0,i}$. Thus we have $\bigcap_{n=0}^{\mu(i)}\,
\mathbb{A}_{n,i}\not=\varnothing$ for all $i\in\mathcal{I}$. Hence (by
Axiom of Choice) there exists
$(A_i)\in\mathbb{C}^{\mathcal{I}}$ such that 
$A_i\in\bigcap_{n=0}^{\mu(i)}\,\mathbb{A}_{n,i}$ for all
$i\in\mathcal{I}$. We intend to show that 
$\bra A_i\ket\in\bigcap_{n=0}^\infty\,\mathcal{A}_n$ or equivalently, to show
that for every
$m\in\mathbb{N}_0$ we have $\{i\in\mathcal{I}\mid A_i\in\mathbb{A}_{m, i}\}\in\mathcal{U}$.
We observe that 
\[
\Phi_m\subseteq\{i\in\mathcal{I} \mid
A_i\in\mathbb{A}_{m, i}\}.
\]
Indeed, $i\in\Phi_m$ implies $\bigcap_{n=0}^m\, \mathbb{A}_{n,i}\not=\varnothing$ which
implies $0\leq m\leq\mu(i)$ (by the definition of $\mu(i)$) leading to
$A_i\in\mathbb{A}_{m,i}$, by the choice of $(A_i)$. On the other hand, we have
$\Phi_m\in\mathcal{U}$,  by (\ref{E: Assumption}) implying $\{i\in\mathcal{I}\mid
A_i\in\mathbb{A}_{m, i}\}\in\mathcal{U}$, as required, by property (3) of $\mathcal{U}$.
$\blacktriangle$

\begin{definition}[Superstructure]\label{D: Superstructure} {\em Let S be an infinite set.  The
\textbf{superstructure}
$V(S)$ on $S$ is the union 
\[
V(S) = \bigcup_{n =0}^\infty V_n(S),
\] 
where the $V_n(S)$ are defined inductively by 
\begin{align}
&V_0(S)= S,\quad V_1(S) =S\cup\mathcal{P}(S),\notag\\
&V_{n+1}(S) = V_n(S) \cup\mathcal{P}(V_n(S)). 
\end{align}
The members of $V(S)$ are called \textbf{entities}. The members of $V(S)\setminus S$ are
called the \textbf{sets} of the superstructure V(S) and the members of $S$ are called
the \textbf{individuals} of the superstructure $V(S)$.
}\end{definition}

\begin{definition}[The Language $\mathcal{L}(V(S))$]\label{D: The Language L(V(S))}{\em  The
language
$\mathcal{L}(V(S))$ is the usual ``language of the analysis'' with the following restrictions: All
quantifiers are bounded by sets in the superstrucure $V(S)$, i.e. quantifiers appear in the
formulae of the language $\mathcal{L}(V(S))$ only in the form 
\[
(\forall x\in A)P(x) \text{\; or\; }  (\exists x\in A)P(x), 
\]
where  $P(x)$ is a predicate in one or more variables and $A\in V(S)\setminus S$. In particular,
formulae such as 
\begin{align}
&(\forall x)P(x),\notag \\
&(\exists x)P(x),\notag\\
&(\forall x\in s)P(x),\notag\\
&(\exists x\in s)P(x)\notag,
\end{align}
 where $s\in S$, do not belong the the language $\mathcal{L}(V(S))$. 
}\end{definition}

	In what follow $V(^*S)$ stands for the supersructure of $^*S$ and $\mathcal{L}(V(^*S))$ stands for
the language on $V(^*S)$ which are defined exactly as $V(S)$ and $\mathcal{L}(V(S))$ after
replacing $S$ by $^*S$. 

\begin{theorem}[Axiom 3. Transfer Principle]\label{T: Transfer Principle} Let $P(x_1, x_2,\dots
x_n)$ be a predicate in $\mathcal{L}(V(S))$ and
$A_1, A_2,\dots, A_n\in V(S)$. Then $P(A_1, A_2,\dots A_n)$ is true $\mathcal{L}(V(S))$ \ifff
$P({^*A_1}, {^*A_2},\dots {^*A_n})$ is true in $\mathcal{L}(V(^*S))$.
\end{theorem}

	For examples of application of the Transfer Principle we refer to the first proofs of Lemma~\ref{L:
No Zero Divisors} and Lemma~\ref{L: Trichotomy} later in this text.

\section{A. Robinson's Non-Standard Numbers}

	In this section we apply the non-standard construction in the particular case $S=\mathbb{C}$,
where $\mathbb{C}$ is the field of the complex numbers. 
\begin{definition}[Non-Standard Numbers]\label{D: Non-Standard} {\em 
\begin{enumerate} 
	\item We define the {\bf complex non-standard numbers} as the factor ring
$^*\mathbb{C}=\mathbb{C}^\mathcal{I}/\sim$,
where $(a_i)\sim(b_i)$ if $a_i=b_i$ a.e., i.e. if
\[
p\left(\{i\in\mathcal{I}\mid a_i=b_i \}\right)=1
\]
(or, equivalently, if $\{i\in\mathcal{I}\mid a_i=b_i\}\in\mathcal{U}$, where
$\mathcal{U}=\{A\in\mathcal{P}\mid p(A)=1\}$.) We denote by $\bra a_i\ket\in{^*\mathbb{C}}$ 
the equivalence class determined by $(a_i)$. The algebraic operations and the absolute value in
$^*\mathbb{C}$ is inherited from $\mathbb{C}$. For example, $|\bra x_i\ket|=\bra |x_i|\ket$.
\item The set of
{\bf real non-standard numbers}
$^*\mathbb{R}$ is (by definition) the non-standard extension of $\mathbb{R}$, i.e.
\[
^*\mathbb{R}=\{\bra x_i\ket\in{^*\mathbb{C}}\mid x_i\in\mathbb{R}\text{\; a.e.\;} \}.
\]
The order relation if $^*\mathbb{R}$ is defined by $\bra a_i\ket <\bra b_i\ket$ if $a_i< b_i$ in
$\mathbb{R}$ a.e., i.e. if 
\[
p\left(\{i\in\mathcal{I}\mid a_i<b_i \}\right)=1.
\]
\item The mapping $r\to{^*r}$ defines an embeddings $\mathbb{C}\subset{^*\mathbb{C}}$
and
$\mathbb{R}\subset{^*\mathbb{R}}$ by the constant nets, i.e. $^*r=\bra
a_i\ket$, where $a_i=r$ for all $i\in\mathcal{I}$.
\end{enumerate}
}\end{definition}

\begin{theorem}[Basic Properties]\label{T: Basic Properties}
\begin{enumerate}
\item $^*\mathbb{C}$ is an algebraically closed non-archimedean field.
\item  $^*\mathbb{R}$ is a real closed  (totally ordered) non-archimedean
field.
\end{enumerate}
\end{theorem}

\begin{proof} We shall separate the proof of the above theorem in several small lemmas and prove some of
them. We shall present also two proofs to each of the lemmas; one of them based on the Saturation
Principle (Theorem~\ref{T: Transfer Principle}) and the other on the properties of the measure $p$.
The content of the next lemma is a small (but typical) part of the statement that both
$^*\mathbb{C}$ and $^*\mathbb{R}$ are fields
\end{proof}

\begin{lemma}[No Zero Divisors]\label{L: No Zero Divisors} $^*\mathbb{C}$  is free of zero divisors.
\end{lemma}
\noindent{\bf Proof 1:} The statement
\[
(\forall x, y\in {\mathbb{C}})(xy=0 \Rightarrow x=0 \vee y=0), 
\]
is true because $\mathbb{C}$ is free of zero divisors. Thus
\[
(\forall x, y\in{^*\mathbb{C}})(xy=0 \Rightarrow x=0 \vee y=0), 
\]
is true (as required) by Transfer Principle (Theorem~\ref{T: Transfer Principle}). 

$\blacktriangle$

\noindent{\bf Proof 2:} Suppose $\bra a_i\ket\bra b_i\ket=0$ in $^*\mathbb{C}$ for some $\bra
a_i\ket,\, \bra b_i\ket\in{^*\mathbb{C}}$. Thus $\bra a_i b_i\ket=0$ implying
$p(\{i\in\mathcal{I}\mid a_i b_i=0\})=1$. On the other hand,
\[
\{i\in\mathcal{I}\mid a_i b_i=0\}=\{i\in\mathcal{I}\mid a_i =0\}\cup\{i\in\mathcal{I}\mid
b_i=0\},
\]
because $\mathbb{C}$ is free of zero divisors. It follows that 
\[
p(\{i\in\mathcal{I}\mid a_i =0\})+p(\{i\in\mathcal{I}\mid b_i =0\})\geq 1,
\]
by the additivity of $p$. Since the range of $p$ is $\{0, 1\}$, it follows that ether
$p(\{i\in\mathcal{I}\mid a_i =0\})=1$ or $p(\{i\in\mathcal{I}\mid b_i =0\})=1$, i.e. either $\bra
a_i\ket=0$ or $\bra
b_i\ket=0$, as required. $\blacktriangle$
\begin{lemma}[Trichotomy]\label{L: Trichotomy} Let $a, b\in{^*\mathbb{R}}$. Then ether $a<b$ or
$a=b$ or $a>b$.
\end{lemma}
\noindent{\bf Proof 1:} The statement
\[
(\forall x, y\in {\mathbb{R}})(x\not=y \Rightarrow x<y \vee x>y), 
\]
is true because $\mathbb{R}$ is a totally ordered set. Thus
\[
(\forall x, y\in{^*\mathbb{R}})(x\not=y \Rightarrow x<y \vee x>y), 
\]
is true (as required) by Transfer Principle (Theorem~\ref{T: Transfer Principle}). 

$\blacktriangle$

\noindent{\bf Proof 2:} Suppose that $\bra a_i\ket, \bra b_i\ket\in{^*\mathbb{R}}$. We observe that
the sets
\[
A=\{i\in\mathcal{I}\mid a_i<b_i\},\quad B=\{i\in\mathcal{I}\mid a_i=b_i\},\quad 
C=\{i\in\mathcal{I}\mid a_i>b_i\},
\]
are mutually disjoint and $A\cup B\cup C=\mathcal{I}$ because $\mathbb{R}$ is a totally ordered set.
Thus $p(A)+p(B)+p(C)=1$ by the additivity of the measure $p$. It follows that exactly one of the
following is true: $p(A)=1$ or $p(B)=1$ or $p(C)=1$, since the range of
$p$ is $\{0, 1\}$. Thus exactly one of the following is true: $\bra a_i\ket<\bra b_i\ket$,
$\bra a_i\ket=\bra b_i\ket$, and $\bra a_i\ket>\bra b_i\ket$.  

$\blacktriangle$

	The rest of the proof of Theorem~\ref{T: Basic Properties} can be done in a
similar manner and we leave it to to the reader.
$\blacktriangle$

\section{Infinitesimals, Finite and Infinitely Large Numbers}\label{S: Infinitesimals, Finite and
Infinitely Large Numbers}
\begin{definition}\label{D: I(*C), F(*C), L(*C)}{\em 
\begin{enumerate}
\item We define the sets of \textbf{infinitesimal},
\textbf{finite}, and
\textbf{infinitely large} numbers as follows:
\[
    \mathcal{I}(^{*}\mathbb{C})=\{x\in\,
^{*}\mathbb{C} : \vert
    x\vert<1/n \text{ for all } n\in\mathbb{N}\},
    \]
\[
    \mathcal{F}(^{*}\mathbb{C})=\{x\in\,
^{*}\mathbb{C} : \vert
    x\vert<n \text{ for some } n\in\mathbb{N}\},
    \]
\[
    \mathcal{L}(^{*}\mathbb{C})=\{x\in\,
^{*}\mathbb{C} : \vert
    x\vert>n \text{ for all } n\in\mathbb{N}\},
    \]
\item Let $x, y\in{^*\mathbb{C}}$. We say $x$ and $y$ are infinitely close, in symbol
$x\approx y$, if $x-y\in\mathcal{I}(^*\mathbb{C})$. The relation $\approx$ is called {\bf
infinitesimal relation} on $^*\mathbb{C}$.

\item Let $x\in{^*\mathbb{C}}$ and $r\in{\mathbb{C}}$. We write $x\leadsto y$ if 
$x-r\in\mathcal{I}(^*\mathbb{C})$. We shall often refer to $\leadsto$ an {\bf asymptotic expansion}
of $x$.
\end{enumerate}
}\end{definition}
\begin{proposition}[Basic Properties]\label{P: Basic Properties} 
\begin{align}
&^*\mathbb{C}=\mathcal{F}(^*\mathbb{C})\cup\mathcal{L}(^*\mathbb{C}),\\
&\mathcal{F}(^*\mathbb{C})\cap\mathcal{L}(^*\mathbb{C}) =\varnothing,\\
&\mathcal{I}(^*\mathbb{C})\subset \mathcal{F}(^*\mathbb{C}),\\
&\mathcal{I}(^*\mathbb{C})\cap\mathbb{C}=\{0\},
\end{align}
and similarly for $^*\mathbb{R}$.
\end{proposition}
\Proof These results follow directly from the definitions of infinitesimal, finite and infinitely
large numbers and the fact that $^*\mathbb{R}$ is a totally ordered field. $\blacktriangle$
\begin{example}[Infinitesimals]\label{Ex: Infinitesimals} Let $\nu=\bra a_i\ket$, where
$(a_i)\in\mathbb{C}^\mathcal{I}$, $a_i=n$, $n=\max\{m\in\mathbb{N}\mid
i\in\mathcal{I}_{m-1}\setminus\mathcal{I}_m\}$. The non-standard number $\nu$ is an infinitely
large natural number in the sense that
$\nu\in{^*\mathbb{N}}$ and $(\forall \varepsilon\in\mathbb{R}_+)(\varepsilon<\nu)$. Indeed, we
choose $n\in\mathbb{N}$ such that $\varepsilon\leq n$ and observe that
$\mathcal{I}_n\subset\{i\in\mathcal{I}\mid a_i> n\geq\varepsilon\}$. Thus $p(\{i\in\mathcal{I}\mid
a_i>\varepsilon\})=1$ since $p(\mathcal{I}_n)=1$. Among other things this example show that 
$^*\mathbb{R}$ and $^*\mathbb{C}$ are proper extensions of $\mathbb{R}$ and $\mathbb{C}$, 
respectively.The numbers $\nu^n,\, \sqrt[n]{\nu},\, \ln{\nu},\, e^\nu$ are infinitely
large numbers in $^*\mathbb{R}$. In contrast, the numbers $1/\nu^n,\, 1/\sqrt[n]{\nu},\,
1/\ln{\nu},\, e^{-\nu}$ are non-zero infinitesimals in  $^*\mathbb{R}$. If
$r\in\mathcal{\mathbb{R}}$, then $r+ 1/\nu^n$ is a finite (but not standard) number in
$^*\mathbb{R}$. Also $e^{i\nu}$ is a finite number in $^*\mathbb{C}$ and 
$e^{i\nu}\nu^2+i\ln{\nu}+5+3i$ is an infinitely large number in
$^*\mathbb{C}$.
\end{example}

	Our next goal is to define and study a ring homomorphism\, $\st$\, from the ring of finite numbers
$\mathcal{F}(^*\mathbb{C})$ to
$\mathbb{C}$, called {\em standard part mapping}.  The standard part 
mapping is, in a sense, a counterpart of the concept of {\em limit} in the usual (standard) analysis. 
In contrast to limit, however,  the standard part mapping is applied to 
non-standard numbers rather than to sequences of standard numbers or functions.

\begin{definition}[Standard Part Mapping]\label{D: Standard Part Mapping}{\em 
\begin{enumerate}
\item The {\bf
standard part mapping}\;  $\st: \mathcal{F}(^*\mathbb{R})\to\mathbb{R}$ is defined by the formula:
\begin{equation}\label{E: St}
{\rm st}(x)=\sup\{r\in\mathbb{R}\mid r<x\}.
\end{equation} 
If $x\in\mathcal{F}(^*\mathbb{R})$, then ${\rm st}(x)$ is called the {\bf standard part} of $x$. 

The standard part mappings defined in (ii) and (iii) below are extensions of the
standard part mapping just defined; we shall keep the same notation,\; \st,\; for all.

	\item The {\bf standard part mapping}\; ${\rm st}: \mathcal{F}(^*\mathbb{C})\to
\mathbb{C}$ is defined by the formula ${\rm st}(x+y\, i)={\rm st}(x)+ {\rm st}(y)\, i$.

	\item The mapping\; 
${\rm st}:{^*\mathbb{R}}\to \mathbb{R}\cup\{\pm\infty\}$ is defined by (i) and by ${\rm
st}(x)=\pm\infty$ for $x\in\mathcal{L}(^*\mathbb{R}_\pm)$, respectively.
\end{enumerate}
}\end{definition}

\begin{theorem}[Standard Part Mapping on Finite Numbers]\label{T: Standard Part Mapping on Finite
Numbers} 
\begin{enumerate}

	\item{\bf (i)} Every finite non-standard number $x\in\mathcal{F}(^*\mathbb{C})$ has a
unique asymptotic expansion 
\begin{equation}\label{E: Asy}
x={\rm st}(x)+dx.
\end{equation}
where $dx\in\mathcal{I}(^*\mathbb{C})$. Consequently, if $x\in{^*\mathbb{C}}$, then
$x\in\mathcal{F}(^*\mathbb{C})$ \ifff  $x=c+dx$ for some
$c\in\mathbb{C}$ and some $dx\in\mathcal{I}(^*\mathbb{C})$.

	\item The standard part mapping
is a ring homomorphism from $\mathcal{F}(^*\mathbb{C})$ onto $\mathbb{C}$, i.e. for
every $x, y\in\mathcal{F}(^*\mathbb{C})$ we have:
\begin{align}\label{E: StHomomorphism}
&{\rm st}(x\pm y)={\rm st}(x)\pm {\rm st}(y),\\\notag
&{\rm st}(x\, y)={\rm st}(x)\, {\rm st}(y),\\\notag
&{\rm st}(x/y)={\rm st}(x)/{\rm st}(y), \text{\; whenever\; } {\rm st}(y)\not= 0.
\end{align}

	\item $\mathbb{C}$ consists exactly of the {\bf fixed points} of\, $\st$ in
$^*\mathbb{C}$, in symbol,
\begin{equation}\label{E: Fixed Points}
\mathbb{C}=\{x\in{^*\mathbb{C}}\mid \st(x)=x\}.
\end{equation}
Consequently, $\st\circ\st=\st$, where $\circ$ denotes ``composition''.

	\item \, $x\in\mathcal{I}(^*\mathbb{R})$ \ifff $\st(x)=0$.
 
	\item The standard part mapping\, \st\,  
is an order preserving ring homomorphism from $\mathcal{F}(^*\mathbb{R})$ onto $\mathbb{R}$, where
``order preserving'' means that if
$x, y\in\mathcal{F}(^*\mathbb{R})$, then $x<y$ implies ${\rm st}(x)\leq {\rm st}(y)$ (hence, $x\leq y$
implies ${\rm st}(x)\leq {\rm st}(y)$).
\end{enumerate}
\end{theorem}
\Proof (i) Suppose, first, that $x\in\mathcal{F}(^*\mathbb{R})$. We have to show that $x-{\rm
st}(x)$ is infinitesimal. Suppose (on the contrary) that
$1/n<|x-{\rm st}(x)|$ for some $n$. In the case $x>{\rm st}(x)$, it follows $1/n<x-{\rm st}(x)$,
contradicting (\ref{E: St}). In the case $x<{\rm st}(x)$, it follows $1/n<{\rm st}(x)-x$, again
contradicting (\ref{E: St}). To show the uniqueness of (\ref{E: Asy}), suppose that $r+dx=s+dy$ for
some $r, s\in\mathbb{R}$ and some $dx, dy\in\mathcal{I}(^*\mathbb{R})$. It follows that $r-s$ is
infinitesimal, hence, $r=s$, since the zero is the only infinitesimal in $\mathbb{R}$. The result
extends to $\mathcal{F}(^*\mathbb{C})$ directly by the formula in part~(ii) of Definition~\ref{D:
Standard Part Mapping}.

	(ii) follows immediately from (i).

	(iii) follows immediately from (i) by letting $dx=0$.

	(iv) follows directly from the definition of $\st$.

	(v) If $x\approx y$, then it follows $\st(x)=\st(y)$ (regardless whether $x<y$, $x=y$ or
$x>y$). Suppose $x<y$ and
$x\not\approx y$. It follows $\st(x)+dx<\st(y)+dy$. We have to show that $\st(x)\leq \st(y)$.
Suppose (on the contrary) that
$\st(x)>\st(y)$. It follows $0<\st(x)-\st(y)<dy-dx$ implying $\st(x)-\st(y)\approx 0$, hence,
$\st(x)=\st(y)$, a contradiction. 
$\blacktriangle$
\begin{corollary}[An Isomorphism]\label{C: An Isomorphism} {\bf (i)} $\mathcal{F}(^*\mathbb{R})/\mathcal{I}(^*\mathbb{R})$ is ordered field
isomorphic to
$\mathbb{R}$ under the mapping $q(x)\to\st(x)$, where $q:
\mathcal{F}(^*\mathbb{R})\to\mathcal{F}(^*\mathbb{R})/\mathcal{I}(^*\mathbb{R})$ is the quotient
mapping. 

	{\bf (ii)} $\mathcal{F}(^*\mathbb{C})/\mathcal{I}(^*\mathbb{C})$ is field isomorphic
to $\mathbb{C}$ under the mapping $Q(x)\to\st(x)$, where
$Q: \mathcal{F}(^*\mathbb{C})\to\mathcal{F}(^*\mathbb{C})/\mathcal{I}(^*\mathbb{C})$ is the quotient
mapping. 

	{\bf (iii)} The isomorphism described in (ii) is an extension of the isomorphism described in (i).
\end{corollary}

	We leave the proof to the reader.

\begin{example} Let $c\in\mathbb{C}$ and let $dx\in\mathcal{I}(^*\mathbb{C})$ be 
a non-zero
infinitesimal. Then we have:
\begin{align}\notag
&\st(c+dx^n)=c,\\\notag
&\st(dx/|dx|)  = \pm 1,\\\notag
&\st\left(\frac{cdx + 7dx^2 + dx^3}{dx}\right) =  \st (c + 7dx + dx^2) =  c,\\\notag
&\st\left(\frac{-3 + 4 dx}{dx}\right) = \st(1/dx)\times\st(-3 + 4 dx) =
 (\pm\infty)\times (-3) =\mp\infty,
\end{align}
where the choice of the sign $\pm$ depends on whether $dx$ is positive or negative, 
respectively.
\end{example}
\begin{definition}[Standard Part of a Set]\label{D: St of a Set} If $\mathcal{A}\subseteq{^*\mathbb{C}}$, we define the {\bf standard part} of $\mathcal{A}$ 
by
\begin{equation}\label{E: St of a Set}
\st[\mathcal{A}]=\{\st(x)\mid x\in\mathcal{A}\cap\mathcal{F}(^*\mathbb{C})\}.
\end{equation}
\end{definition}
\begin{lemma} If $\mathcal{A}\subseteq{^*\mathbb{C}}$, then 
$\mathcal{A}\cap\mathbb{C}\subseteq\st[\mathcal{A}]$. (A proper inclusion might occur; 
see the example below.). In particular, we have $\st[^*\mathbb{R}]=\mathbb{R}$ and 
$\st[^*\mathbb{C}]=\mathbb{C}$.
\end{lemma}
\Proof The inclusion 
$\mathcal{A}\cap\mathbb{C}\subseteq\st[\mathcal{A}]$ follows directly from part~(iii) of 
Theorem~\ref{T: Standard Part Mapping on Finite Numbers}.
\begin{example} Consider the set 
$\mathcal{A}=\{x\in{^*\mathbb{R}}\mid 0<x<1\}$. We have 
$\mathcal{A}\cap\mathbb{C}=\{x\in{\mathbb{R}}\mid 0<x<1\}$. On the other hand, 
$\st[\mathcal{A}]=\{x\in{\mathbb{R}}\mid 0\leq x\leq1\}$. Indeed, if $\epsilon$ is a 
positive infinitesimal in $^*\mathbb{R}$, then 
$\epsilon,\; 1-\epsilon\in\mathcal{A}$ and $\st(\epsilon)=0$, and $\st(1-\epsilon)=1$.
\end{example}

\section{NSA and the Usual Topology on $\mathbb{R}^d$}\label{S: NSA and the 
Usual Topology on Rd}

	In what follows we let
${^*\mathbb{R}^d}={^*\mathbb{R}}\times{^*\mathbb{R}}\times\dots\times{^*\mathbb{R}}$ ($d$ times).
If $x\in{^*\mathbb{R}^d}$, then $x\approx 0$ means that $||x||$ is infinitesimal. 

\begin{definition}[Monads]\label{D:Monads} If $\mathbb{X}\subseteq{\mathbb{R}^d}$, then
\[
\mu(\mathbb{X})=\{r+dx \mid r\in\mathbb{X},\; dx\in{^*\mathbb{R}^d}, ||dx||\approx 0\}.
\]
is called the {\bf monad} of $\mathbb{X}$ in $^*\mathbb{R}^d$. If $r\in\mathbb{R}^d$, we shall
write simply $\mu(r)$ instead of the more precise $\mu(\{r\})$, i.e. 
\[
\mu(r)=\{r+dx\mid dx\in\mathcal{I}(\mathbb{R}^d)\}.
\]
\end{definition}

	We observe that
$\mu(\mathbb{X})=\bigcup_{r\in\mathbb{X}}\mu(r)$.  

	In what follows $\mathcal{T}$ stands for the usual topology on $\mathbb{R}^d$.

\begin{theorem}[Boolean Properties]\label{T:Boolean Properties} The mapping
$\mu:\mathcal{T}\to\mathcal{P}(^*\mathbb{R}^d)$ is a Boolean homomorphism. Also $\mu$ preserves the
arbitrary unions in the sense that
$\mu\left(\bigcup_{\lambda\in\Lambda}\Omega_\lambda\right)=
\bigcup_{\lambda\in\Lambda}\mu(\Omega_\lambda)$ for any set $\Lambda$ and any family of open sets 
$\{\Omega_\lambda\}_{\lambda\in\Lambda}$.
\end{theorem}

\begin{theorem}[The Usual Topology on $\mathbb{R}^d$]\label{T: The Usual Topology on Rd} Let
$\mathbb{X}\subseteq\mathbb{R}^d$. Then:
\begin{enumerate}

\item A set $\mathbb{X}$ is {\bf open} in $\mathbb{R}^d$ \ifff
$\mu(\mathbb{X})\subseteq{^*\mathbb{X}}$. 

\item $\mathbb{X}$ is {\bf compact} in $\mathbb{R}^d$ \ifff
$^*\mathbb{X}\subseteq\mu(\mathbb{X})$. 
\end{enumerate}
\end{theorem} 
\newpage
\section{Non-Standard Smooth Functions}\label{S: Non-Standard Smooth Functions}

\begin{definition}[Non-Standard Smooth Functions]\label{D: Non-Standard Smooth Functions} {\em Let
$\Omega$ is an open set of
$\mathbb{R}^d$. Then: 
\begin{enumerate}
\item The ring (algebra) of the {\bf non-standard smooth functions} is
defined the factor ring
\[
^*\mathcal{E}(\Omega)=\mathcal{E}(\Omega)^\mathcal{I}/\sim,
\]
where
$(f_i)\sim(g_i)$ if $f_i=g_i$ in $\mathcal{E}(\Omega)$ {\bf for almost all} $i$ in the sense that 
\[
p\left(\{i\mid f_i=g_i\}\right)=1.
\]

	We denote by $\bra f_i\ket\in{^*\mathcal{E}}(\Omega)$ the equivalence class determined by $(f_i)$.

\item The algebraic operations and
partial differentiation in $^*\mathcal{E}(\Omega)$  is inherited from  $^*\mathcal{E}(\Omega)$. For
example, $\partial^\alpha\bra f_i\ket=\bra\partial^\alpha f_i\ket$.

\item The mapping $f\to{^*f}$ defines an embedding
$\mathcal{E}(\Omega)\hookrightarrow{^*\mathcal{E}(\Omega)}$ by the constant families, i.e.
$f_i=f$ for all $i\in\mathcal{I}$. We say that $^*f$ is the {\bf non-standard extension } of $f$.
\item{\bf (iv)} Every 
$\bra f_i\ket\in{^*\mathcal{E}(\Omega)}$ is a {\bf pointwise mapping} of the form 
$\bra f_i\ket:{^*\Omega}\to{^*\mathbb{C}}$,
where $\bra f_i\ket(\bra x_i\ket)=\bra f_i(x_i)\ket$ and 
\[
^*\Omega=\{\bra x_i\ket\in{^*\mathbb{R}^d}\mid x_i\in\Omega \text{\; a.e.\;}\},
\]
is the {\bf non-standard extension} of $\Omega$. 
\item{\bf (v)} Let $X\subseteq{\mathcal{E}}$. The non-standard extension $^*X$ of $X$ is defined
by 
\[
^*X=\{\bra f_i\ket\in{^*\mathcal{E}}(\Omega)\mid f_i\in X \text{\ a.e.\,}\}.
\]
In particular,
\[
^*\mathcal{D}(\Omega)=\{\bra f_i\ket\in{^*\mathcal{E}}(\Omega)\mid f_i\in \mathcal{D}(\Omega)
\text{\ a.e.\,}\}.
\]
\end{enumerate}
}\end{definition}

\begin{proposition} $^*\mathcal{E}(\Omega)$ is a differential algebra over the field $^*\mathbb{C}$.
\end{proposition}
\begin{definition}[Sup and Support]\label{D: Sup and Support}{\em Let $\bra
f_i\ket\in{^*\mathcal{E}}(\Omega)$ and let $K\subset\subset\Omega$. Then
\begin{enumerate}
\item $\sup_{x\in{^*K}}|\bra f_i\ket(x)|=\bra \sup_{x\in{K}}|f_i(x)|\ket$.
\item $\supp\bra f_i\ket=\bra\supp(f_i)\ket$.
\end{enumerate}
We shall refer to these as {\bf internal sup} and {\bf internal support} of $\bra f_i\ket$,
respectively.
}\end{definition}
\begin{proposition} Let $f\in{^*\mathcal{E}}(\Omega)$. Then:
\begin{enumerate} 
\item $(\forall K\subset\subset\Omega)(\sup_{x\in{^*K}}f(x)\in{^*\mathbb{R}})$.
\item $\supp(f)$ is a closed set of $^*\mathbb{R}$ in the interval
topology of $^*\mathbb{R}$.
\end{enumerate}
\end{proposition}
\begin{lemma}[Characterizations] Let $f\in{^*\mathcal{E}(\Omega)}$ and $\supp(f)$ denote
the (internal) support of $f$ in $^*\Omega$. Then {\bf the following are equivalent:}
\begin{description}
\item{\bf (i)} $\supp(f)\subset\mu(\Omega)$.
\item{\bf (ii)} $\exists K\subset\subset\Omega$ such that $\supp(f)\subseteq {^*K}$.
\item{\bf (iii)} There exists an open relatively
compact subset $\mathcal{O}$ of $\Omega$ such that $f\in{^*\mathcal{D}}(\mathcal{O})$ (The latter implies $f(x)=0$
for all $x\in{^*(\Omega}\setminus\mathcal{O})$.)
\end{description}
\end{lemma}


\begin{definition}[Compact Support] {\em Let $\mathcal{X}\subseteq{^*\mathcal{E}}(\Omega)$. We
denote
\[
\mathcal{X}_c=\{f\in\mathcal{X}\mid \supp(f)\subset \mu(\Omega)\}.
\]
In particular, we have:
\begin{align}
&^*\mathcal{D}_c(\Omega)=\{f\in{^*\mathcal{D}(\Omega)}\mid \supp(f)\subset\mu(\Omega)\},\\
&\mathcal{X}_c={^*\mathcal{D}_c}(\Omega)\cap\mathcal{X},\\
&{^*\mathcal{D}_c}(\Omega)={^*\mathcal{E}_c}(\Omega)=\{f\in{^*\mathcal{E}(\Omega)}\mid
\supp(f)\subset\mu(\Omega)\}.
\end{align}
}\end{definition}

\begin{lemma}[Characterizations] Let $f\in{^*\mathcal{E}(\Omega)}$. Then the following are
equivalent:
\begin{enumerate}
\item $(\forall x\in\mu(\Omega))\left[f(x)\in\mathcal{M}_\rho(^*\mathbb{C})\right]$.
\item $(\forall K\subset\subset\Omega)(\exists n\in\mathbb{N})
(\sup_{x\in{^*K}} |f(x)|\leq\rho^{-n})$.
\item{\bf (iii)} $(\forall K\subset\subset\Omega)(\forall n\in{^*\mathbb{N}}\setminus\mathbb{N})
(\sup_{x\in{^*K}} |f(x)|\leq\rho^{-n})$.
\end{enumerate}
\end{lemma}
\begin{lemma}[Characterizations] Let $f\in{^*\mathcal{E}(\Omega)}$. Then the following are
equivalent:
\begin{enumerate}
\item $(\forall x\in\mu(\Omega))\left[f(x)\in\mathcal{N}_\rho(^*\mathbb{C})\right]$.
\item $(\forall K\subset\subset\Omega)(\forall n\in\mathbb{N})
(\sup_{x\in{^*K}} |f(x)|\leq\rho^{n})$.
\item $(\forall K\subset\subset\Omega)(\exists n\in{^*\mathbb{N}}\setminus\mathbb{N})
(\sup_{x\in{^*K}} |f(x)|\leq\rho^{n})$.
\end{enumerate}
\end{lemma}

\section{Local Properties of $^*\mathcal{E}(\Omega)$}\label{S: Local Properties of *E(Omega}

	In what follows $\mathcal{T}_d$ stands for the usual topology on $\mathbb{R}^d$ and we denote by
$(\mathbb{R}^d, {\mathcal{T}_d})$ the corresponding topological space. Also we denote by
$^*\mathcal{T}_d$ the order topology of $^*\mathbb{R}^d$ (more precisely, 
$^*\mathcal{T}_d$ stands for the product topology on $^*\mathbb{R}^d$ generated by the order
topology on $^*\mathbb{R}$). We denote by $(^*\mathbb{R}^d, {^*\mathcal{T}_d})$ the
corresponding topological space. 

The purpose of this section is to show that the collection of the non-standard spaces
$\{^*\mathcal{E}(\Omega)\}_{\Omega\in{^*\mathcal{T}_d}}$ (Section~\ref{S: Non-Standard Smooth
Functions}) is a {\em sheaf} on $(^*\mathbb{R}^d, {^*\mathcal{T}_d})$, but in contrast,
$\{^*\mathcal{E}(\Omega)\}_{\Omega\in{^*\mathcal{T}_d}}$ is only a {\em
presheaf} on $(\mathbb{R}^d, {\mathcal{T}_d})$.  For the relevent terminology we refer to A.
Kaneko~\cite{aKan88}. 

\begin{theorem}[Non-Standard Sheaf]\label{T: Non-Standard Sheaf} The collection
$\{^*\mathcal{E}(\Omega)\}_{\Omega\in{^*\mathcal{T}_d}}$ is a
\textbf{sheaf of differential rings on} $(^*\mathbb{R}^d, {^*\mathcal{T}_d})$ under the usual
pointwise restriction in $^*\mathcal{E}(\Omega)$.
\end{theorem}
\Proof From the (standard) functional analysis we know that the collection
$\{\mathcal{E}(\Omega)\}_{\Omega\in{\mathcal{T}_d}}$ is a sheaf of differential rings on
$\mathbb{R}^d$ in the sense that $f\in{\mathcal{E}}(\Omega)$ and
$\mathcal{O}\subseteq\Omega$ implies $f|\mathcal{O}\in{\mathcal{E}}(\mathcal{O})$ for every
$\Omega, \mathcal{O}\in{\mathcal{T}_d}$. Thus  $f\in{^*\mathcal{E}}(\Omega)$ implies
$f|\mathcal{O}\in{^*\mathcal{E}}(\mathcal{O})$ for every
$\Omega\in{\mathcal{T}_d}$ and $\mathcal{O}\in{^*\mathcal{T}_d}$ such that
$\mathcal{O}\subseteq{^*\Omega}$ by Transfer
Principle (Theorem~\ref{T: Transfer Principle}). $\blacktriangle$
\begin{corollary}[Non-Standard Support]\label{C: Non-Standard Support} Let
$f\in{^*\mathcal{E}}(\Omega)$ and $\supp(f)$ be the support of $f$ (Definition~\ref{D: Sup and
Support}). Then $\supp(f)$ is a closed set of $^*\Omega$ in the topology $^*\mathcal{T}_d$ on
$^*\mathbb{R}^d$.
\end{corollary}
\Proof The result follows (also) by Transfer Principle (or directly from the above theorem).

$\blacktriangle$

	Let $\mathcal{O},\Omega\in\mathcal{T}_d$ be two (standard) open sets such that
$\mathcal{O}\subseteq\Omega$ and $f\in{^*\mathcal{E}}(\Omega)$.  We define the restriction
$f\upharpoonright\mathcal{O}=f\,|{^*\mathcal{O}}$.

\begin{theorem}[Standard Presheaf]\label{T: Standard Presheaf} The collection
$\{^*\mathcal{E}(\Omega)\}_{\Omega\in{\mathcal{T}_d}}$ is a
\textbf{presheaf of differential rings on} $(\mathbb{R}^d, {\mathcal{T}_d})$ under the
restriction
$\rest$ in the sense that:
\begin{enumerate}
\item \quad
$(\forall\Omega\in\mathcal{T}_d)(\forall f\in{^*\mathcal{E}}(\Omega))(f
\rest\Omega=f)$.
\item\;
$(\forall\Omega_1, \Omega_2, \Omega\in\mathcal{T}_d)(\forall f\in
{^*\mathcal{E}}(\Omega))(\Omega_1\subseteq\Omega_2\subseteq\Omega$
implies $(f\rest\Omega_2)\rest\Omega_1= f\rest\Omega_1$.
\item $(\forall\Omega, \mathcal{O}\in\mathcal{T}_d)(\forall
f, g\in{^*\mathcal{E}}(\Omega))(\mathcal{O}\subseteq\Omega \Rightarrow (f+g)
\rest\mathcal{O}=f\rest\mathcal{O}+g\rest\mathcal{O})$.

\item $(\forall\Omega, \mathcal{O}\in\mathcal{T}_d)(\forall
f, g\in{^*\mathcal{E}}(\Omega))\left(\mathcal{O}\subseteq\Omega \Rightarrow (fg)
\rest\mathcal{O}=(f\rest\mathcal{O})(g\rest\mathcal{O})\right)$.

\item $(\forall\Omega, \mathcal{O}\in\mathcal{T}_d)(\forall
f\in{^*\mathcal{E}}(\Omega))(\forall\alpha\in\mathbb{N}_0^d)\left(\mathcal{O}\subseteq\Omega
\Rightarrow (\partial^\alpha f)
\rest\mathcal{O}=(\partial^\alpha (f\rest\mathcal{O})\right)$.

\end{enumerate}
\end{theorem}
\begin{proof} (1) $f\rest\Omega= f\, |\, {^*\Omega}=f$ (as required) since $^*\Omega$ is the
domain of $f$.

	(2) $(f\rest\Omega_2)\rest\Omega_1=(f\, |\, {^*\Omega_2})\, |\, {^*\Omega_1}=f\, |\,
{^*\Omega_1}=f\rest\Omega_1$ (as required)  since
$^*\Omega_1\subseteq{^*\Omega_2}\subseteq{^*\Omega}$. The rest of the properties are
proved similarly and we leave them to the reader.
\end{proof}
\begin{remark}[A Counter Example] The next example shows that the collection
$\{^*\mathcal{E}(\Omega)\}_{\Omega\in{\mathcal{T}_d}}$ {\bf is not a sheaf} on
$(\mathbb{R}^d, \mathcal{T}_d)$ under the restriction
$f\upharpoonright\mathcal{O}=f\, |{^*\mathcal{O}}$. Indeed, let
$\Omega=\mathbb{R}_+$ and $\Omega_n= (0, n)$ for $n\in\mathbb{N}$. Let
$\varphi\in\mathcal{D}(\mathbb{R}_+), \varphi\not=0$, and let $\nu$ be an infinitely large number in
$^*\mathbb{R}_+$ (see Example~\ref{Ex: Infinitesimals}). We define $f(x)={^*\varphi}(x-\nu)$ for all
$x\in{^*\mathbb{R}_+}$. It is clear that $\bigcup_{n\in\mathbb{N}}(0, n)=\mathbb{R}_+$ and
$f\upharpoonright (0, n)= f|^*(0, n)=0$ for all $n$. Yet, $f\upharpoonright\mathbb{R}_+=
f|^*\mathbb{R}_+=f\not= 0$. 
\end{remark}

	{\bf Our conclusion is}  that in order to convert the non-standard smooth
functions $^*\mathcal{E}(\Omega)$ into an algebra of generalized functions, {\bf we have to perform
a factorization of the space} $^*\mathcal{E}(\Omega)$. A general method for such factorization
will be presented in Section~\ref{S: Local Properties of Asymptotic Functions}. 

\section{Asymptotic Fields}\label{S: Asymptotic Fields}

	Let $\kappa$ be an infinite cardinal and let $^*\mathbb{C}$ be a $\kappa$-saturated non-standard
extension of  the field of the complex numbers $\mathbb{C}$. We describe those {\em
algebraically closed subfields} $\widehat{\mathcal{M}}$ of $^*\mathbb{C}$ which are {\em Cantor 
$\kappa$-complete}.  The fields $\widehat{\mathcal{M}}$ are constructed as factor rings of a given
convex subring $\mathcal{M}$ of $^*\mathbb{C}$. We call these fields {\bf $\mathcal{M}$-asymptotic
fields}  and their elements {\bf $\mathcal{M}$-asymptotic numbers}. 
Our {\em asymptotic field construction} can be viewed as a generalization of A. Robinson's theory
of asymptotic numbers (Lightstone \& Robinson~\cite{LiRob}). In our approach Robinson
field $^\rho\mathbb{C}$ appears as a subfield of $^*\mathbb{C}$. We also
generalize some more recent results in (T. Todorov and R. Wolf~\cite{TodWolf}) on the A. Robinson field
$^\rho\mathbb{R}$. A
construction similar to the presented here appears in the H. Vernaeve Ph.D. Thesis~\cite{hVernaevePhD}
(for a comparison see the equivalence relation $\sim$ defined on p. 87, Sec. 3.6, altered by the additional
condition used in Lemma 3.32 on p. 89).

	Algebraically closed non-archimedean fields had been studied in model theory of fields
(Ribenboim~\cite{pRib}) in the form of generalized power series (Hahn~\cite{hHahn}). These fields are
usually defined without connection with non-standard analysis. In contrast, in our approach many of these
fields are defined in the framework of 
$^*\mathbb{C}$ and they can be embedded as subfields of
$^*\mathbb{C}$. In particular, we show that the Hahn field of generalized power series
$\mathbb{C}(t^\mathbb{R})$ and the logarithmic-exponential field  
$\mathbb{R}((t))^{LE}$ (Marker, Messmer \& Pillay~\cite{MarkerMessPill}) are subfields of
$^*\mathbb{C}$. For that reason we hope  that our
{\em asymptotic field construction} might facilitate the communication between the mathematicians working in
non-standard analysis and those working in model theory of fields.

	The {\bf main purpose} of our algebraic approach however is to support the {\bf
theory of
$\mathcal{M}$-asymptotic functions} $\widehat{\mathcal{M}(\Omega)}$ on an open set
$\Omega\subseteq\mathbb{R}^d$ presented in Section~\ref{S: M-Asymptotic Functions}. Each
$\widehat{\mathcal{M}(\Omega)}$ is an algebra of generalized functions over field of scalars
$\widehat{\mathcal{M}}$.  We show that each
$\widehat{\mathcal{M}(\Omega)}$ contains a copy of the space of Schwartz distributions
$\mathcal{D}^\prime(\Omega)$ and in this sense $\widehat{\mathcal{M}(\Omega)}$ are algebras of
Colombeau type (Colombeau~\cite{jCol84a}-\cite{jCol85}). In particular, we show that the space of
non-standard functions $^*\mathcal{E}(\Omega)$ also contains a copy of  $\mathcal{D}^\prime(\Omega)$.
Here $\mathcal{E}(\Omega)=\mathcal{C}^\infty(\Omega)$ stands for the usual class of
$\mathcal{C}^\infty$-functions on
$\Omega$ and $^*\mathcal{E}(\Omega)$ is its non-standard extension.

\section{Convex Rings in $^*\mathbb{C}$}\label{S: Convex Rings in *C}

	In what follows  $^*\mathbb{R}$ stands for a non-standard extension of the field of real numbers $\mathbb{R}$
and $^*\mathbb{C}={^*\mathbb{R}(i)}$. If $S\subseteq{^*\mathbb{C}}$, then $\mathcal{I}(S), \mathcal{F}(S)$
and
$\mathcal{L}(S)$ stand for the sets of infinitesimal, finite and infinitely large numbers in $S$, respectively. 

\begin{definition}[Convex Rings]\label{D: Convex Rings}  We say that a subring
$\mathcal{M}$ of $^*\mathbb{C}$ is {\bf convex} in
$^*\mathbb{C}$ if $(\forall z\in{^*\mathbb{C}})(\forall \zeta\in\mathcal{M})(|z|\leq
|\zeta| \Rightarrow z\in\mathcal{M})$. We denote by
$\mathcal{M}_0$ the set of all {\bf non-invertible elements} of $\mathcal{M}$, i.e.
\begin{equation}\label{E: Non-Invertible}
\mathcal{M}_0=\{z\in\mathcal{M}\mid z=0 \; \text{or}\; 1/z\notin\mathcal{M}\}.
\end{equation}
\end{definition}

\begin{lemma}[Convex Rings]\label{L: Convex Rings} Let $\mathcal{M}$ be a convex subring of
${^*\mathbb{C}}$. Then:

\begin{description}
\item[\bf (i)]\quad $\mathcal{M}$ contains a copy of the ring
$\mathcal{F}({^*\mathbb{C}})$ of the finite elements of ${^*\mathbb{C}}$. Consequently, $\mathcal{M}$
contains a copy $\mathbb{C}$. We summarize all these as
$\mathbb{C}\subset\mathcal{F}({^*\mathbb{C}})\subseteq\mathcal{M}\subseteq{^*\mathbb{C}}$. 

\item[\bf (ii)]  ${^*\mathbb{R}\cap\mathcal{M}}$ is
a real ring  and  ${^*\mathbb{R}\cap\mathcal{M}}$ is
convex in $^*\mathbb{R}$ in the sense that if $x\in{^*\mathbb{R}}$ and
$y\in{^*\mathbb{R}\cap\mathcal{M}}$, then $|x|\leq |y|$ implies
$x\in{^*\mathbb{R}\cap\mathcal{M}}$. Also, $\mathcal{F}({^*\mathbb{R}})\subseteq
{^*\mathbb{R}\cap\mathcal{M}}$.

\item[\bf (iii)] $(^*\mathbb{R}\cap\mathcal{M})(i)=\mathcal{M}$, where
$(^*\mathbb{R}\cap\mathcal{M})(i)=:\{x+iy : x, y\in{^*\mathbb{R}}\cap\mathcal{M} \}$.

\item[\bf (iv)]  $\mathcal{M}$ is an
archimedean ring \ifff $\mathcal{M}=\mathcal{F}(^*\mathbb{C})$ \ifff  
${^*\mathbb{R}}\cap\mathcal{M}=\mathcal{F}(^*\mathbb{R})$.

\item[\bf (v)]  $\mathcal{M}$ is a field \ifff $\mathcal{M}={^*\mathbb{C}}$ \ifff
${^*\mathbb{R}}\cap\mathcal{M}$ is a field \ifff ${^*\mathbb{R}}\cap\mathcal{M}={^*\mathbb{R}}$.
\end{description}
\end{lemma}

\begin{proof} (i) Observe that ${^*\mathbb{R}}\cap\mathcal{M}$ is a totally ordered ring (as a subring of
$^*\mathbb{R}$) and thus it contains the ring of the integers $\mathbb{Z}$. Thus  $\mathcal{M}$ contains
$\mathbb{Z}$. With this in mind, suppose
$z\in\mathcal{F}(^*\mathbb{C})$, i.e. 
$|z|\le n$ for some
$n\in\mathbb{N}$. The latter implies $z\in\mathcal{M}$ (as desired) by the convexity of $\mathcal{M}$
since  $n\in\mathcal{M}$.

	(ii) follows immediately from (i).

(iii)  $(^*\mathbb{R}\cap\mathcal{M})(i)\subseteq\mathcal{M}$ holds because both
$^*\mathbb{R}$ and
$\mathcal{M}$ are rings and $\mathbb{C}\subset\mathcal{M}$ by (i). To  show that
$(^*\mathbb{R}\cap\mathcal{M})(i)\supseteq\mathcal{M}$ observe that  $\mathcal{M}\subseteq\Re
e(\mathcal{M})+i\Im m(\mathcal{M})$, where
$\Re e(\mathcal{M})=\{\Re e(z) : z\in\mathcal{M}\}$ and $\Im m(\mathcal{M})=\{\Im m(z) :
z\in\mathcal{M}\}$. It remains to show that
$\Re e(\mathcal{M})=\Im m(\mathcal{M})={^*\mathbb{R}}\cap\mathcal{M}$. Indeed,
$^*\mathbb{R}\cap\mathcal{M}$ is (trivially) a subset of 
$\Re e(\mathcal{M})$. Also, $^*\mathbb{R}\cap\mathcal{M}\subseteq\Im m(\mathcal{M})$, because 
$y\in{^*\mathbb{R}\cap\mathcal{M}}$ implies
$iy\in\mathcal{M}$ which implies $\Im m(iy)=y$. Finally,
$\Re e(\mathcal{M})\subseteq\mathcal{M}$ and
$\Im m(\mathcal{M})\subseteq\mathcal{M}$ by the convexity of $\mathcal{M}$.

(iv) Notice that $\mathcal{F}(^*\mathbb{C})$ is an archimedean ring (by the definition of
$\mathcal{F}(^*\mathbb{C})$). Suppose (on the contrary) that there exists
$\lambda\in\mathcal{M}\setminus\mathcal{F}(^*\mathbb{C})$. That means that $\lambda$ is
infinitely large number, i.e. $n<|\lambda|$ for all $n\in\mathbb{N}$. Thus
$\mathcal{M}$ is a non-archimedean ring.

	(v) Let $\mathcal{M}$ be a field and suppose (on
the contrary) that there exists
$\lambda\in{^*\mathbb{C}}\setminus\mathcal{M}$. We choose $\zeta\in\mathcal{M}, \zeta\not= 0$, and
observe that $|\lambda|>|\zeta|$ by the convexity of 
$\mathcal{M}$. The latter implies $|1/\lambda|<|1/\zeta|$ which implies $1/\lambda\in\mathcal{M}$
again by the convexity of $\mathcal{M}$ since $1/\zeta\in\mathcal{M}$. Thus $\lambda\in\mathcal{M}$
since $\mathcal{M}$ is a field, a contradiction. This  reverse is clear since $^*\mathbb{C}$ is a field.  
\end{proof}

\begin{lemma}[Convex Ideals]\label{L: Convex Ideals} Let $\mathcal{M}$ be a convex subring of
${^*\mathbb{C}}$ and let $\mathcal{M}_0$ be the set of the non-invertible elements of $\mathcal{M}$
(Definition~\ref{D: Asymptotic Fields}). Then
\begin{description}
\item[(i)]\;  If $z\in{^*\mathbb{C}},\;  z\not=0$, then
$z\in{\mathcal{M}_0}$ \ifff $1/z\notin\mathcal{M}$. Consequently, we have 
$\mathcal{M}_0=\{z\in{^*\mathbb{C}}\mid z=0\;
\text{or}\; 1/z\notin\mathcal{M}\}$.

\item[(ii)] $\mathcal{M}_0$ consists of \textbf{infinitesimals only}, i.e.
$\mathcal{M}_0\subseteq\mathcal{I}(^*\mathbb{C})$.

\item[(iii)]  $\mathcal{M}_0$ is a \textbf{convex maximal ideal} in $\mathcal{M}$, i.e. $\mathcal{M}_0$ is a
maximal ideal in $\mathcal{M}$ such that if $z\in\mathcal{M}$ and $h\in\mathcal{M}_0$, then $|z|\leq |h|$
implies
$z\in\mathcal{M}_0$. Consequently, $\mathcal{M}$ is a \textbf{local ring} ($\mathcal{M}_0$ is the only
maximal ideal in $\mathcal{M}$.)

\item[(iv)]  $^*\mathbb{R}\cap\mathcal{M}_0$ is a convex maximal ideal in
$^*\mathbb{R}\cap\mathcal{M}$. Consequently, $^*\mathbb{R}\cap\mathcal{M}$ is also a \textbf{local real
ring} (i.e. $^*\mathbb{R}\cap\mathcal{M}_0$ is the only maximal ideal in $^*\mathbb{R}\cap\mathcal{M}$.)

\item[(v)] The sets $\mathcal{M}_0$,\;  $\mathcal{M}\setminus\mathcal{M}_0$ and
$^*\mathbb{C}\setminus\mathcal{M}$ are \textbf {disconnected} in the sense that 
\[
(\forall
z_1\in\mathcal{M}_0)(\forall z_2\in\mathcal{M}\setminus\mathcal{M}_0)(\forall
z_3\in{^*\mathbb{C}}\setminus\mathcal{M})(|z_1|<|z_2|<|z_3|).
\]
\end{description}
\end{lemma}
\begin{proof} (i) Let $z\in{^*\mathbb{C}}, z\not= 0$. We have $1/z\notin\mathcal{M}\Rightarrow
|1/z|>1\Rightarrow |z|<1\Rightarrow$ 
$ z\in\mathcal{M}$ by the convexity of $\mathcal{M}$. The latter (along with $1/z\notin\mathcal{M}$)
implies
$z\in\mathcal{M}_0$ as required. 

	(ii) We observe that $x\in{\mathcal{M}_0}\setminus\mathcal{I}(^*\mathbb{C})$ implies
$1/x\notin\mathcal{M}$, which implies $1/x\notin\mathcal{F}(^*\mathbb{C})$ since
$\mathcal{F}(^*\mathbb{C})\subseteq\mathcal{M}$ by (i) of Lemma~\ref{L: Convex Rings}. The latter implies
$1/x\in\mathcal{L}(^*\mathbb{C})$, which implies
$x\in\mathcal{I}(^*\mathbb{C})$, a contradiction. 

	(iii) Let $z\in\mathcal{M}$ and $h\in\mathcal{M}_0$ and suppose (on the contrary)
that
$zh\notin\mathcal{M}_0$. It follows that $1/zh\in\mathcal{M}$ which implies
$1/h\in\mathcal{M}$ thus $h\notin\mathcal{M}_0$ by (i), a contradiction. The fact that $\mathcal{M}_0$
is closed under the addition follows from (i). Indeed, let
$h_1, h_2\in\mathcal{M}_0$. If $h_1=0$ or $h_2=0$
there is nothing to prove. Let
$h_1\not=0$ or
$h_2\not=0$ and suppose (on the contrary) that $h_1+ h_2\notin\mathcal{M}_0$. It follows $1/(h_1+
h_2)\in\mathcal{M}$ (by the definition of $\mathcal{M}_0$ (\ref{E: Non-Invertible})), which implies
$h_1/(h_1+ h_2)\in\mathcal{M}_0$ (by what was proved above) implying
$(h_1+h_2)/h_1\notin\mathcal{M}$ (by the definition of $\mathcal{M}_0$ again)
implying
$h_2/h_1\notin\mathcal{M}$. The latter implies
$h_1/h_2\in\mathcal{M}_0$ by (i). Similarly, we conclude that
$h_2/h_1\in\mathcal{M}_0$. Thus
$1\in\mathcal{M}_0$, a contradiction. The ideal $\mathcal{M}_0$ is maximal (and $\mathcal{M}$ is a
local ring) because
$\mathcal{M}_0$ consists of all non-invertible elements of $\mathcal{M}$. To show the convexity of
$\mathcal{M}_0$, observe that $h=0$ implies (trivially) $z=0$. Let $h\not= 0$ and suppose (on the contrary)
that $z\notin\mathcal{M}_0$, i.e. $1/z\in\mathcal{M}$. The latter implies
$1/h\in\mathcal{M}$ by the convexity of $\mathcal{M}$ which implies $h\notin\mathcal{M}_0$, a
contradiction. 

	(iv) follows directly from (iii).

	(v) follows directly from the convexity of both $\mathcal{M}$ and $\mathcal{M}_0$ and (i).
\end{proof}
\newpage

\section{Examples of Convex Rings}\label{S: Examples of Convex Rings}

	We present several examples for convex subrings of $^*\mathbb{C}$ and their maximal ideals
(Section~\ref{S: Convex Rings in *C}). 

\begin{definition}[Generating Sequences]\label{D: Generating Sequences}
\begin{description}
\item{\bf (i)} A \textbf{decreasing sequence} $(\delta_n)$ of infinitely large positive numbers in
$^*\mathbb{R}$ is called \textbf{generating} if 
\begin{description}
\item[(a)]  For every $n\in\mathbb{N}$ there exists $m\in\mathbb{N}$
such that $2\delta_m\leq\delta_n$.
\item[(b)] For every $n\in\mathbb{N}$ there exists $m\in\mathbb{N}$
such that $\delta^2_m\leq\delta_n$.
\end{description}
\item{\bf (ii)} An \textbf{increasing sequence} $(\lambda_n)$ of infinitely large positive numbers in
$^*\mathbb{R}$ is called \textbf{generating} if 
\begin{description}
\item[(a)]  For every $n\in\mathbb{N}$ there exists $m\in\mathbb{N}$
such that $2\lambda_n\leq\lambda_m$.
\item[(b)] For every $n\in\mathbb{N}$ there exists $m\in\mathbb{N}$
such that $\lambda^2_n\leq\lambda_m$.
\end{description}
\end{description}
\end{definition}

		The next lemma is useful for generating examples of convex subrings (see below).

\begin{lemma}[Generated Rings]\label{L: Generated Rings}
\begin{description}
\item{\bf (i)} Let $(\delta_n)$ be a decreasing generating sequence. Then 
\[
\mathcal{M}=\{z\in{^*\mathbb{C}} : |z|<\delta_n\;  \text{for all}\; n\in\mathbb{N}\},
\]
is a convex subring of $^*\mathbb{C}$ and its (unique) maximal ideal  is given by 
\[
\mathcal{M}_0=\{z\in{^*\mathbb{C}} : |z|\leq 1/\delta_n\;  \text{for some}\; n\in\mathbb{N}\}.
\]
We say that $\mathcal{M}$ is generated by the sequence $(\delta_n)$.
\item{\bf (ii)}  Let $(\lambda_n)$ be an increasing generating sequence. Then 
\[
\mathcal{M}=\{z\in{^*\mathbb{C}} : |z|\leq\lambda_n\;  \text{for some}\; n\in\mathbb{N}\},
\]
is a convex subring of $^*\mathbb{C}$ and its (unique) maximal ideal  is given by
\[
\mathcal{M}_0=\{z\in{^*\mathbb{C}} : |z|< 1/\lambda_n\;  \text{for all}\; n\in\mathbb{N}\}.
\]
We say that $\mathcal{M}$ is generated by the sequence $(\lambda_n)$.
\end{description}
\end{lemma}

\begin{proof} The proof is immediate and we leave the  proof to the reader. 
\end{proof}

	Here are {\bf several example} of convex subrings of $^*\mathbb{C}$ and their maximal ideals.  All but
the first example are about non-archimedean rings. 

\begin{example}[Finite Numbers]\label{Ex: Finite Numbers} The ring of the finite complex
non-standard numbers  $\mathcal{F}(^*\mathbb{C})$ is a convex subring of $^*\mathbb{C}$. Its maximal ideal is
the set of infinitesimals $\mathcal{I}(^*\mathbb{C})$. We shall often write $\mathcal{F}$ and $\mathcal{I}$
instead of  $\mathcal{F}(^*\mathbb{C})$ and  $\mathcal{I}(^*\mathbb{C})$, respectively.
\end{example}

\begin{example}[Multiple-Logarithmic Rings]\label{Ex: Multiple-Logarithmic Rings} Let 
$\rho$ be a positive infinitesimal in
${^*\mathbb{R}}$ and let 
\[
\mathcal{L}_\rho=\{z\in{^*\mathbb{C}}: \;
|z|< \log_n{(1/\rho)} \text{\; for all\; } n\in\mathbb{N}\}.
\]
Here $\log_1(x)={^*\ln{x}}$, where $^*\ln{x}$
is the non-standard extension of the usual natural logarithmic function $\ln{x}$, and we also define
$\log_2={^*\ln}\circ{^*\ln}$, and $\log_{n}={^*\ln}\circ{^*\ln}\circ\dots\circ{^*\ln}$ ($n$ times). Notice
that $(\log_n{(1/\rho)})$ is a decreasing generating sequence in
$^*\mathbb{R}$ (Definition~\ref{D: Generating Sequences}). Indeed, for every $n$ we have
$\lim_{x\to\infty}\frac{\log_{n+1}{x}}{\log_n{x}}=0$ by the L'Hopital rule. Thus
$\frac{\log_{n+1}{(1/\rho)}}{\log_n{(1/\rho)}}\approx 0$ implying
$2\log_{n+1}{(1/\rho)}<\log_{n}{(1/\rho)}$ for all $n$. Similarly,
$\lim_{x\to\infty}\frac{(\log_{n+1}{x})^2}{\log_n{x}}=0$ by L'Hopital rule. Thus
$(\log_{n+1}{(1/\rho)})^2<\log_{n}{(1/\rho)}$ for all $n$. Consequently, 
$\mathcal{L}_\rho(^*\mathbb{C})$ is a convex subring of
$^*\mathbb{C}$ by Lemma~\ref{L: Generated Rings}. For its maximal ideal we have
\[
\mathcal{L}_{\rho,0}=\{z\in{^*\mathbb{C}}: \;
|z|\leq \frac{1}{\log_n(1/\rho)} \text{\; for some\; } n\in\mathbb{N}\}.
\]
Let $\nu$ be an infinitely large number in $^*\mathbb{N}$. Then
$\log_\nu{(1/\rho)}$ is a typical element of
$\mathcal{L}_\rho\setminus\mathcal{L}_{\rho,0}$.  Here
$\log_\nu{(1/\rho)}={^*\ln{(^*\ln{(\dots^*\ln{({^*\ln{(1/\rho}})}}}\dots)))}$ ($\nu$ times) hence, the name
\textbf{multiple logarithmic ring} for $\mathcal{L}_\rho$.
\end{example}

\begin{example}[Logarithmic Rings]\label{Ex: Logarithmic Rings} Let 
$\rho$ be (as before) a positive infinitesimal in
${^*\mathbb{R}}$. We define 
\begin{equation}\label{E: rho-finite}
\mathcal{F}_\rho=\{z\in{^*\mathbb{C}}: \;
|z|<1/\sqrt[n]{\rho} \text{\; for all\; } n\in\mathbb{N}\}.
\end{equation}
We observe that
$(1/\sqrt[n]{\rho})$ is a decreasing generating sequence in $^*\mathbb{R}$ (Definition~\ref{D: Generating
Sequences}) because
$2/\sqrt[n+1]{\rho}<1/\sqrt[n]{\rho}$ and also (trivially) $(1/\sqrt[2n]{\rho})^2\leq1/\sqrt[n]{\rho}$ for all
$n$. Thus
$\mathcal{F}_\rho$ is a covex subring of $^*\mathbb{C}$ by Lemma~\ref{L: Generated Rings}.
For its maximal ideal we have 
\begin{equation}\label{E: rho-infinitesimal}
\mathcal{I}_\rho=\{z\in{^*\mathbb{C}} : \;
|z|\leq\sqrt[n]{\rho} \text{\, for some\, } n\in\mathbb{N}\}.
\end{equation}
We call $\mathcal{F}_\rho$ \textbf{logarithmic
rings} because $\ln{\rho}$ is a typical element of
$\mathcal{F}_\rho$. The numbers in $\mathcal{F}_\rho$ are called
\textbf{logarithmic numbers} or \textbf{$\rho$-finite numbers}. The numbers in
$\mathcal{I}_\rho$ are called
\textbf{$\rho$-infinitesimal numbers}. Notice as well that 
$\sqrt[\nu]{\rho}\in\mathcal{F}_\rho$
for every $\nu\in{^*\mathbb{N}}\setminus\mathbb{N}$.
\end{example}

\begin{example}[Robinson Rings]\label{Ex: Robinson Rings}
Let $\rho$ be (as before) a positive infinitesimal in
${^*\mathbb{R}}$. The ring of the the \textbf{$\rho$-moderate non-standard numbers} is defined by
\begin{equation}\label{E: rho-moderate}
\mathcal{M}_\rho=\{z\in{^*\mathbb{C}} : \;
|z|\leq\rho^{-n} \text{\: for some\;}
n\in\mathbb{N}\},
\end{equation}
(Robinson~\cite{aRob73}). Notice that $(\rho^{-n})$ is an increasing generating sequence in $^*\mathbb{R}$
(Definition~\ref{D: Generating Sequences}). Thus 
$\mathcal{M}_\rho$ is
a covex subring of $^*\mathbb{C}$ by Lemma~\ref{L: Generated Rings}. For its maximal ideal we have
\begin{equation}\label{E: rho-negligible}
\mathcal{N}_\rho=\{z\in{^*\mathbb{C}}:\;
|z|\leq\rho^{n} \text{\: for all\;}
n\in\mathbb{N}\}.
\end{equation}
We call the numbers in $\mathcal{N}_\rho$
\textbf{$\rho$-negligible} (or \textbf{iota numbers}). The numbers in
$^*\mathbb{C}\setminus\mathcal{M}_\rho$ are called \textbf{mega numbers}. If
$x\in\mathbb{R}$, then
$\rho^x$ is a typical
element of $\mathcal{M}_\rho\setminus\mathcal{N}_\rho$.
\end{example}

\begin{example}[Logarithmic-Exponential Rings] \label{Ex: Logarithmic-Exponential Rings} Let 
$\rho$ be (as before) a positive infinitesimal in
${^*\mathbb{R}}$ and let 
\[
\mathcal{E}_\rho=\{z\in{^*\mathbb{C}} : \;
|z|\leq\exp_n(1/\rho) \text{\: for some\;} n\in\mathbb{N}\}.
\]
Here $\exp_1(x)={^*e}^x$ is
the non-standard extension of the usual natural exponential function $e^x$, and we also let
$\exp_2=\exp_1\circ\, {\exp_1}$, and
$\exp_{n}=\exp_1\circ\exp_1\circ\dots\circ\exp_1$ ($n$ times). We observe
that $(\exp_n(1/\rho))$ is an increasing generating sequence in
$^*\mathbb{R}$ (Definition~\ref{D: Generating Sequences}). Indeed, for every $n$ we have
$\lim_{x\to\infty}\frac{\exp_{n}{x}}{\exp_{n+1}{x}}=0$ by the L'Hopital rule. Thus
$\frac{\exp_{n}{(1/\rho)}}{\exp_{n+1}{(1/\rho)}}\approx 0$ implying
$2\exp_{n}{(1/\rho)}<\exp_{n+1}{(1/\rho)}$ for all $n$. Similarly,
$\lim_{x\to\infty}\frac{(\exp_{n}{x})^2}{\exp_{n+1}{x}}=0$ by L'Hopital rule implying
$\frac{(\exp_{n}{(1/\rho)})^2}{\exp_{n+1}{(1/\rho)}}\approx 0$. Thus 
$(\exp_{n}{(1/\rho)})^2<\exp_{n+1}{(1/\rho)}$ for all $n$. Consequently,
$\mathcal{E}_\rho$ is a convex subring of $^*\mathbb{C}$ by Lemma~\ref{L: Generated Rings}. 
For its (unique) maximal ideal we have
\[
\mathcal{E}_{\rho,0}=\{z\in{^*\mathbb{C}}:\;
|z|< \frac{1}{\exp_n(1/\rho)} \text{\: for all\;} n\in\mathbb{N}\}.
\]
The numbers $e^{1/\rho},\, \ln{\rho}$ are both
in $\mathcal{E}_\rho\setminus \mathcal{E}_{\rho,0}$ hence, the name
\textbf{logarithmic-exponential rings} for $\mathcal{E}_\rho$.
\end{example}

\begin{example}[Non-Standard Complex Numbers]\label{Ex: Non-Standard Complex Numbers} The field of the
complex numbers ${^*\mathbb{C}}$ is (trivially) a convex subring of
$^*\mathbb{C}$. Its maximal ideal is $\{0\}$. 
\end{example}

	We observe that 
\begin{align}\label{E: Ring Inclusions}
&\mathcal{F}\subset\mathcal{L}_\rho\subset\mathcal{F}_\rho\subset
\mathcal{M}_\rho\subset
\mathcal{E}_\rho\subset{^*\mathbb{C}},\\
&\{0\}\subset\mathcal{E}_{\rho,0}\subset\mathcal{N}_\rho\subset
\mathcal{I}_\rho\subset\mathcal{L}_{\rho,0}\subset\mathcal{I}.
\end{align}
\section{Spilling Principles}\label{S: Spilling Principles}
		In this section we present several \textbf{spilling principles} for a given convex subring
$\mathcal{M}$ of $^*\mathbb{C}$ and its maximal ideal (Section~\ref{S: Convex Rings in *C}).
These principles generalize the more familiar
\textbf{underflow and overflow principles} in non-standard analysis (Corollary~\ref{C: The Usual
Spilling Principles}). Also in Corollary~\ref{C:
Lightstone/Robinson's Principles of Permanence} we show
that our spilling principles reduce to the Forth, Fifth and Sixth Principle of
Permanence due to Lightstone\&Robinson~(\cite{LiRob}, p. 97-99) in the
particular case
$\mathcal{M}=\mathcal{M}_\rho$ (Example~\ref{Ex: Robinson Rings}). 

	Let $X$ and $Y$ be two subsets of $^*\mathbb{C}$. We say
that $X$ {\em contains arbitrarily large numbers} in $Y$ if $X\cap Y\not=\varnothing$ and $(\forall
z\in X\cap Y)(\exists \zeta\in X\cap Y)(|z|<|\zeta|)$. Similarly, we say that
$X$ {\em contains arbitrarily small numbers} in $Y$ if $X\cap Y\not=\varnothing$ and $(\forall z\in
X\cap Y)(\exists \zeta\in X\cap Y)(|z|>|\zeta|)$. With this in mind we have the following result.

\begin{theorem}[Spilling Principles]\label{T: Spilling Principles} Let $\mathcal{M}$ be a
convex subring of $^*\mathbb{C}$ (Section~\ref{S: Convex Rings in *C}) and
$\mathcal{A}\subseteq{^*\mathbb{C}}$ be an internal set.  Then:

\begin{description}
\item{\bf (i) Overflow of $\mathcal{M}$\, :} If $\mathcal{A}$ contains arbitrarily large numbers in
$\mathcal{M}$, then $\mathcal{A}$ contains arbitrarily small numbers in
$^*\mathbb{C}\setminus\mathcal{M}$. In particular, 
\[
\mathcal{M}\setminus\mathcal{M}_0\subset\mathcal{A} \Rightarrow
\mathcal{A}\cap(^*\mathbb{C}\setminus\mathcal{M})\not=\varnothing.
\]

\item{\bf (ii) Underflow of $\mathcal{M}\setminus\mathcal{M}_0$ :} If $\mathcal{A}$ contains
arbitrarily small numbers in
\newline
$\mathcal{M}\setminus\mathcal{M}_0$, then $\mathcal{A}$ contains arbitrarily large numbers in
$\mathcal{M}_0$. In particular, 
\[
\mathcal{M}\setminus\mathcal{M}_0\subset\mathcal{A}
\Rightarrow\mathcal{A}\cap\mathcal{M}_0\not=\varnothing.
\]
\item{\bf (iii) Overflow of $\mathcal{M}_0$ :} If $\mathcal{A}$ contains arbitrarily large numbers
in\newline $\mathcal{M}_0$, then $\mathcal{A}$ contains arbitrarily small
numbers in $\mathcal{M}\setminus\mathcal{M}_0$. In particular, 
\[
\mathcal{M}_0\subset\mathcal{A}
\Rightarrow\mathcal{A}\cap(\mathcal{M}\setminus\mathcal{M}_0)\not=\varnothing.
\]
\item{\bf (iv) Underflow of $^*\mathbb{C}\setminus\mathcal{M}$\, :} If $\mathcal{A}$
contains arbitrarily small numbers in \newline $^*\mathbb{C}\setminus\mathcal{M}$, then
$\mathcal{A}$ contains arbitrarily large numbers in $\mathcal{M}$.
In particular, 
\[
^*\mathbb{C}\setminus\mathcal{M}\subset\mathcal{A}
\Rightarrow\mathcal{A}\cap(\mathcal{M}\setminus\mathcal{M}_0)\not=\varnothing.
\]
\end{description}
\end{theorem}

\begin{proof} (i) If $\mathcal{A}$ is unbounded in $^*\mathbb{C}$, there is nothing to prove. If 
$\mathcal{A}$ is bounded in $^*\mathbb{C}$, then $\sup (|\mathcal{A}|)=x$ exists in
$^*\mathbb{R}$, where $|\mathcal{A}|=\{|z| : z\in\mathcal{A}\}$. Notice that
$x\notin\mathcal{M}$ because $x\in\mathcal{M}$ contradicts the assumption for $\mathcal{A}$.
Next, there exists $z\in\mathcal{A}$ such that $x/2<|z|<x$ by the choice of $x$ and
we have $z\notin\mathcal{M}$ because
$x/2\notin\mathcal{M}$ (notice that  $x/2\in\mathcal{M}$ implies $x/2+x/2\in\mathcal{M}$). We just
proved that $\mathcal{A}\cap(^*\mathbb{C}\setminus\mathcal{M})\not=\varnothing$. It
remains to show that $\mathcal{A}\cap(^*\mathbb{C}\setminus\mathcal{M})$ does not have a
lower bound in $^*\mathbb{C}\setminus\mathcal{M}$. Suppose (on the contrary)
that there exists $\lambda\in{^*\mathbb{C}\setminus\mathcal{M}}$ such that
$\lambda\leq|z|$ for all $z\in\mathcal{A}\cap(^*\mathbb{C}\setminus\mathcal{M})$. The set
$\mathcal{A}_\lambda=\{z\in\mathcal{A}: |z|<\lambda\}$ is internal and we have
$\mathcal{A}_\lambda=\mathcal{A}\cap\mathcal{M}$ by the choice of $\lambda$. It follows
that $\mathcal{A}_\lambda$ has arbitrarily large elements in 
$\mathcal{M}$ because $\mathcal{A}$ has arbitrarily large elements in 
$\mathcal{M}$ by assumption. We conclude that
$\mathcal{A}_\lambda\cap(^*\mathbb{C}\setminus\mathcal{M})\not=\varnothing$ by what was
proved above. Thus there exists $z\in\mathcal{A}\cap(^*\mathbb{C}\setminus\mathcal{M})$ such
that 
$|z|<\lambda$, a contradiction.

	(ii) follows immediately from (i) and the fact that $z\in\mathcal{M}\setminus\mathcal{M}_0$ implies
$1/z\in\mathcal{M}\setminus\mathcal{M}_0$ and also that
$z\in{^*\mathbb{C}}\setminus\mathcal{M}$ implies  $1/z\in\mathcal{M}_0$ by part (i) of Lemma~\ref{L:
Convex Ideals}.

	The proof of (iii) is similar to the proof of (i) and we leave it to the reader. 

	(iv) follows immediately from (iii) and the fact that
$z\in\mathcal{M}_0\setminus\{0\}$  implies
$1/z\in{^*\mathbb{C}}\setminus\mathcal{M}$ and also that
$z\in\mathcal{M}\setminus\mathcal{M}_0$ implies
$1/z\in\mathcal{M}\setminus\mathcal{M}_0$.
\end{proof}

	Here are the  more familiar spilling (underflow and overflow) principles about
$\mathcal{F}(^*\mathbb{C})$, $\mathcal{I}(^*\mathbb{C})$ and
$\mathcal{L}(^*\mathbb{C})$. 

\begin{corollary}[The Usual Spilling Principles]\label{C: The Usual Spilling Principles} Let
$\mathcal{A}\subseteq{^*\mathbb{C}}$ be an internal set. Then:

\begin{description}
\item{\bf (i) Overflow of $\mathcal{F}(^*\mathbb{C})$:} If $\mathcal{A}$ contains arbitrarily
large finite numbers, then $\mathcal{A}$ contains arbitrarily small infinitely large numbers.
In particular, 
\[
\mathcal{F}(^*\mathbb{C})\setminus\mathcal{I}(^*\mathbb{C})\subset\mathcal{A}
\Rightarrow\mathcal{A}\cap\mathcal{L}(^*\mathbb{C})\not=\varnothing.
\]
\item{\bf (ii) Underflow of $\mathcal{F}(^*\mathbb{C})\setminus\mathcal{I}(^*\mathbb{C})$:} 
If $\mathcal{A}$ contains arbitrarily small finite non-infinitesimals, then $\mathcal{A}$
contains arbitrarily large infinitesimals. In particular, 
\[
\mathcal{F}(^*\mathbb{C})\setminus\mathcal{I}(^*\mathbb{C})\subset\mathcal{A}
\Rightarrow\mathcal{A}\cap\mathcal{I}(^*\mathbb{C})\not=\varnothing.
\]
\item{\bf (iii) Overflow of $\mathcal{I}(^*\mathbb{C})$:} If $\mathcal{A}$ contains arbitrarily
large infinitesimals, then $\mathcal{A}$ contains arbitrarily small finite non-infinitesimals.
In particular, 
\[
\mathcal{I}(^*\mathbb{C})\subset\mathcal{A}\Rightarrow
\mathcal{A}\cap(\mathcal{F}(^*\mathbb{C})\setminus\mathcal{I}(^*\mathbb{C}))\not=
\varnothing.
\]
\item{\bf (iv) Underflow of $\mathcal{L}(^*\mathbb{C})$:} If $\mathcal{A}$
contains arbitrarily small infinitely large numbers, then
$\mathcal{A}$ contains arbitrarily large finite numbers. In particular, 
\[
\mathcal{L}(^*\mathbb{C})\subset\mathcal{A}\Rightarrow
\mathcal{A}\cap(\mathcal{F}(^*\mathbb{C})\setminus\mathcal{I}(^*\mathbb{C}))\not=
\varnothing.
\]
\end{description}
\end{corollary}

\begin{proof} The result follows directly from the previous theorem in the particular case of 
$\mathcal{M}=\mathcal{F}(^*\mathbb{C})$ taking into account that in this case
$\mathcal{M}_0=\mathcal{I}(^*\mathbb{C})$ (Example~\ref{Ex:
Finite Numbers}).      
\end{proof}

\begin{corollary}[Generating Sequences]\label{C: Generating Sequences}
\begin{description}
\item{\bf (i)}  Let $(\delta_n)$ be a decreasing
generating  sequence in
$^*\mathbb{R}$ and $\mathcal{M}$ be the convex subring of $^*\mathbb{C}$ generated by $(\delta_n)$
(part~(i) of Lemma~\ref{L: Generated Rings}). Let  $(^*\delta_n)$ be the non-standard extension of 
$(\delta_n)$. Then 
\begin{description}
\item[(a)] $z\in\mathcal{M}$ \ifff $(\exists \nu\in{^*\mathbb{N}\setminus\mathbb{N}})(
|z|\leq{^*\delta_\nu})$.
\item[(b)] $z\in\mathcal{M}_0$ \ifff $(\forall \nu\in{^*\mathbb{N}\setminus\mathbb{N}})(
|z|<1/{^*\delta_\nu})$.
\end{description}
\item{\bf (ii)}  Let $(\lambda_n)$ be a increasing generating 
sequence in
$^*\mathbb{R}$ and $\mathcal{M}$ be the convex subring of $^*\mathbb{C}$ generated by $(\lambda_n)$
(part~(ii) of Lemma~\ref{L: Generated Rings}). Let  $(^*\lambda_n)$ be the non-standard extension of 
$(\lambda_n)$. Then
\begin{description}
\item[(a)] $z\in\mathcal{M}$ \ifff $(\forall \nu\in{^*\mathbb{N}\setminus\mathbb{N}})(
|z|<{^*\lambda_\nu})$.
\item[(b)] $z\in\mathcal{M}_0$ \ifff $(\exists \nu\in{^*\mathbb{N}\setminus\mathbb{N}})( |z|\leq
1/{^*\lambda_\nu})$.
\end{description}
\end{description}
\end{corollary}

\begin{proof} (i)-(a): Suppose $z\in\mathcal{M}$. The internal set $\mathcal{A}=\{n\in{^*\mathbb{N}} :
|z|<{^*\delta_n}\}$ contains $\mathbb{N}$ by assumption hence,
there exists $\nu\in({^*\mathbb{N}\setminus\mathbb{N}})\cap\mathcal{A}$ (as required) by the overflow of
$\mathcal{F}(^*\mathbb{C})$ (Corollary~\ref{C: The Usual Spilling Principles}) since
${^*\mathbb{N}\setminus\mathbb{N}}\subset\mathcal{F}(^*\mathbb{C}))$. Conversely, $|z|<{^*\delta_\nu}$
for some $\nu\in{^*\mathbb{N}}\setminus\mathbb{N}$ implies $z\in\mathcal{M}$ by the convexity of
$\mathcal{M}$, since $^*\delta_\nu<\delta_n$ for all $n\in\mathbb{N}$.  

	(i)-(b) follows immediately from (i)-(a) and part~(i) of Lemma~\ref{L: Convex Ideals}.

	The proof of (ii) is similar to the proof of (i) and leave it to the reader. 
\end{proof}

		In the next corollary we derive the Robinson's
Principles of Permanence  as a particular case of our more general Spilling Principles
for $\mathcal{M}=\mathcal{M}_\rho$   (Example~\ref{Ex: Robinson
Rings}). We should note that in  Lightstone\&Robinson~(\cite{LiRob},
p. 97-99) the numbers in
$\mathcal{N}_\rho(^*\mathbb{C})$ are called \textbf{iota numbers} and the numbers in
$^*\mathbb{C}\setminus\mathcal{M}_\rho$ are called \textbf{mega
numbers}.

\begin{corollary}[Robinson's Principles of Permanence]\label{C: Robinson's Principles of Permanence} Let
$\mathcal{A}\subseteq{^*\mathbb{C}}$ be an internal set.  Then:

\begin{description}
\item{\bf (i) Overflow of $\mathcal{M}_\rho(^*\mathbb{C})$:} If $\mathcal{A}$ contains
arbitrarily large numbers in
$\mathcal{M}_\rho(^*\mathbb{C})$, then $\mathcal{A}$ contains arbitrarily small numbers in
$^*\mathbb{C}\setminus\mathcal{M}_\rho(^*\mathbb{C})$. In particular, 
\[
\mathcal{M}_\rho(^*\mathbb{C})\setminus\mathcal{N}_\rho(^*\mathbb{C})\subset
\mathcal{A}\Rightarrow
\mathcal{A}\cap(^*\mathbb{C}\setminus\mathcal{M}_\rho(^*\mathbb{C}))\not=\varnothing.
\]

\item{\bf (ii) Underflow of
$\mathcal{M}_\rho(^*\mathbb{C})\setminus\mathcal{N}_\rho(^*\mathbb{C})$:} If
$\mathcal{A}$ contains arbitrarily small numbers in
$\mathcal{M}_\rho(^*\mathbb{C})\setminus\mathcal{N}_\rho(^*\mathbb{C})$, then
$\mathcal{A}$ contains arbitrarily large numbers in
$\mathcal{N}_\rho(^*\mathbb{C})$. In particular, 
\[
\mathcal{M}_\rho(^*\mathbb{C})\setminus\mathcal{N}_\rho(^*\mathbb{C})\subset
\mathcal{A}\Rightarrow\mathcal{A}\cap\mathcal{N}_\rho(^*\mathbb{C})\not=\varnothing.
\]
\item{\bf (iii) Overflow of $\mathcal{N}_\rho(^*\mathbb{C})$:} If $\mathcal{A}$ contains
arbitrarily large numbers in\newline $\mathcal{N}_\rho(^*\mathbb{C})$, then $\mathcal{A}$
contains arbitrarily small numbers in
$\mathcal{M}_\rho(^*\mathbb{C})\setminus\mathcal{N}_\rho(^*\mathbb{C})$. In particular, 
\[
\mathcal{N}_\rho(^*\mathbb{C})\subset\mathcal{A}
\Rightarrow\mathcal{A}\cap(\mathcal{M}_\rho(^*\mathbb{C})\setminus
\mathcal{N}_\rho(^*\mathbb{C}))\not=\varnothing.
\]
\item{\bf (iv) Underflow of $^*\mathbb{C}\setminus\mathcal{M}_\rho(^*\mathbb{C})$:} If
$\mathcal{A}$ contains arbitrarily small numbers in
$^*\mathbb{C}\setminus\mathcal{M}_\rho(^*\mathbb{C})$, then
$\mathcal{A}$ contains arbitrarily large numbers in $\mathcal{M}_\rho(^*\mathbb{C})$.
In particular, 
\[
^*\mathbb{C}\setminus\mathcal{M}_\rho(^*\mathbb{C})\subset\mathcal{A}
\Rightarrow\mathcal{A}\cap(\mathcal{M}_\rho(^*\mathbb{C})\setminus
\mathcal{N}_\rho(^*\mathbb{C}))\not=\varnothing.
\]
\end{description}
\end{corollary}

\begin{proof} These results follow immediately from
our general Spilling Principles (Theorem~\ref{T: Spilling Principles}) for
$\mathcal{M}=\mathcal{M}_\rho$   (Example~\ref{Ex: Robinson Rings}). 
\end{proof}
\newpage

\section{Asymptotic Fields}\label{S: Asymptotic Fields}
In this section we study particular type of algebraically closed subfields of $^*\mathbb{C}$ called {\em
asymptotic fields}.
\begin{definition}[Asymptotic Fields]\label{D: Asymptotic Fields} \begin{enumerate} \item An
\textbf{asymptotic field} is a field of the form $\mathcal{M}/\mathcal{M}_0$, where  $\mathcal{M}$ is a
convex subring of $^*\mathbb{C}$ and $\mathcal{M}_0$ is its maximal ideal (Definition~\ref{D: Convex
Rings}). A field which is isomorphic to a field of the form $\mathcal{M}/\mathcal{M}_0$ will be also
called {\em asymptotic}. We shall call the elements of
$\mathcal{M}/\mathcal{M}_0$ {\bf complex $\mathcal{M}$-asymptotic numbers} (or simply {\em
asymptotic numbers} if no confusion could arise). We denote by
$q_\mathcal{M}: \mathcal{M}\to\mathcal{M}/\mathcal{M}_0$ the corresponding {\bf quotient mapping}.
\item The elements of $q_\mathcal{M}[{^*\mathbb{R}}\cap\mathcal{M}]$ are called {\bf real
$\mathcal{M}$-asymptotic numbers} (or simply {\em real asymptotic numbers} if no confusion could arise).
We shall sometime refer to $q_\mathcal{M}[{^*\mathbb{R}}\cap\mathcal{M}]$ as the \textbf{real part} of 
$\mathcal{M}/\mathcal{M}_0$ and denote it by $\Re
e(\mathcal{M}/\mathcal{M}_0)$.

\item We define the \textbf{order relation} in $q_\mathcal{M}[{^*\mathbb{R}}\cap\mathcal{M}]$ by:
$q_\mathcal{M}(x)>0$ if  $q_\mathcal{M}(x)\not=0$  and $x>0$ in $^*\mathbb{R}$.

\item We define the \textbf{absolute value} $|\cdot | :
\mathcal{M}/\mathcal{M}_0\to q_\mathcal{M}[{^*\mathbb{R}}\cap\mathcal{M}]$ by  the formula
$|q_\mathcal{M}(z)|=q_\mathcal{M}(|z|)$.
\end{enumerate}
\end{definition}
\begin{notation}[Suppressing $\mathcal{M}$]\label{N: Suppressing M} Let $\mathcal{M}$ be a convex subring
of $^*\mathbb{C}$ (Definition~\ref{D: Asymptotic Fields}). We shall often \textbf{suppress the
dependence on} $\mathcal{M}$ and use the following simplified notation:
\begin{description}
\item[(i)] If $z\in\mathcal{M}$, we shall often write $\widehat{z}$ instead of the more precise
$q_\mathcal{M}(z)$. Also, we shall write $z\to\widehat{z}$ for the quotient mapping $q_\mathcal{M}$.

\item[(ii)]  If $S\subseteq{^*\mathbb{C}}$, we denote
$\widehat{S}=q_\mathcal{M}[S\cap\mathcal{M}]$. Observe that we have 
$\widehat{^*\mathbb{C}}=\widehat{\mathcal{M}}=q_\mathcal{M}[\mathcal{M}]$. Also, we have
$\widehat{^*\mathbb{R}}=q_\mathcal{M}[{^*\mathbb{R}}\cap\mathcal{M}]$.   We shall often prefer the
simpler notation $\widehat{^*\mathbb{C}}$ instead of the more precise $\mathcal{M}/\mathcal{M}_0$ or
$\widehat{\mathcal{M}}$, when no confusion could arise. Also, we shall often write 
$\widehat{^*\mathbb{R}}$ instead of
$q_\mathcal{M}[{^*\mathbb{R}}\cap\mathcal{M}]$.  Summarizing, we have
\begin{equation}\label{E: Hat Notation}
\widehat{^*\mathbb{C}}=\widehat{\mathcal{M}}=q_\mathcal{M}[\mathcal{M}] \text{\; and\; } 
\widehat{^*\mathbb{R}}=\Re e(\widehat{^*\mathbb{C}})=\Re
e(\widehat{\mathcal{M}})=q_\mathcal{M}[{^*\mathbb{R}}\cap\mathcal{M}].
\end{equation}

\item[(iii)]  In this notation the
\textbf{order relation} in $\widehat{^*\mathbb{R}}$ (defined above) is phrase as follows: 
$\widehat{x}> 0$  in $\widehat{^*\mathbb{R}}$ if $\widehat{x}\not= 0$ and $x> 0$ in  $^*\mathbb{R}$. 

\item[(iv)]  In this notation the \textbf{absolute value} $|\cdot | :
{\widehat{^*\mathbb{C}}}\to \widehat{^*\mathbb{R}}$ (defined above) is given  by  the
formula $|\widehat{z}|=\widehat{|z|}$.

\item[(v)]  If $S\subseteq{\mathbb{C}}$, we shall often write simply $S$ instead of the more precise
$q_\mathcal{M}[S\cap\mathcal{M}]$ or
$\widehat{S}$. In particular, we shall often write simply $\mathbb{C}$ instead of 
$\widehat{\mathbb{C}}$ or $q_\mathcal{M}[\mathbb{C}\cap\mathcal{M}]$. Similarly, we shall often write
simply $\mathbb{R}$ instead of  $\widehat{\mathbb{R}}$ or 
$q_\mathcal{M}[\mathbb{R}\cap\mathcal{M}]$.
\item[(vi)] When we dealing not with one (fixed) but rather with two (or more than two) convex subrings, say
$\mathcal{M}_1$ and $\mathcal{M}_2$, we shall prefer the original notation introduced in Definition~\ref{D:
Asymptotic Fields} or the ``hat'' notation $\widehat{\mathcal{M}_1}$
and $\widehat{\mathcal{M}_2}$ instead of $\widehat{^*\mathbb{C}}$.
\end{description}
\end{notation}


\begin{theorem}[Asymptotic Fields]\label{T: Asymptotic Fields} Let $\mathcal{M}$ be a convex subring in
${^*\mathbb{C}}$ and  $\widehat{^*\mathbb{C}}=\mathcal{M}/\mathcal{M}_0$ be the corresponding
asymptotic field and let $\widehat{^*\mathbb{R}}=q_\mathcal{M}[^*\mathbb{R}\cap\mathcal{M}]$ be its
real part. Then:
\begin{description}
\item[(i)]  $\widehat{^*\mathbb{C}}$ is a field. Also, $\widehat{^*\mathbb{R}}$ is a totally ordered field
and we have $\widehat{^*\mathbb{C}}=\widehat{^*\mathbb{R}}(i)$.  

\item[(iv)] Either of $\widehat{^*\mathbb{C}}$ or $\widehat{^*\mathbb{R}}$  is an archimedean field \ifff
$\mathcal{M}=\mathcal{F}(^*\mathbb{C})$. 
\end{description}
\end{theorem}
\begin{proof}  (i) $\widehat{^*\mathbb{C}}$ is a field because $\mathcal{M}_0$ is a maximal ideal in
$\mathcal{M}$ by Lemma~\ref{L: Convex Ideals}. $\widehat{^*\mathbb{R}}$ is a real field because $^*\mathbb{R}\cap\mathcal{M}$ is a
real ring and $^*\mathbb{R}\cap\mathcal{M}_0$ is a convex maximal ideal in
$^*\mathbb{R}\cap\mathcal{M}$ by Lemma~\ref{L: Convex Ideals}. The connection
$\widehat{^*\mathbb{C}}=\widehat{^*\mathbb{R}}(i)$ follows from
$\mathcal{M}=(^*\mathbb{R}\cap\mathcal{M})(i)$ (Lemma~\ref{L: Convex Rings}). 

	(ii) Either $\widehat{^*\mathbb{C}}$ or $\widehat{^*\mathbb{C}}$ is an archimedean field
\ifff  $\mathcal{M}$ is an archimedean ring \ifff 
$\mathcal{M}=\mathcal{F}(^*\mathbb{C})$ by Lemma~\ref{L: Convex Rings}.
\end{proof}

	Our next goal is to show that every asymptotic field $\widehat{^*\mathbb{C}}$ is algebraically closed
field and its real part $\widehat{^*\mathbb{R}}$ is a real closed field. We start
with the following observation.

\begin{lemma}[Isomorphic Fields]\label{L: Isomorphic Fields}  Let
 $\mathcal{M}$ be (as before) a convex subring in ${^*\mathbb{C}}$. Let $\mathbb{K}$ be a field
which is a subring of $\mathcal{M}$ and let $\widehat{\mathbb{K}}=q_\mathcal{M}[\mathbb{K}]$ (\# 1 of
Definition~\ref{D: Asymptotic Fields}). Then 
\begin{description}
\item[(i)] The fields $\mathbb{K}$ and $\widehat{\mathbb{K}}$ are isomorphic under the mapping
$z\to\widehat{z}$ (or, equivalently, under the quotient mapping
$q_\mathcal{M}\, |\, \mathbb{K}$). In particular,
$\mathbb{C}$ and $\widehat{\mathbb{C}}$ are isomorphic and $\mathbb{R}$ and
$\widehat{\mathbb{R}}$ are isomorphic. 
\item[(ii)] The following are equivalent: 
$\mathcal{F}(^*\mathbb{C})\subseteq\mathbb{K}$ \ifff 
$\mathcal{I}(^*\mathbb{C})\subseteq\mathbb{K}$ \ifff
$\mathcal{L}(^*\mathbb{C})\subseteq\mathbb{K}$ \ifff $\mathbb{K}={^*\mathbb{C}}$. 

\item[(iii)] Similarly, the following are equivalent:  $\mathcal{F}(^*\mathbb{R})\subseteq\mathbb{K}$ \ifff 
$\mathcal{I}(^*\mathbb{R})\subseteq\mathbb{K}$ \ifff
$\mathcal{L}(^*\mathbb{R})\subseteq\mathbb{K}$ \ifff $\mathbb{K}={^*\mathbb{R}}$. 
\end{description}
\end{lemma}
\begin{proof}  (i) We observe that $\mathbb{K}\cap\mathcal{M}_0=\{0\}$. Indeed, suppose (on the
contrary) that
$z\in\mathbb{K}\cap\mathcal{M}_0$ for some
$z\not= 0$. It follows
$1/z\in\mathbb{K}$ (since $\mathbb{K}$ is a field) and also $1/z\notin\mathcal{M}$ (by the definition of
$\mathcal{M}_0$) which contradicts the assumption $\mathbb{K}\subseteq\mathcal{M}$. Consequently,
$\mathbb{K}$ and $\widehat{\mathbb{K}}$ are isomorphic. We leave the verification of (ii) and (iii) to the
reader.
\end{proof}

	The notation introduced in part (v) of Notation~\ref{N: Suppressing M} is justified by the following result.

\begin{corollary}[Embedding of Complex Numbers]\label{C: Embedding of Complex Numbers} The mapping
$\sigma: {\mathbb{C}}\to \widehat{^*\mathbb{C}}$, defined by
$\sigma(z)=\widehat{z}$, is a \textbf{field embedding} of
$\mathbb{C}$ into $\widehat{^*\mathbb{C}}$.
\end{corollary}
\begin{proof}  The result follows from the above lemma since $\mathbb{C}\subseteq\mathcal{M}$ by
Lemma~\ref{L: Convex Rings}. 
\end{proof}

\begin{definition}[Maximal Fields]\label{D: Maximal Fields}  Let $\mathcal{M}$ be a 
subring of $^*\mathbb{C}$ containing $\mathbb{C}$. A subfield
$\mathbb{M}$ of 
${^*\mathbb{C}}$ is called \textbf{maximal in $\mathcal{M}$} if:
(a) $\mathbb{M}$ is a subring of
$\mathcal{M}$ and $\mathbb{M}$ also contains a copy of $\mathbb{C}$, i.e.
$\mathbb{C}\subseteq\mathbb{M}\subseteq\mathcal{M}$; (b) There is no subfield
$\mathbb{K}$ of ${^*\mathbb{C}}$ which is a subring of $\mathcal{M}$ and which is a proper field
extension of $\mathbb{M}$. We denote by $\mathcal{M}ax(\mathcal{M})$ the {\bf set of all maximal fields} in
$\mathcal{M}$. 
\end{definition}

	For example, the field of the complex numbers
$\mathbb{C}$ is a maximal field in the ring of finite numbers $\mathcal{F}(^*\mathbb{C})$
(Example~\ref{Ex: Finite Numbers}).

\begin{lemma}[Existence of Maximal Fields]\label{L: Existence of Maximal Fields} Let 
$\mathbb{K}$ be subfield
of $^*\mathbb{C}$ such that $\mathbb{C}\subseteq\mathbb{K}\subseteq\mathcal{M}$. Then there exists a
maximal field
$\mathbb{M}\in\mathcal{M}ax(\mathcal{M})$ which is a field extension of $\mathbb{K}$. Consequently,
$\mathcal{M}ax(\mathcal{M})\not=\varnothing$. 
\end{lemma}

\begin{proof} Let  $\mathcal{L}_\mathbb{K}$  denote the
set of all subfields $\mathbb{L}$ of ${^*\mathbb{C}}$ such that
$\mathbb{K}\subseteq\mathbb{L}\subseteq\mathcal{M}$. We order
$\mathcal{L}_\mathbb{K}$ by inclusion. We obviously have
$\mathcal{L}_\mathbb{K}\not=\varnothing$, since $\mathbb{K}\in\mathcal{L}_\mathbb{K}$. Also, we
observe that if $S$ is a totally ordered subset of
$\mathcal{L}_\mathbb{K}$ under the inclusion $\subset$, then $\bigcup_{\mathbb{L}\in
S}\mathbb{L}\in\mathcal{L}_\mathbb{K}$. Thus
$\mathcal{L}_\mathbb{K}$ has a maximal element, say $\mathbb{M}$, as required, by Zorn's lemma.
Consequently, $\mathcal{M}ax(\mathcal{M})\not=\varnothing$ because $\mathbb{C}$ is a subring of
$\mathcal{M}$ by Lemma~\ref{L: Convex Rings}.
\end{proof}

\begin{theorem}[Field of Representatives]\label{T: Field of Representatives} Let $\mathcal{M}$ 
be (as before) a convex subring of ${^*\mathbb{C}}$ and $\mathbb{M}\in\mathcal{M}ax(\mathcal{M})$. 
Then: 
\begin{description}
\item[(i)]\; $\mathbb{M}$ and $\widehat{\mathbb{M}}$  are \textbf{algebraically closed
isomorphic fields}. 

\item[(ii)] Let $z_0 \in
\mathcal{M}$ be a point which is \textbf{away from} $\mathbb{M}$ in the sense that
\begin{equation}\label{E: Away From M}
(\forall z\in\mathbb{M})\left(z-z_0 \notin \mathcal{M}_0\right).
\end{equation}
Then $P(z_0) \notin
\mathcal{M}_0$ for  any non-zero polynomials $P$ with coefficients in $\mathbb{M}$. Consequently, the
field of the rational functions $\mathbb{M}(z_0)$ is a proper field extension of $\mathbb{M}$ within
$\mathcal{M}$, in symbol, $\mathbb{M}\subsetneqq\mathbb{M}(z_0)\subseteq\mathcal{M}$.
\item[(iii)]  We have the following \textbf{characterization} of $\mathcal{M}$ and $\mathcal{M}_0$ 
\begin{align} 
&\mathcal{M}=\{z\in{^*\mathbb{C}}\mid (\exists\varepsilon\in\mathbb{M}_+)(|z|\leq\varepsilon\},
\label{E: Characterization of M}\\ 
&\mathcal{M}_0=\{z\in{^*\mathbb{C}}\mid (\forall\varepsilon\in\mathbb{M}_+)(|z|<\varepsilon\},
\label{E: Characterization of M0}
\end{align}
where $\mathbb{M}_+=\{|z| :  z\in\mathbb{M}, z\not=0\}$.

\item[(iv)]  We have $\mathcal{M}=\mathbb{M}\oplus
\mathcal{M}_0$ in the sense that every  $z\in\mathcal{M}$ has a 
unique asymptotic expansion $z = c+dz$, where $c\in\mathbb{M}$\, and\,
$dz\in\mathcal{M}_0$. Consequently, $\mathbb{M}$ is a \textbf{field of representatives} for 
$\widehat{^*\mathbb{C}}$ in the sense that
$\widehat{^*\mathbb{C}}=\widehat{\mathbb{M}}$ or $\widehat{\mathcal{M}}=\widehat{\mathbb{M}}$
depending on the choice of the notation (Notation~\ref{N: Suppressing M}). 
\end{description}

\end{theorem}

\begin{proof} (i) We intend to show that $\mathbb{M}$ is algebraically closed. We denote by ${\rm
cl}(\mathbb{M})$ the relative algebraic closure of $\mathbb{M}$ in
${^*\mathbb{C}}$. Since ${^*\mathbb{C}}$ is an algebraically closed field, so is ${\rm cl}(\mathbb{M})$.
To show that ${\rm cl}(\mathbb{M})\subseteq\mathcal{M}$, suppose that $z \in {\rm
cl}(\mathbb{M})$. Since $z$ is algebraic over $\mathbb{M}$, it follows that
$z$ is a root of some polynomial $P(x)=x^n + a_1\,x^{n-1} + \dots + a_n $ with coefficients in
$\mathbb{M}$. The estimation $|z| \leq 1 + |a_1| +\dots+|a_n|$ implies $z \in\mathcal{M}$ by the
convexity of
$\mathcal{M}$. Now, $\mathbb{M} \subseteq {\rm cl}(\mathbb{M})\subseteq\mathcal{M}$ implies
$\mathbb{M} = {\rm cl}(\mathbb{M})$ by the
maximality of $\mathbb{M}$ in $\mathcal{M}$. The fields $\mathbb{M}$ and
$\widehat{\mathbb{M}}$ are isomorphic by Lemma~\ref{L: Isomorphic Fields}.

	(ii)  Suppose (on the contrary) that $P(z_0)\in\mathcal{M}_0$ for some polynomial
$P$. It follows that $\widehat{P(z_0)} = 0$ implying
$\widehat{P}(\widehat{z_0}) = 0$,  where $\widehat{P}$ denotes the polynomial, obtained from P by
replacing the coefficients $a_k$ in $P$ by $\widehat{a_k}$. Since
$\widehat{\mathbb{M}}$ is an algebraically  closed field, it follows that
$\widehat{z_0}\in \widehat{\mathbb{M}}$ meaning $z_0-z
\in\mathcal{M}_0$ for some 
$z\in\mathbb{M}$, a contradiction. 

	(iii) Let $z_0\in\mathcal{M}$ and suppose (on the contrary) that
$(\forall\varepsilon\in\mathbb{M}_+)(|z_0|>\varepsilon)$. Observe that $z_0$ is
away from 
$\mathbb{M}$ in the sense of (\ref{E: Away From M}).  Thus
$\mathbb{M}(z_0)$ is a proper field extension of
$\mathbb{M}$ within
$\mathcal{M}$ by (ii), contradicting the maximality of
$\mathbb{M}$. This proves the formula (\ref{E: Characterization of M}) about
$\mathcal{M}$ since the inclusion in the opposite direction follows from the convexity of $\mathcal{M}$.
Let
$z\in\mathcal{M}_0$. If $z=0$, there is nothing to prove. If $z\not= 0$, we have
$1/z\notin\mathcal{M}$ by the definition of $\mathcal{M}_0$.  Next, suppose (on the
contrary) that
$|z|\geq\varepsilon$ for some
$\varepsilon\in\mathbb{M}_+$. It follows that $|1/z|\leq1/\varepsilon$ implying
$1/z\in\mathcal{M}$ by formula (\ref{E:
Characterization of M}), a contradiction. Conversely, suppose that $|z|<\varepsilon$ for all
$\varepsilon\in\mathbb{M}_+$ and some $z\in{^*\mathbb{C}}$. It follows that
$1/\varepsilon<|1/z|$ for all $\varepsilon\in\mathbb{M}_+$ implying
$1/z\notin\mathcal{M}$ by the formula (\ref{E: Characterization of M}).
It follows $z\in\mathcal{M}_0$ (by part (i) of Lemma~\ref{L: Convex Ideals}),  which proves formula
(\ref{E: Characterization of M0}). 

	(iv) To show the existence of asymptotic expansion, suppose (on the contrary) that
there exists $z\in\mathcal{M}$ such that 
$z-c \notin\mathcal{M}_0$ for all $c\in\mathbb{M}$. We have
$\mathbb{M}\subsetneqq\mathbb{M}(z)\subseteq{M}$ by (ii), contradicting the maximality of
$\mathbb{M}$. To show the
uniqueness, suppose that $c+dz=c_1+dz_1$. It follows $c-c_1\in\mathcal{M}_0$ thus $c-c_1=0$ (as
required) since $\mathbb{M}\cap\mathcal{M}_0=\{0\}$.
\end{proof}

\begin{definition}[$\mathbb{M}$-Standard Part Mapping]\label{D: M-Standard Part Mapping} The mapping
$\st_\mathbb{M}: \mathcal{M}\to {^*\mathbb{C}}$, defined by $\st_\mathbb{M}(c+dz)=c$, is
called \textbf{$\mathbb{M}$-standard part mapping}.  
\end{definition}

\begin{lemma}[$\mathbb{M}$-Standard Part Mapping]\label{L: M-Standard Part Mapping} 
\begin{description}
\item[(i)] For every $z\in\mathcal{M}$
 we have $z=\st_\mathbb{M}(z)+dz$, 
where $\st_\mathbb{M}(z)\in\mathbb{M}$ and $dz\in\mathcal{M}_0$.
\item[(ii)] The $\mathbb{M}$-standard part
mapping
$\st_\mathbb{M}: \mathcal{M}\to {^*\mathbb{C}}$  is a ring homomorphism with
range
$\st_\mathbb{M}\left[\mathcal{M}\right]=\mathbb{M}$. 

\item[(iii)]  $\st_\mathbb{M}$
is an extension of the usual standard part mapping
$\st:\mathcal{F}(^*\mathbb{C})\to\mathbb{C}$, i.e.
$\st_\mathbb{M}\, |\, \mathcal{F}(^*\mathbb{C})=\st$. 
\end{description}
\end{lemma}
\begin{proof} (i) is a notational modification of the result of part (iv) of Theorem~\ref{T: Field of
Representatives}.

(ii) The fact that $\st_\mathbb{M}$ is homomorphism follows directly from the formula
$z=\st_\mathbb{M}(z)+dz$.

(iii) follows directly from the fact that $\mathbb{C}\subseteq\mathbb{M}$ and
$\mathcal{M}_0\subseteq\mathcal{I}(^*\mathbb{C})$.
\end{proof}

\begin{theorem}[Algebraically Closed Field]\label{T: Algebraically Closed Field} Let $\mathcal{M}$ be a
convex subring of $^*\mathbb{C}$ and let
$\widehat{^*\mathbb{C}}=\widehat{\mathcal{M}}$ be its
asymptotic field (Notation~\ref{N: Suppressing M}). Let $\mathbb{M}\in\mathcal{M}ax(\mathcal{M})$.
Then the fields
$\mathbb{M}$ and $\widehat{^*\mathbb{C}}=\widehat{\mathcal{M}}$ are isomorphic under 
the mapping $z\to\widehat{z}$ (or, equivalently, under the quotient mapping $q_\mathcal{M}\, |\,
\mathbb{M}$).
 The
situation can be summarized in the following commutative diagram:
\[
\begin{CD}
\mathcal{M}@>q_\mathcal{M}>>\widehat{^*\mathbb{C}}\\
@A{id}AA         @AA{id}A\\   
\mathbb{M}@>{q_\mathcal{M}|\mathbb{M}}>>\widehat{\mathbb{M}}.
\end{CD}
\]
Consequently, every asymptotic field is an
\textbf{algebraically closed field} and the real part of an asymptotic field is a \textbf{real closed field}. 
\end{theorem}
\begin{proof} The fields $\mathbb{M}$
and $\widehat{\mathbb{M}}$ are isomorphic by Lemma~\ref{L: Isomorphic Fields}. Also, 
$\widehat{^*\mathbb{C}}$ and $\widehat{\mathbb{M}}$ are isomorphic (under the identity) since 
$\widehat{^*\mathbb{C}}=\widehat{\mathbb{M}}$ by Theorem~\ref{T: Field of Representatives}. Thus
$\widehat{^*\mathbb{C}}$ is an algebraically closed field because $\mathbb{M}$ is an algebraically
closed field (Theorem~\ref{T: Field of Representatives}). The field $\widehat{^*\mathbb{R}}$ is a real
closed by Artin-Schreier theorem (Marker, Messmer, Pillay~\cite{MarkerMessPill}, p. 9),  since 
$\widehat{^*\mathbb{C}}=\widehat{^*\mathbb{R}}(i)$ by Theorem~\ref{T: Asymptotic Fields}.
\end{proof}
\begin{definition}[Convex Cover]\label{D: Convex Cover} Let $S$ be a subset of $^*\mathbb{C}$. The set
${\rm cov}(S)=\{z\in{^*\mathbb{C}} : |z|\leq | \zeta| \text{\, for some\, }
\zeta\in S \}$ is the \textbf{convex
cover} of $S$ in $^*\mathbb{C}$.
\end{definition}
\begin{lemma} \label{L: Convex Cover} Let $S$ be a subring of $^*\mathbb{C}$ which is closed
under the absolute value in the sense that $z\in S$ implies
$|z|\in S$. Then  ${\rm cov}(S)$ is a convex subring of
$^*\mathbb{C}$ (Definition~\ref{D: Convex Rings}). In particular, ${\rm cov}(\mathbb{K})$ is a convex
subring of
$^*\mathbb{C}$ for any algebraically closed subfield $\mathbb{K}$ of $^*\mathbb{C}$. Also ${\rm
cov}(\mathcal{M})=\mathcal{M}$ for any convex subring $\mathcal{M}$ of $^*\mathbb{C}$.
\end{lemma}
\begin{proof} We leave the details to the reader.
\end{proof}
\begin{theorem}[A Characterization]\label{T: A Characterization} Let $\mathbb{K}$ be a
subring of $^*\mathbb{C}$ which is closed under the absolute value. Then the following are equivalent:
\begin{description}
\item[(i)]\;  $\mathbb{K}$ is an asymptotic field (Definition~\ref{D: Asymptotic Fields}).
\item[(ii)] $\mathbb{K}$ is a maximal field in ${\rm cov}(\mathbb{K})$, in symbol, 
$\mathbb{K}\in\mathcal{M}ax({\rm cov}(\mathbb{K}))$ (Definition~\ref{D:
Maximal Fields}).
\end{description}
\end{theorem}
\begin{proof} (i)$\Rightarrow$(ii) Suppose that $\mathbb{K}$ is isomorphic to some asymptotic field
$\widehat{\mathcal{M}}$, where $\mathcal{M}$ is a convex subring of $^*\mathbb{C}$. Suppose (on the
contrary) that $\mathbb{K}\notin\mathcal{M}ax({\rm cov}(\mathbb{K}))$. It follows that
$\mathbb{K}\notin\mathcal{M}ax(\mathcal{M})$ because ${\rm
cov}(\mathbb{K})\subseteq\mathcal{M}$. Thus there exists a maximal field
$\mathbb{M}\in\mathcal{M}ax(\mathcal{M})$ which is a proper field extension of $\mathbb{K}$
(Lemma~\ref{L: Existence of Maximal Fields}). On the other hand, $\widehat{\mathcal{M}}$ and
$\mathbb{M}$ are isomorphic by Theorem~\ref{T: Algebraically Closed Field} thus
$\mathbb{K}$ and $\mathbb{M}$ must be isomorphic, a contradiction.

(i)$\Leftarrow$(ii) $\mathbb{K}\in\mathcal{M}ax({\rm cov}(\mathbb{K}))$ implies that the fields
$\mathbb{K}$ and $\widehat{{\rm cov}(\mathbb{K})}$ are isomorphic by Theorem~\ref{T: Algebraically
Closed Field}. Thus $\mathbb{K}$ is an asymptotic field since ${\rm cov}(\mathbb{K})$ is a convex subrings
of $^*\mathbb{C}$. 
\end{proof}
\begin{example} Let $\mathbb{C}(t)$ be the field of rational functions in one variable with complex
coefficients. Let $\rho$ be a positive infinitesimal in $^*\mathbb{R}$. Then $\mathbb{C}(\rho)$ is a subfield
of $^*\mathbb{C}$ which is closed under the absolute value. It is easy to see that
${\rm cov}(\mathbb{C}(\rho))=\mathcal{M}_\rho$ (Example~\ref{Ex: Robinson Rings}). Thus 
$\mathbb{C}(\rho)$ is not an asymptotic field because $\mathbb{C}(\rho)$ is not maximal in
$\mathcal{M}_\rho$. Indeed, we have 
$\mathbb{C}(\rho)\subsetneqq\mathbb{C}(\rho^\mathbb{Z})\subset
{^\rho\mathbb{C}}\subset\mathcal{M}_\rho$, where $\mathbb{C}(\rho^\mathbb{Z})$ stands for the field of
the Laurent series with complex coefficients. Similarly, the field of the Levi-Civita series
$\mathbb{C}\left<\rho\right>$ and the Hanh field $\mathbb{C}(\rho^\mathbb{R})$ are not asymptotic
(Section~\ref{S: Power Series in *C}).
\end{example}
\section{Embeddings in $^*\mathbb{C}$}\label{S: Embeddings in *C}
 In this section we show that the asymptotic fields can be treated as subfields of $^*\mathbb{C}$.

\begin{definition}[Embedding in $^*\mathbb{C}$]\label{D: Embedding in *C} Let $\mathcal{M}$ be a convex
subring of
$^*\mathbb{C}$ and let $\widehat{^*\mathbb{C}}=\widehat{\mathcal{M}}$ is the corresponding asymptotic
field (Definition~\ref{D: Asymptotic Fields}). For every $\mathbb{M}\in\mathcal{M}ax(\mathcal{M})$
(Definition~\ref{D: Maximal Fields}) we define the mapping
$\sigma_\mathbb{M}: \widehat{^*\mathbb{C}}\to{^*\mathbb{C}}$ by
$\sigma_\mathbb{M}(\widehat{z})=\st_\mathbb{M}(z)$ (or, equivalently, by
$\sigma_\mathbb{M}=(q_\mathcal{M}\, |\mathbb{M})^{-1}$), where $\st_\mathbb{M}$ is the
$\mathbb{M}$-standard part mapping (Definition~\ref{D: M-Standard Part Mapping}).

\end{definition}

\begin{theorem}[Embedding in $^*\mathbb{C}$]\label{T: Embedding in *C} The mapping
$\sigma_\mathbb{M}$ is a
\textbf{field embedding} of $\widehat{^*\mathbb{C}}$ into ${^*\mathbb{C}}$ with range
$\sigma_\mathbb{M}[\widehat{^*\mathbb{C}}]=\mathbb{M}$. We shall often write simply
$\widehat{^*\mathbb{C}}\subseteq{^*\mathbb{C}}$ (suppressing the dependence of  $\sigma_\mathbb{M}$
on the choice of $\mathbb{M}$). We summarize all these in
\begin{equation}\label{E: Inclusions of Mhat}
\mathbb{C}\subseteq\widehat{^*\mathbb{C}}\subseteq{^*\mathbb{C}} \text{\; or\; }
\mathbb{C}\subseteq\widehat{\mathcal{M}}\subseteq{^*\mathbb{C}},
\end{equation}
depending on the choice of the notation $\widehat{^*\mathbb{C}}$ or $\widehat{\mathcal{M}}$
(part~(ii) of Notation~\ref{N: Suppressing M}).
\end{theorem}

\begin{proof} The fields $\widehat{^*\mathbb{C}}$ and $\mathbb{M}$ are isomorphic by
Theorem~\ref{T: Algebraically Closed Field} and $\mathbb{M}$ is a subfield of $^*\mathbb{C}$.
\end{proof}

\begin{remark} According
to the above theorem, every maximal field $\mathbb{M}$ determines a unique field embedding
$\sigma_\mathbb{M}$. Conversely, every field embedding 
$\sigma_\mathbb{M}$ of $\widehat{^*\mathbb{C}}$ into $^*\mathbb{C}$ determines 
a maximal field $\mathbb{M}\subset\mathcal{M}$ by
$\sigma_\mathbb{M}[\, \widehat{^*\mathbb{C}}\, ]=\mathbb{M}$. On the ground of the isomorphism
between
$\mathbb{M}$ and $\widehat{^*\mathbb{C}}$ we shall sometimes identify $\mathbb{M}$ with
$\widehat{^*\mathbb{C}}$ by simply letting $\mathbb{M}=\widehat{^*\mathbb{C}}$. In this environment
$\st_\mathbb{M}$ reduces to the quotient
mapping $q_\mathcal{M}:\mathcal{M}\to\widehat{^*\mathbb{C}}$.  We should note that the embedding of 
$\widehat{\mathcal{M}}$ into $^*\mathbb{C}$ is neither unique, nor canonical because the existence of a
maximal field $\mathbb{M}$ depends on the axiom of choice (Lemma~\ref{L: Existence of Maximal Fields}).
\end{remark}
Every convex subring $\mathcal{M}$ of $^*\mathbb{C}$ determines an asymptotic field
$\widehat{\mathcal{M}}$ which we denoted by short by $\widehat{^*\mathbb{C}}$ (Notation~\ref{N:
Suppressing M}). In what follows we shall consider several subrings $\mathcal{M}_1, \mathcal{M}_2$, etc.
simultaneously and we shall prefer the more precise notations  $\widehat{\mathcal{M}_1},
\widehat{\mathcal{M}_2}$, etc. instead of
$\widehat{^*\mathbb{C}}$. Now, suppose that $\mathcal{M}_1\subset\mathcal{M}_2$. Our next goal is to show
that there exists a (non-unique, non-canonical) field embedding of $\widehat{\mathcal{M}_1}$ into
$\widehat{\mathcal{M}_2}$. We should note that the obvious candidate for a canonical embedding
$q_1(z)\to q_2(z)$ of $\widehat{\mathcal{M}_1}$ into
$\widehat{\mathcal{M}_2}$ is not defined correctly
due to the reverse inclusion of the ideals $\mathcal{M}_{2,0}\subset\mathcal{M}_{1,0}$ ($q_1$ and
$q_2$  stand for the corresponding quotient mappings).

\begin{lemma}[Synchronized Embeddings]\label{L: Synchronized Embeddings} Let $\mathcal{M}_1$ and
$\mathcal{M}_2$ be two convex subrings of $^*\mathbb{C}$ such that
$\mathcal{M}_1\subset\mathcal{M}_2$. Let $\mathcal{M}_{1,0}$ and
$\mathcal{M}_{2,0}$ be their maximal ideals and let
$\widehat{\mathcal{M}_1}=\mathcal{M}_1/\mathcal{M}_{1,0}$ and
$\widehat{\mathcal{M}_2}=\mathcal{M}_2/\mathcal{M}_{2,0}$ be the asymptotic fields generated by 
$\mathcal{M}_1$ and
$\mathcal{M}_2$, respectively (Definition~\ref{D: Asymptotic Fields}).  Then for every field embedding
$\sigma_1$ of $\widehat{\mathcal{M}_1}$ into $^*\mathbb{C}$ there exists a field embedding $\sigma_2$ of
$\widehat{\mathcal{M}_2}$ into  $^*\mathbb{C}$ such that
$\sigma_1[\widehat{\mathcal{M}_1}]\subset\sigma_2[\widehat{\mathcal{M}_2}]$. We say that the
\textbf{embeddings $\sigma_1$ and $\sigma_2$ are} \textbf{synchronized}. A similar result holds for every
finite many convex subrings $\mathcal{M}_1\subset\dots\subset\mathcal{M}_n$ of $^*\mathbb{C}$.
\end{lemma}
\begin{proof} Let $\mathbb{M}_1\in\mathcal{M}ax(\mathcal{M}_1)$ (chosen arbitrarily). There exists
$\mathbb{M}_2\in\mathcal{M}ax(\mathcal{M}_2)$ such that $\mathbb{M}_1\subset\mathbb{M}_2$ by
Lemma~\ref{L: Existence of Maximal Fields}, since $\mathbb{M}_1\subseteq\mathcal{M}_2$ by assumption.
We let $\sigma_1=\sigma_{\mathbb{M}_1}$ and $\sigma_2=\sigma_{\mathbb{M}_2}$ (Definition~\ref{D:
Embedding in *C}). We have
$\sigma_1[\widehat{\mathcal{M}_1}]=\mathbb{M}_1\subset\mathbb{M}_2=
\sigma_2[\widehat{\mathcal{M}_2}]$ by Theorem~\ref{T: Embedding in *C}.
\end{proof}
\begin{theorem}[Two Fields]\label{T: Two Fields} Let $\mathcal{M}_1$ and
$\mathcal{M}_2$ be two convex subrings of
$^*\mathbb{C}$ such that $\mathcal{M}_1\subseteq\mathcal{M}_2$. Then there exists (non-unique) a field
embedding $\sigma :
\widehat{\mathcal{M}_1}\to
\widehat{\mathcal{M}_2}$ of $\widehat{\mathcal{M}_1}$ into
$\widehat{\mathcal{M}_2}$.  We shall write all this simply as inclusions $
\widehat{\mathcal{M}_1}\subset\widehat{\mathcal{M}_2}$.
A similar result holds for every
finite many convex subrings $\mathcal{M}_1\subset\dots\subset\mathcal{M}_n$ of $^*\mathbb{C}$.
\end{theorem}
\begin{proof}
Let $\sigma_1$ and $\sigma_2$ be two synchronized
embeddings of $\widehat{\mathcal{M}_1}$ and $\widehat{\mathcal{M}_2}$ into $^*\mathbb{C}$, respectively,
in the sense of the above lemma. Then $\sigma : \widehat{\mathcal{M}_1}\to \widehat{\mathcal{M}_2}$, defined
by the formula $\sigma=\sigma_2^{-1}\circ\sigma_1$, is a field imbedding of $\widehat{\mathcal{M}_1}$ into
$\widehat{\mathcal{M}_2}$. 
\end{proof}

\begin{remark}[Non-Canonnical]\label{R: Non-Canonical}We should note that the field embedding $\sigma$ is
neither unique, nor canonical in the sense that it can not be determined uniquely in the terms of used in the
definitions of $\widehat{\mathcal{M}_1}$ and $\widehat{\mathcal{M}_2}$. Indeed, the embedding $\sigma$
depends on the choices of maximal fields the existence of which was proved with the help of the Zorn lemma
(Lemma~\ref{L: Existence of Maximal Fields}).
\end{remark}
\section{Examples of Asymptotic Fields}\label{S: Examples of Asymptotic Fields}
	In this section we present several examples of asymptotic fields closely related with the convex rings in
Section~\ref{S: Examples of Convex Rings}.

\begin{examples}[Examples of Asymptotic Fields]\label{Ex: Examples of Asymptotic Fields} Let $\rho$ be a
positive infinitesimal in $^*\mathbb{R}$ and let
$\mathcal{L}_\rho(^*\mathbb{C})$, $\mathcal{F}_\rho(^*\mathbb{C}), 
\mathcal{M}_\rho(^*\mathbb{C})$ and $\mathcal{E}_\rho(^*\mathbb{C})$ be the convex subrings of
$^*\mathbb{C}$ defined in Section~\ref{S: Examples of Convex Rings}. We denote by
\begin{align}\label{E: Examples}
&\widehat{\mathcal{F}}=\mathcal{F}/\mathcal{I}=\mathbb{C},
\notag\\
&\widehat{\mathcal{L}_\rho}=\mathcal{L}_\rho/\mathcal{L}_{\rho,0},\notag\\
&\widehat{\mathcal{F}_\rho}=\mathcal{F}_\rho/\mathcal{I}_\rho,
\notag\\
&\widehat{\mathcal{M}_\rho}=\mathcal{M}_\rho/\mathcal{N}_\rho={^\rho\mathbb{C}},\notag\\
&\widehat{\mathcal{E}_\rho}=\mathcal{E}_\rho/\mathcal{E}_{\rho,0},\notag\\
&\widehat{^*\mathbb{C}}={^*\mathbb{C}}/\{0\}={^*\mathbb{C}},\notag
\end{align}
the corresponding asymptotic fields. We denote by $\mathbb{R}, \Re e(\widehat{\mathcal{L}_\rho}), 
\Re e(\widehat{\mathcal{F}_\rho}), \Re e(^\rho\mathbb{C})={^\rho\mathbb{R}},\Re
e(\widehat{\mathcal{E}_\rho})$ and ${^*\mathbb{R}}$ their real parts,
respectively (Definition~\ref{D: Asymptotic Fields}). We call the field
$\widehat{\mathcal{L}_\rho}$
\textbf{multiple logarithmic field} because $\widehat{\log_\nu{(1/\rho)}}$ is a typical element of
$\widehat{\mathcal{L}_\rho}$, where 
$\nu$ is an infinitely large number in $^*\mathbb{N}$ (Example~\ref{Ex: Multiple-Logarithmic Rings}). We call 
the field $\widehat{\mathcal{F}_\rho}$ \textbf{logarithmic field} because $\widehat{\ln{\rho}}$ is 
a typical element of
$\widehat{\mathcal{F}_\rho}$ (Example~\ref{Ex: Logarithmic Rings}). We call $^\rho\mathbb{C}$
\textbf{Robinson field of (complex) asymptotic numbers} because $\widehat{\rho^m}$, where
$m\in\mathbb{Z}$, and more generally the series of the form
$\sum_{n=m}^\infty c_n\widehat{\rho}\,^n$, where 
$c_n\in\mathbb{C}$, are typical elements of $^\rho\mathbb{C}$ (see the remark
below). We call the field
$\widehat{\mathcal{E}_\rho}$
\textbf{logarithmic-exponential} field because $\widehat{e^{1/\rho}}$ and
$\widehat{\ln{\rho}}$ are typical element of  $\widehat{\mathcal{E}_\rho}$ (Example~\ref{Ex:
Logarithmic-Exponential Rings}).
\end{examples}
\begin{theorem}\label{T: The Examples}
\begin{description}
\item[(i)] Each of $\mathbb{C}, \widehat{\mathcal{L}_\rho}, \widehat{\mathcal{F}_\rho}, 
\widehat{\mathcal{M}_\rho}={^\rho\mathbb{C}}, \widehat{\mathcal{E}_\rho}$ and $^*\mathbb{C}$ is an
\textbf{algebraically closed subfield} of
$^*\mathbb{C}$. 
\item[(ii)] There exists an embeddings of each of these fields into $^*\mathbb{C}$ such that
\begin{equation}\label{E: ChainFields} \notag
\mathbb{C}\subset{\widehat{\mathcal{L}_\rho}}\subset
{\widehat{\mathcal{F}_\rho}}\subset\widehat{\mathcal{M}_\rho}\subset{\widehat{\mathcal{E}_\rho}}
\subset{^*\mathbb{C}}.
\end{equation}
\item[(iii)] Consequently, each of $\Re e({\widehat{\mathcal{L}_\rho}}), 
\Re e({\widehat{\mathcal{F}_\rho}}), \Re e(\widehat{\mathcal{M}_\rho})={^\rho\mathbb{R}}$ and $\Re
e({\widehat{\mathcal{E}_\rho}})$  is a
\textbf{real closed subfield} of $^*\mathbb{R}$ and 
\begin{equation}\label{E: ChainRealFields}\notag
\mathbb{R}\subset\Re e(\widehat{\mathcal{L}_\rho})\subset
\Re e(\widehat{\mathcal{F}_\rho})\subset\Re e(\widehat{\mathcal{M}_\rho})\subset\Re
e(\widehat{\mathcal{E}_\rho})\subset{^*\mathbb{R}}.
\end{equation}
\end{description}
\end{theorem}
\begin{proof} (i) The result follows directly from Theorem~\ref{T: Algebraically Closed Field} because all these
fields are asymptotic fields, i.e. fields of the form
$\widehat{\mathcal{M}}$ for some convex subring $\mathcal{M}$ of $^*\mathbb{C}$ (Section~\ref{S: Examples
of Convex Rings}). 

	(ii) follows from Theorem~\ref{T: Embedding in *C}, since
$\mathcal{F}(^*\mathbb{C})=\mathcal{F}\subset\mathcal{L}_\rho\subset\mathcal{F}_\rho\subset
\mathcal{M}_\rho\subset
\mathcal{E}_\rho\subset{^*\mathbb{C}}$ (see the end of Section~\ref{S: Examples of Convex
Rings}).

	(iii) follows directly from (i) and (ii).
\end{proof}
\begin{example}[A. Robinson's Asymptotic Numbers]\label{Ex: A. Robinson's Asymptotic Numbers}
 The field of the \textbf{real $\rho$-asymptotic numbers} ${^\rho\mathbb{R}}$
is introduced by  A. Robinson~\cite{aRob73}  and it is intimately connected with the asymptotic
expansions of standard functions (Lightstone\&Robinson~\cite{LiRob}). The fields
$^\rho\mathbb{R}$ and $^\rho\mathbb{C}$ are also known as \textbf{Robinson's valuation fields} because
it is endowed with a non-archimedean valuation $v:{^\rho\mathbb{C}}\to\mathbb{R}\cup\{\infty\}$ 
defined by 
\[
v(\widehat{z})=\sup\{r\in\mathbb{Q}\mid z/\rho^r\approx 0\}, \quad
\widehat{z}\not= 0, 
\]
and $v(0)=\infty$. We also have the following {\bf valuation formula} (due to A. Robinson):
$v(\widehat{z})={\rm st}\left(\ln{|z|}/\ln{\rho}\right)$ if
$z\in\mathcal{M}_\rho\setminus\mathcal{N}_\rho$
and $v(\widehat{z})=\infty$ if $z\in\mathcal{N}_\rho$. Notice that
$v(\widehat{\rho^x})=\st(x)$ for every finite number $x$ in
$^*\mathbb{R}$. The {\bf valuation metric} $d_v: {^\rho\mathbb{C}}\times
{^\rho\mathbb{C}}\to\mathbb{R}$ is defined by $d_v(\widehat{z},
\widehat{\zeta})=e^{-v(\widehat{z-\zeta})}$ under the convention that
$e^{-\infty}=0$. We also define the \textbf{valuation norm} $|\widehat{z}|_v=e^{-v(\widehat{z})}$.
We should note that the valuation topology and the order topology on
$^\rho\mathbb{C}$ are the same. Also, the series of the form $\sum_{n=m}^\infty
c_n\widehat{\rho}\,^n$ (mentioned earlier) are always convergent. For more recent results on
$^\rho\mathbb{R}$ we refer to (Todorov\&Wolf~\cite{TodWolf}). 
\end{example}
\begin{example}[Logarithmic-Exponential Power Series]\label{Ex: Logarithmic-Exponential Power Series}
The real part 
$\Re e(\widehat{\mathcal{E}_\rho})$ of $\widehat{\mathcal{E}_\rho}$ is a field extension of the
field
$\mathbb{R}((t))^{LE}$ of the {\em logarithmic-exponential power series} introduced in  L. Van den Dries, A. MacIntyre and D. Marker~ 
\cite{DriesMacintyreMarker}). We should mention that the field
$\mathbb{R}((t))^{LE}$ has several important applications including the solution of a Hardy's asymptotic
open problem. For more on the ordered exponential fields we refer to  S. Kuhlmann~\cite{KuhlKuhlShel},
where the reader will find more references on the subject.
\end{example}
\section{Cantor Completeness}

	We show that all asymptotic fields are Cantor complete and some are algebraically saturated. 

\begin{definition}[Cantor Completeness] \label{D: Cantor Completeness} Let $\mathbb{K}$ be an
ordered field (totally ordered field) and $\kappa$ be an infinite cardinal. 
\begin{description}
\item[(i)] We say that $\mathbb{K}$ is
\textbf{Cantor $\kappa$-complete} if every collection of fewer than $\kappa$ closed intervals in $\mathbb{K}$
with the finite intersection property (f.i.p.) has a non-empty intersection. We say that
$\mathbb{K}$ is simply {\bf Cantor complete} (or, $\mathbb{K}$ is a \textbf{semi-$\eta_1$-set}) if it is Cantor
$\aleph_1$-complete, where
$\aleph_1$ is the successor of $\aleph_0=\card(\mathbb{N})$. This means every nested sequence of
closed intervals in $\mathbb{K}$ has a non-empty intersection.
\item[(ii)]  We say that $\mathbb{K}(i)$ is Cantor $\kappa$-complete if $\mathbb{K}$ is
Cantor $\kappa$-complete.
\end{description}
\end{definition}

	 It is easy to show that a Cantor complete ordered field must be
sequentially complete. Two counter-examples to the converse are described
in the next section. The terminology  for semi-$\eta_1$-sets is introduced in 
Dales\&Woodin~\cite{DalWoodin} (see pp.7, 35, 50, 98). 
\begin{theorem}[Cantor Completeness] \label{T: Cantor Completeness} Let $^*\mathbb{C}$ be
$\kappa$-saturated for some infinite cardinal $\kappa$. Then every asymptotic subfield
of $^*\mathbb{C}$ is Cantor $\kappa$-complete. In particular, the asymptotic
fields ${\widehat{\mathcal{L}_\rho}}, {\widehat{\mathcal{F}_\rho}}, {^\rho\mathbb{C}}$ and
${\widehat{\mathcal{E}_\rho}}$ and their real parts
$\Re e({\widehat{\mathcal{L}_\rho}}), \Re e({\widehat{\mathcal{F}_\rho}}),
{^\rho\mathbb{R}}$ and $\Re e({\widehat{\mathcal{E}_\rho}})$ (Section~\ref{S:
Examples of Asymptotic Fields}) are all Cantor
$\kappa$-complete.  
\end{theorem}

\begin{proof}  Every asymptotic subfield of $^*\mathbb{C}$ is of the form $\widehat{\mathcal{M}}$ for some
convex subring $\mathcal{M}$ of $^*\mathbb{C}$ (Definition~\ref{D: Asymptotic Fields}). Thus
$\widehat{\mathcal{M}}=\mathbb{K}(i)$, where $\mathbb{K}=\widehat{^*\mathbb{R}\cap\mathcal{M}}$
is the real part of $\widehat{\mathcal{M}}$. Notice that $\widehat{^*\mathbb{R}\cap\mathcal{M}}$ is a totally
ordered field as a real closed field (Theorem~\ref{T: Algebraically Closed Field}). We have to show that
$\widehat{^*\mathbb{R}\cap\mathcal{M}}$ is Cantor $\kappa$-complete. Suppose that
$\{[a_\gamma, b_\gamma]\}_{\gamma\in
\Gamma}$ is a family of closed intervals in $\widehat{^*\mathbb{R}\cap\mathcal{M}}$ with the finite
intersection property and
$\card(\Gamma) <
\kappa$. We have $a_\gamma = q_\mathcal{M}(\alpha_\gamma)$ for some $\alpha_\gamma$
in $^*\mathbb{R}\cap\mathcal{M}$. Define $B_\gamma = \{\beta\in{^*\mathbb{R}\cap
q_\mathcal{M}^{-1}}(b_\gamma) : 
\alpha_\gamma\leq
\beta
\}$ and observe that $B_\gamma \neq \varnothing$ for each $\gamma \in \Gamma$.
Indeed, if $a_\gamma < b_\gamma$, then $B_\gamma={^*\mathbb{R}\cap q_\mathcal{M}^{-1}}(b_\gamma)$. If
$a_\gamma= b_\gamma$, then $\alpha_\gamma \in B_\gamma$. Thus (by the axiom of choice)
there exists a family $\{\beta_\gamma\}_{\gamma\in \Gamma}$ in $^*\mathbb{R}\cap\mathcal{M}$ such that
$\beta_\gamma \in B_\gamma$ for all $\gamma \in \Gamma$. As a
result, $\{[\alpha_\gamma, \beta_\gamma]\}_{\gamma\in \Gamma}$ is a family of closed
intervals in $^*\mathbb{R}$ with the finite
intersection property. It follows that there exists
$x\in{^*\mathbb{R}}$ such that $\alpha_\gamma \leq x \leq \beta_\gamma$ for all
$\gamma\in \Gamma$ by the $\kappa$-saturation of $^*\mathbb{R}$. It is clear that $x
\in \mathcal{M}$ by the convexity of $\mathcal{M}$; hence, $ a_\gamma \leq q_\mathcal{M}(x) \leq
b_\gamma$ for all $\gamma\in\Gamma$.
\end{proof}
\begin{definition}[Algebraic Saturation] \label{D: Algebraic Saturation} Let $\mathbb{K}$ be an
ordered field (totally ordered field) and $\kappa$ be an infinite cardinal. 
\begin{description}
\item[(i)] We say that $\mathbb{K}$ is
\textbf{algebraically $\kappa$-saturated} if every collection of fewer than $\kappa$ open intervals in
$\mathbb{K}$ with the finite intersection property has a non-empty intersection. We say that
$\mathbb{K}$ is simply {\bf algebraically saturated} if it is algebraically
$\aleph_1$-saturated. This means every nested sequence of
open intervals in $\mathbb{K}$ has a non-empty intersection.
\item[(ii)]  We say that $\mathbb{K}(i)$ is  $\kappa$-saturated if $\mathbb{K}$ is
algebraically $\kappa$-saturated.
\end{description}
\end{definition}
	The fields $^*\mathbb{R}$ and $^*\mathbb{C}$ are always algebraically $\kappa$-saturated for some infinite
cardinal $\kappa$, since the open intervals in $^*\mathbb{R}$ are internal sets. 

\begin{theorem}[Algebraically Saturated Fields]\label{T: Algebraically Saturated Fields} Let $^*\mathbb{C}$ be
$\kappa$-saturated for some infinite cardinal $\kappa$. Let $(\delta_n)$ be a decreasing generating 
sequence in
$^*\mathbb{R}$ and $\mathcal{M}$ be the convex subring of $^*\mathbb{C}$ generated by $(\delta_n)$
(part~(i) of Definition~\ref{D: Generating Sequences}). Then the asymptotic field $\widehat{\mathcal{M}}$ and its
real part
$\widehat{^*\mathbb{R}\cap\mathcal{M}}$ are algebraically $\kappa$-saturated. In particular, the asymptotic
fields
${\widehat{\mathcal{L}_\rho}}$ and ${\widehat{\mathcal{F}_\rho}}$ as well as their real parts
$\Re e({\widehat{\mathcal{L}_\rho}})$ and $\Re e({\widehat{\mathcal{F}_\rho}})$ (Section~\ref{S:
Examples of Asymptotic Fields}) are algebraically
$\kappa$-saturated.
\end{theorem}

\begin{proof} Suppose that $(a_\gamma, b_\gamma)_{\gamma\in
\Gamma}$ is a family of open intervals in $\widehat{^*\mathbb{R}\cap\mathcal{M}}$ with the finite intersection
property and
$\card(\Gamma) <
\kappa$. We have $a_\gamma = \widehat{\alpha_\gamma}$ and $b_\gamma = \widehat{\beta_\gamma}$ for
some $\alpha_\gamma$ and
$\beta_\gamma$ in $^*\mathbb{R}\cap\mathcal{M}$ such that $\alpha_\gamma<\beta_\gamma$ and
$\alpha_\gamma-\beta_\gamma\notin\mathcal{M}_0$ for all $\gamma\in\Gamma$. Thus the
family of open intervals $(\alpha_\gamma,
\beta_\gamma)_{\gamma\in
\Gamma}$ in $^*\mathbb{R}\cap\mathcal{M}$ has also the finite intersection
property. Next, we observe that family of open intervals 
$(\alpha_\gamma+1/\delta_n, \beta_\gamma-1/\delta_n)_{(\gamma, n)\in\Gamma\times\mathbb{N}}$ in
$^*\mathbb{R}\cap\mathcal{M}$ has also the finite intersection property, since $1/\delta_n\in\mathcal{M}_0$
for all $n\in\mathbb{N}$. Thus there exists $x\in{^*\mathbb{R}}$ such that $\alpha_\gamma+1/\delta_n<x<
\beta_\gamma-1/\delta_n$ for all $\gamma\in\Gamma$ and all $n\in\mathbb{N}$ by the $\kappa$-saturation
of $^*\mathbb{R}$ since $\card(\Gamma\times\mathbb{N})=\card(\Gamma)$. Also $x\in\mathcal{M}$ by
the convexity of $\mathcal{M}$. Next, we observe that
$x-\alpha_\gamma, x-\beta_\gamma\notin\mathcal{M}_0$ for all $\gamma\in\Gamma$ because for each (fixed)
$\gamma\in\Gamma$ we have $x-\alpha_\gamma>1/\delta_n$ and $\beta_\gamma-x>1/\delta_n$ for all
$n\in\mathbb{N}$. Thus 
$a_\gamma=\widehat{\alpha_\gamma}=\widehat{\alpha_\gamma+1/\delta_n}<
\widehat{x}<\widehat{\beta_\gamma-1/\delta_n}=\widehat{\beta_\gamma}=b_\gamma$ as required.
\end{proof}

\begin{remark}[Non-Saturated Fields]\label{R: Non-Saturated Fields} We should note  that not
every asymptotic field is algebraically saturated.  For example, the fields $\mathbb{R}$ and $\mathbb{C}$ are
certainly not saturated. Let
$(\lambda_n)$ be a increasing generating  sequence in
$^*\mathbb{R}$ and $\mathcal{M}$ be the convex subring of $^*\mathbb{C}$ generated by $(\lambda_n)$
(part~(ii) of Definition~\ref{D: Generating Sequences}). Then the asymptotic field $\widehat{\mathcal{M}}$ and
its real part $\widehat{^*\mathbb{R}\cap\mathcal{M}}$ are not algebraically saturated. In particular, the
asymptotic fields 
${^\rho\mathbb{C}}$ and ${\widehat{\mathcal{E}_\rho}}$ and their real parts
${^\rho\mathbb{R}}$ and $\Re e({\widehat{\mathcal{E}_\rho}})$ (Section~\ref{S:
Examples of Asymptotic Fields}) are not algebraically saturated. Indeed, we observe that the nested sequence
of open intervals
$(0, \widehat{1/\lambda_n})$ in $\widehat{\mathcal{M}}$ has an empty intersection. To show that,  suppose
(on the contrary) that there exists $x\in\mathcal{M}$ such that 
$0<\widehat{x}<\widehat{1/\lambda_n}$ for all $n\in\mathbb{N}$. It follows
$0< x< 1/\lambda_n$ for all $n\in\mathbb{N}$ implying $x\in\mathcal{M}_0$. Thus $\widehat{x}=0$, a
contradiction.
\end{remark}

\begin{theorem}[Hypothesis 1]\label{T: Characterization} Let $^*\mathbb{C}$ be
$\kappa$-saturated. Let $\mathbb{K}$ be a subring of $^*\mathbb{C}$ which is closed
under the absolute value in the sense that $z\in \mathbb{K}$ implies
$|z|\in \mathbb{K}$. Then the following are equivalent:
\begin{description}
\item[(i)] $\mathbb{K}$ is an asymptotic field (Definition~\ref{D: Asymptotic Fields}).
\item[(ii)] $\mathbb{K}$ is an algebraically closed Cantor $\kappa$-complete field.
\end{description}
\end{theorem}
\begin{proof} (i)$\Rightarrow$(ii) If $\mathbb{K}$ is an asymptotic field, then  $\mathbb{K}$ is algebraically
closed Cantor $\kappa$-complete by Theorem~\ref{T: Algebraically Closed Field} and Theorem~\ref{T: Cantor
Completeness}. 

(i)$\Leftarrow$(ii) We let $\mathcal{M}={\rm
cov}(\mathbb{K})$ (Definition~\ref{D: Convex Cover}) and observe that $\mathcal{M}$ is a convex subring of
$^*\mathbb{C}$ by Lemma~\ref{L: Convex Cover}.  In view of Theorem~\ref{T: Algebraically Closed Field}, to
show that
$\mathbb{K}$ is an asymptotic field, it suffices to show that $\mathbb{K}$ is a maximal field in
$\mathcal{M}$. 

\end{proof}
\newpage
\section{Power Series in $^*\mathbb{C}$}\label{S: Power Series in *C}
In this section we show that different fields of generalized power series can be embedded as subfields of
$^*\mathbb{C}$. We start with some preliminaries.
	
\begin{enumerate}
\item Let $\mathbb{K}$ be a field. We denote (as usual) by $\mathbb{K}[t]$ the ring of polynomials in
one  variable with coefficients in $\mathbb{K}$. If $\mathbb{K}$ is ordered, we supply
$\mathbb{K}[t]$ with the ordering in which a polynomial is positive if the coefficient in front of 
the least power of $t$ is positive. We denote by
$\mathbb{K}(t)$ the field of rational functions in one  variable with coefficients in $\mathbb{K}$. If
$\mathbb{K}$ is ordered, then $\mathbb{K}(t)$ is an ordered field under the ordering inherited from 
$\mathbb{K}[t]$. In this ordering every rational function of the form $t^n$ for some $n\in\mathbb{N}$,
is between 0 and every positive element of
$\mathbb{K}$. The field $\mathbb{K}$ is naturally embedded in $\mathbb{K}(t)$ by mapping $a\to
at^0$. 

\item We denote by $\mathbb{K}(t^\mathbb{Z})$ the field of
\textbf{Laurent series} with coefficients in $\mathbb{K}$. If $\mathbb{K}$ is ordered, then
$\mathbb{K}(t^\mathbb{Z})$ is an ordered field under the ordering in which a series is positive if its
leading coefficient is positive. The field $\mathbb{K}(t)$ has a
natural embedding into $\mathbb{K}(t^\mathbb{Z})$ by $f\to L(f)$, where $L(f)$ is the Laurent
expansion of $f$. 

	\item Let  $G$ an ordered abelian group.  For any formal power series $f = \sum_{g \in G} {{a_g} t^g}$,
where each 
$a_g\in\mathbb{K}$, the \textbf{support} of
$f$ is defined by $\supp(f)=\{g \in G : a_g \not= 0\}$. Recall that the \textbf{Hahn field}
$\mathbb{K}(t^G)$ or
$\mathbb{K}((G))$ is the set of all such $f$'s whose support $\supp(f)$ is a well-ordered set
(Hahn~\cite{hHahn}). We supply
$\mathbb{K}(t^G)$  with the ordinary
polynomial-like addition and multiplication. The field  
$\mathbb{K}(t^G)$ has a canonical $G$-valued Krull valuation in which each non-zero power series is
mapped to the least exponent in its support (Krull~\cite{wKrull}). If $\mathbb{K}$ is ordered, then
$\mathbb{K}(t^G)$ has a natural ordering in which a series is positive if the coefficient
corresponding to the least element in its support is positive. This ordering is compatible with the
canonical valuation, and is the unique ordering on $\mathbb{K}(t^G)$ in which every positive power of
$t$ is between 0 and every positive element of $\mathbb{K}$. Notice that every ordered abelian
group contains a copy of $\mathbb{Z}$. Thus
$\mathbb{K}(t^\mathbb{Z})\subset\mathbb{K}(t^\mathbb{R})$. Summarizing, we have
\[
\mathbb{K}\subset\mathbb{K}(t)\subset\mathbb{K}(t^\mathbb{Z})
\subset\mathbb{K}(t^G).
\]
The field
$\mathbb{K}(t^\mathbb{R})$ ($\mathbb{K}(t^G)$ is an algebraically closed or real-closed
valuation field whenever
$\mathbb{K}$ is algebraically closed or real-closed, respectively. 

	\item Let $G=(\mathbb{R}, +, <)$ be the abelian group of $\mathbb{R}$ with the usual addition and
order. In addition to the Hahn field $\mathbb{K}(t^\mathbb{R})$, we also consider the field
$\mathbb{K}\left<t^\mathbb{R}\right>$ of  \textbf{Levi-Civita's series} with coefficients in
$\mathbb{K}$; the field $\mathbb{K}\left<t^\mathbb{R}\right>$ consists of the series of the form
$\sum_{n=0}^\infty c_nt^{\nu_n}$,  where $c_n\in\mathbb{K}$, which are either finite sums, or
$(c_n)$ is a sequence in $\mathbb{K}$, such that $c_n \neq 0$ for all $n$,  and $(\nu_n)$ is a strictly
increasing unbounded sequence in $\mathbb{R}$.  It is clear that
$\mathbb{K}\left<t^\mathbb{R}\right>
\subset\mathbb{K}(t^\mathbb{R})$. Thus we have 
\[
\mathbb{K}\subset\mathbb{K}(t)\subset\mathbb{K}(t^\mathbb{Z})\subset\mathbb{K}\left<t^\mathbb{R}\right>
\subset\mathbb{K}(t^\mathbb{R}).
\] 
The field $\mathbb{K}\left<t^\mathbb{R}\right>$ is an algebraically closed or real-closed valuation
field whenever $\mathbb{K}$ is algebraically closed or real-closed, respectively. The field $\mathbb{R}\left<t^\mathbb{R}\right>$ was introduced by
Levi-Civita in \cite{tLC} and later was investigated by Laugwitz in \cite{dLaug} as a potential framework
for the rigorous foundation of infinitesimal calculus before the advent of Robinson's nonstandard
analysis. 

\item From Krull~\cite{wKrull} and Theorem 2.12 in Luxemburg \cite{wLux}, it is
known that every Hahn field of the form $\mathbb{K}(t^\mathbb{R})$ is spherically complete in its
canonical valuation. In particular,
$\mathbb{Q}(t^{\mathbb{R}})$ is spherically complete, hence, sequentially complete. But
$\mathbb{Q}(t^{\mathbb{R}})$ is not Cantor complete (for the same reason that $\mathbb{Q}$ is not
Cantor complete).  Also, $\mathbb{K}\left<t^\mathbb{R}\right>$ is sequentially complete but not spherically complete
(Pestov~\cite{vPes}, pp. 67). 

\item The field $\mathbb{C}\left<t^\mathbb{R}\right>$ is embedded as a subfield of
$^\rho\mathbb{C}$ by the mapping $\sum_{n=0}^\infty c_nt^{\nu_n}\to\sum_{n=0}^\infty
c_n\rho^{\nu_n}$ (Robinson~\cite{aRob73}). The same formula defines a field embedding $M_\rho$ of
the Levi-Civita field $\widehat{\mathcal{F}_\rho}\left<t^\mathbb{R}\right>$ into the ring
$\mathcal{M}_\rho$ and thus into the field
$^\rho\mathbb{C}$ (Todorov and Wolf~\cite{TodWolf}, Section 5), where
$\widehat{\mathcal{F}_\rho}$ is the logarithmic field defined in Section~\ref{S: Examples of Asymptotic
Fields}. (We should note that in \cite{TodWolf} the real part $\Re e(\widehat{\mathcal{F}_\rho})$ of the
field
$\widehat{\mathcal{F}_\rho}$ is denoted by 
$\widehat{^\rho\mathbb{R}}$.) We sometimes write simply $t\to\rho$ instead of the more
precise $M_\rho$. Finally, the embedding
$M_\rho$ is extended to a field isomorphism between the Hahn field
$\widehat{\mathcal{F}_\rho}(t^\mathbb{R})$ and
$^\rho\mathbb{C}$ (Todorov and Wolf~\cite{TodWolf}, Section 6). We shall write this isomorphism
simply as equality
$\widehat{\mathcal{F}_\rho}(\rho^\mathbb{R})={^\rho\mathbb{C}}$ and summarize all these in
\begin{equation}\label{E: RobTodWolf}\notag
\widehat{\mathcal{F}_\rho}(\rho)\subset\widehat{\mathcal{F}_\rho}(\rho^\mathbb{Z})\subset
\widehat{\mathcal{F}_\rho}\left<\rho^\mathbb{R}\right>\subset\mathcal{M}_\rho
\subset\widehat{\mathcal{F}_\rho}(\rho^\mathbb{R})={^\rho\mathbb{C}}.
\end{equation} 
Also, $\mathbb{C}(\rho^\mathbb{R})\subset\widehat{\mathcal{F}_\rho}(\rho^\mathbb{R})$  (trivially) since
$\mathbb{C}\subset\widehat{\mathcal{F}_\rho}$.  Thus we have also 
\begin{equation}\label{E: RobTodWolf}\notag
\mathbb{C}\subset\mathbb{C}(\rho)\subset\mathbb{C}(\rho^\mathbb{Z})\subset
\mathbb{C}\left<\rho^\mathbb{R}\right>\subset\mathbb{C}(\rho^\mathbb{R})\subset\mathcal{M}_\rho
\subset{^\rho\mathbb{C}}.
\end{equation} 
\end{enumerate}

\begin{theorem}[Power Series in $^*\mathbb{C}$]\label{T: Power Series in *C} There exists a field
embedding of the Hahn field $\widehat{\mathcal{F}_\rho}(\rho^\mathbb{R})$ into $^*\mathbb{C}$.
Consequently, we have the embeddings
\begin{align}
&\widehat{\mathcal{F}_\rho}(\rho)\subset\widehat{\mathcal{F}_\rho}(\rho^\mathbb{Z})\subset
\widehat{\mathcal{F}_\rho}\left<\rho^\mathbb{R}\right>\subset
\widehat{\mathcal{F}_\rho}\left<\rho^\mathbb{R}\right>
\subset\mathcal{M}_\rho
\subset\widehat{\mathcal{F}_\rho}(\rho^\mathbb{R})\subset{^*\mathbb{C}},\notag\\
&\Re e(\widehat{\mathcal{F}_\rho})(\rho)\subset\Re
e(\widehat{\mathcal{F}_\rho})(\rho^\mathbb{Z})\subset
\Re e(\widehat{\mathcal{F}_\rho})\left<\rho^\mathbb{R}\right>\subset\Re e(\mathcal{M}_\rho)
\subset\Re e(\widehat{\mathcal{F}_\rho})(\rho^\mathbb{R})\subset{^*\mathbb{R}},\notag\\
&\mathbb{C}(\rho)\subset\mathbb{C}(\rho^\mathbb{Z})\subset
\mathbb{C}\left<\rho^\mathbb{R}\right>\subset\mathbb{C}(\rho^\mathbb{R})
\subset\mathcal{M}_\rho
\subset{^*\mathbb{C}},\notag\\
&\mathbb{R}(\rho)\subset\mathbb{R}(\rho^\mathbb{Z})\subset
\mathbb{R}\left<\rho^\mathbb{R}\right>\subset\mathbb{R}(\rho^\mathbb{R})
\subset\Re e(\mathcal{M}_\rho)
\subset{^*\mathbb{R}}.\notag
\end{align}
\end{theorem}
\begin{proof} There exists a field embedding of $^\rho\mathbb{C}$ into $^*\mathbb{C}$ by
Theorem~\ref{T: The Examples}. Thus there exists a field embedding of
$\widehat{\mathcal{F}_\rho}(\rho^\mathbb{R})$ into $^*\mathbb{C}$ because
$\widehat{\mathcal{F}_\rho}(\rho^\mathbb{R})$ and $^\rho\mathbb{C}$ are field isomorphic 
(Todorov\&Wolf~\cite{TodWolf}, Section 6). Also,
$\mathbb{C}(\rho^\mathbb{R})\subset\widehat{\mathcal{F}_\rho}(\rho^\mathbb{R})$  (trivially) since
$\mathbb{C}\subset\widehat{\mathcal{F}_\rho}$.
\end{proof}

	We should note that the embedding of $\widehat{\mathcal{F}_\rho}(\rho^\mathbb{R})$ into
$^*\mathbb{C}$ is neither canonical, nor unique because the embedding of $^\rho\mathbb{C}$ into
$^*\mathbb{C}$ is neither canonical, nor unique (Remark~\ref{R: Non-Canonical}).
\section{Asymptotic Vectors}\label{S: Asymptotic Vectors}

Let $\mathcal{M}$ be a convex subring of $^*\mathbb{C}$ (Definition~\ref{D: Convex Rings}). Let
$\widehat{\mathcal{M}}$ be the associated asymptotic field and $\Re e(\widehat{\mathcal{M}})$ 
be its real part (Definition~\ref{D: Asymptotic Fields}).
Recall that $\Re e(\widehat{\mathcal{M}})$ can be denoted equivalently by $\widehat{^*\mathbb{R}}$
(see (\ref{E: Hat Notation}) in Notation~\ref{N: Suppressing M}). We also let
$\widehat{^*\mathbb{R}}^d=\widehat{^*\mathbb{R}}\times\widehat{^*\mathbb{R}}\times\dots\times
\widehat{^*\mathbb{R}}$ ($d$ times). 

	In this section we present $\widehat{^*\mathbb{R}}^d$ as a factor space in $^*\mathbb{R}^d$. We also discuss
the concept of monad in $\widehat{^*\mathbb{R}}^d$.

\begin{definition}[Asymptotic Vectors]\label{D: Asymptotic Vectors} Let $\mathcal{M}$ be a
convex subring of $^*\mathbb{C}$. 
\begin{enumerate}
\item We define the linear spaces
\begin{align}
&\mathcal{M}^d(^*\mathbb{R}^d)=\left\{\mathbf{x}\in{^*\mathbb{R}^d}: ||\mathbf{x}||\in\mathcal{M}\right\},
\notag\\ 
&\mathcal{M}^d_0(^*\mathbb{R}^d)=\left\{\mathbf{x}\in{^*\mathbb{R}^d}:
||\mathbf{x}||\in\mathcal{M}_0\right\},\notag
\end{align}
where $||\cdot||$ stands for the usual Euclidean norm in $^*\mathbb{R}^d$. We define the factor vector space
$\widehat{\mathcal{M}^d(^*\mathbb{R}^d)}=\mathcal{M}^d(^*\mathbb{R}^d)/\mathcal{M}^d_0(^*\mathbb{R}^d)$
and denote by $q_\mathcal{M}^d:
\mathcal{M}^d(^*\mathbb{R}^d)\to\widehat{\mathcal{M}^d(^*\mathbb{R}^d)}$ the quotient mapping. We shall
often write $\widehat{\mathbf{x}}$ instead of  $q_\mathcal{M}^d(\mathbf{x})$.
\item We supplied $\widehat{\mathcal{M}^d(^*\mathbb{R}^d)}$ by the addition inherited from
$\mathcal{M}^d(^*\mathbb{R}^d)$. We define multiplication between an asymptotic number
$\widehat{\lambda}\in \widehat{^*\mathbb{R}}$ and an asymptotic vector 
$\widehat{\mathbf{x}}\in\widehat{\mathcal{M}^d(^*\mathbb{R}^d)}$ by the
formula $\widehat{\lambda}\; \widehat{\mathbf{x}}=\widehat{\lambda\mathbf{x}}$. 
\item Let $S\subseteq{^*\mathbb{R}^d}$. We shall sometimes write simply $\widehat{S}$ instead of the more
precise $q_\mathcal{M}^d[S\cap\mathcal{M}^d(^*\mathbb{R}^d)]$ (suppressing the dependence on
$\mathcal{M}$). In this simplified notation we have
$\widehat{\mathcal{M}^d(^*\mathbb{R}^d)}=\widehat{^*\mathbb{R}^d}$ (cf. Notation~\ref{N: Suppressing M}).
\end{enumerate}
\end{definition}
\begin{lemma}[Finite Points] \label{L: Finite Points} Let $\mathcal{M}$ be a convex subring of $^*\mathbb{C}$
and let
$\mathcal{F}(^*\mathbb{R}^d)$ denote the set of the finite points in $^*\mathbb{R}^d$. Then 
$\mathcal{F}(^*\mathbb{R}^d)\subseteq\mathcal{M}^d(^*\mathbb{R}^d)$. Consequently,
$\mathbb{R}^d\subseteq\mathcal{M}^d(^*\mathbb{R}^d)$.
\end{lemma}
\begin{proof} The result follows immediately from part (i) of Lemma~\ref{L: Convex Rings}. 
\end{proof}

	To the end of this section we shall use $\widehat{^*\mathbb{R}}$, $\widehat{^*\mathbb{R}}^d,
\widehat{^*\mathbb{R}^d}$ and $\widehat{S}$ instead of the more precise $\Re e(\widehat{\mathcal{M}})$,
$(\Re e(\widehat{\mathcal{M}}))^d$, $\widehat{\mathcal{M}^d(^*\mathbb{R}^d)}$ and 
$q_\mathcal{M}^d[S\cap\mathcal{M}^d(^*\mathbb{R}^d)]$, respectively, suppressing the dependence on
$\mathcal{M}$ (cf. Notation~\ref{N: Suppressing M}).

\begin{theorem}\label{T: Vector Spaces} Let $\mathcal{M}$ be a convex subring of $^*\mathbb{C}$. Then:
\begin{description}
\item[(i)] $\widehat{^*\mathbb{R}^d}$ is a vector space over the field $\widehat{^*\mathbb{R}}$ (\ref{E: Hat
Notation}). Also 
$\widehat{^*\mathbb{R}^d}$ and
${\widehat{^*\mathbb{R}}^d}$ are isomorphic vector spaces under the mapping $\widehat{\mathbf{x}}\to
(\widehat{x_1}, \widehat{x_2},\dots, \widehat{x_d})$, where $\mathbf{x}=(x_1, x_2,\dots,
x_d)\in\mathcal{M}^d(^*\mathbb{R}^d)$. 
\item[(ii)]$\mathbb{R}^d$ is a linear subspace of
$\widehat{^*\mathbb{R}^d}$ over the field $\mathbb{R}$.
\end{description}
\end{theorem}
\begin{proof} The verification is straightforward and we leave it to the reader.
\end{proof}
\begin{definition}[Monads in $\widehat{^*\mathbb{R}^d}$]\label{D: Monads in *Rdhat} Let
$\Omega\subseteq\mathbb{R}^d$ and let $\mu(\Omega)\subset{^*\mathbb{R}^d}$ be the monad of $\Omega$
in
$^*\mathbb{R}^d$, i.e. 
\begin{equation}\label{E: Monad in *Rd}
\mu(\Omega)=\{\mathbf{r}+\mathbf{h}: \mathbf{r}\in\Omega,\;
\mathbf{h}\in{^*\mathbb{R}^d}, ||\mathbf{h}||\approx 0\}.
\end{equation}
Let $\mathcal{M}$ be a convex subring of
$^*\mathbb{C}$. We define the \textbf{monad of
$\Omega$ in $\widehat{^*\mathbb{R}^d}$} by the formula 
$\widehat{\mu(\Omega)}=q_\mathcal{M}^d[\mu(\Omega)]$.
\end{definition}

	Note that $\mu(\Omega)\subseteq\mathcal{M}(^*\mathbb{R}^d)$ by Lemma~\ref{L: Finite
Points}, since $\mu(\Omega)\subseteq\mathcal{F}(^*\mathbb{R}^d)$, which guarantees the correctness of the
above definition.

\begin{lemma} Let $\Omega\subseteq\mathbb{R}^d$ and $\mathcal{M}$ be a convex subring of
$^*\mathbb{C}$. Then:
\begin{equation}\label{E: Monad in *Rdhat}
\widehat{\mu(\Omega)}=\{\mathbf{r}+\widehat{\mathbf{h}}: \mathbf{r}\in\Omega,\;
\widehat{\mathbf{h}}\in\widehat{{^*\mathbb{R}^d}}, ||\widehat{\mathbf{h}}||\approx 0\}.
\end{equation}
\end{lemma}
\begin{proof} The verification is straightforward and we leave it to the reader.
\end{proof}

	We shall use $\widehat{\mu(\Omega)}$ mostly in the case when $\Omega$ is an open subset of
$\mathbb{R}^d$.
\section{$\mathcal{M}$-Asymptotic Functions}\label{S: M-Asymptotic Functions}

In this section we describe a variety of differential rings
$\widehat{\mathcal{M}}(\Omega)$ of generalized functions on an open set 
$\Omega$ of $\mathbb{R}^d$ in terms of a given convex subring $\mathcal{M}$ of $^*\mathbb{C}$
(Section~\ref{S: Asymptotic Fields: Definitions and Examples}). The elements of
$\widehat{\mathcal{M}}(\Omega)$ are named {\bf $\mathcal{M}$-asymptotic functions}
because  their values are in the field 
$\widehat{\mathcal{M}}$ of the $\mathcal{M}$-asymptotic numbers and
because, more importantly, each $\widehat{\mathcal{M}}(\Omega)$ is an algebra over
the field
$\widehat{\mathcal{M}}$ (Section~\ref{S: Asymptotic Fields: Definitions and Examples}). We
intend to convert some of $\widehat{\mathcal{M}}(\Omega)$ into {\bf algebras of
Colombeau's type} by supplying 
$\widehat{\mathcal{M}}(\Omega)$ with a copy of the space of Schwartz distributions
$\mathcal{D}^\prime(\Omega)$ in one of the next sections. In this section we generalize some 
of the results in (Oberguggenberger~and~T. Todorov~\cite{OberTod98}), where the algebra of
$\rho$-asymptotic functions $^\rho\mathcal{E}(\Omega)$ is introduced; within our more general
theory  the algebra $^\rho\mathcal{E}(\Omega)$ appears as a particular
example (Example~\ref{Ex: rho-Asymptotic Functions}). Similar to some of our results
appear in the H. Vernaeve Ph.D. Thesis~\cite{hVernaevePhD} (for comparison see the definition of
$\mathcal{E}_M(\Omega)$ on p. 90, Sec. 3.6).

	Here is the {\bf summary} of the basic definitions. The justification of the definitions will be presented
later in this section and some of the results will be worked out in detail in some of the next sections.

	In what follows $^*\mathbb{C}$ stands for a non-standard extension of the field of the
complex numbers $\mathbb{C}$. Let  $\mathcal{M}$ be a convex subring in
$^*\mathbb{C}$, $\mathcal{M}_0$ be the ideal of the non-invertible elements of 
$\mathcal{M}$. Let $\widehat{\mathcal{M}}$ be the field of
$\mathcal{M}$-asymptotic numbers (Section~\ref{S: Asymptotic Fields: Definitions and Examples}).  
Let $\Omega$ be an open set of $\mathbb{R}^d$ and let $\mu(\Omega)$ be the monad of $\Omega$ and
$\mu_\mathcal{M}(\Omega)$ denote the $\mathcal{M}$-monad of $\Omega$  (\ref{E:
F-Monad}).  
\begin{definition}[$\mathcal{M}$-Asymptotic Functions]\label{D: M-Asymptotic Functions} Then
\begin{enumerate}  
 \item  We define
the set of {\bf
$\mathcal{M}$-moderate functions}
$\mathcal{M}(\Omega)$ and the set of the {\bf $\mathcal{M}$-negligible functions}
in $^*\mathcal{E}(\Omega)$ by
\begin{align}
&\mathcal{M}(\Omega)=\{f\in{^*\mathcal{E}(\Omega)}\mid
(\forall\alpha\in\mathbb{N}_0^d)(\forall x\in\mu(\Omega)(\partial^\alpha
f(x)\in\mathcal{M})\},\notag\\
&\mathcal{M}_0(\Omega)=\{f\in{^*\mathcal{E}(\Omega)}\mid
(\forall\alpha\in\mathbb{N}_0^d)(\forall x\in\mu(\Omega)(\partial^\alpha
f(x)\in\mathcal{M}_0)\},\notag
\end{align}
respectively. Let  $\widehat{\mathcal{M}}(\Omega)=
\mathcal{M}(\Omega)/\mathcal{M}_0(\Omega)$ be the
corresponding factor ring. We say that $\widehat{\mathcal{M}}(\Omega)$ {\bf is 
generated by} $\mathcal{M}$. The elements of
$\widehat{\mathcal{M}}(\Omega)$ are named {\bf $\mathcal{M}$-asymptotic
functions on} $\Omega$. We denote by $Q_\Omega:
\mathcal{M}(\Omega)\to \widehat{\mathcal{M}}(\Omega)$ the corresponding
quotient mapping. However we shall often $\widehat{f}$ instead of $Q_\Omega(f)$ for the
equivalence class of $f\in\mathcal{M}(\Omega)$.

\item If $S\subseteq{^*\mathcal{E}}(\Omega)$, we shall sometimes suppress the dependence on
$\mathcal{M}$ and write simply 
$\widehat{S}$ instead of the more preices $Q_\Omega[S\cap\mathcal{M}]$. We observe that
$\widehat{\mathcal{M}}(\Omega)=\widehat{\mathcal{M}(\Omega)}=\widehat{^*\mathcal{E}(\Omega)}$,
where in the latter notation, $\widehat{^*\mathcal{E}(\Omega)}$, the dependence on $\mathcal{M}$
has been suppressed.
\item We define the {\bf embedding}
$\mathcal{E}(\Omega)\embed\widehat{\mathcal{M}}(\Omega)$, by the mapping
$f\to\widehat{^*f}$,  where $^*f$ is the {\em non-standard extension} of $f$. 

\item We define a {\bf pairing} between $\widehat{\mathcal{M}}(\Omega)$ and space of
test-functions $\mathcal{D}(\Omega)$ by the formula
\begin{equation}\label{E: Pairing}
\bra\widehat{f},\; \tau\ket= q_\mathcal{M}\left(\int_{^*\Omega}f(x)\, {^*\tau}(x)\,
dx\right),
\end{equation}
where $\widehat{f}\in\widehat{\mathcal{M}}(\Omega)$ and
$\tau\in\mathcal{D}(\Omega)$ and $q_\mathcal{M}: \mathcal{M}\to\widehat{\mathcal{M}}$ is the
quotient mapping (Definition~\ref{D: Asymptotic Fields}).

\item Let $\widehat{f},\,  \widehat{g}\in\widehat{\mathcal{M}}(\Omega)$.  We say that
$\widehat{f}$ and $\widehat{g}$ are {\bf weakly equal}, and write $\widehat{f}\cong\widehat{g}$, if
$\bra\widehat{f},\;
\tau\ket=\bra\widehat{g},\;
\tau\ket$ for all  $\tau\in\mathcal{D}(\mathbb{R}^d)$. We shall call $\cong$ a {\bf weak
equality} in $\widehat{\mathcal{M}}(\Omega)$. Similarly, we say that
$\widehat{f}$ and
$\widehat{g}$ are {\bf weakly infinitely close} (or simply {\em infinitely close} for short), and write
$\widehat{f}\approx\widehat{g}$, if
$\bra\widehat{f},\;
\tau\ket\approx\bra\widehat{g},\;
\tau\ket$ in $^*\mathbb{C}$ for all  $\tau\in\mathcal{D}(\mathbb{R}^d)$. We shall call $\approx$ a
{\bf weak infinitesimal relation} in $\widehat{\mathcal{M}}(\Omega)$.

\item Let $\widehat{f}\in\widehat{\mathcal{M}}(\Omega)$ and
$\widehat{x}\in\mu_\mathcal{M}(\Omega)$ (\ref{E: F-Monad}). We define the {\bf value of
$\widehat{f}$ at}
$\widehat{x}$ by the formula $\widehat{f}(\widehat{x})=\widehat{f(x)}$. We shall use the same notation,
$\widehat{f}$, for the corresponding graph
$\widehat{f}:\mu_\mathcal{M}(\Omega)\to\widehat{\mathcal{M}}$.
\item  Let $\Omega,
\mathcal{O}$ be two open sets of $\mathbb{R}^d$ such that
$\mathcal{O}\subseteq\Omega$. Let
$\widehat{f}\in\widehat{\mathcal{M}}(\Omega)$. We define the {\bf restriction}
$\widehat{f}\upharpoonright\mathcal{O}$ of
$\widehat{f}$ on $\mathcal{O}$ by the formula
\[
\widehat{f}\upharpoonright\mathcal{O}=\widehat{f|^*\mathcal{O}},
\]
where $^*\mathcal{O}$ is the non-standard extension of $\mathcal{O}$ and $f|^*\mathcal{O}$ is 
the usual (pointwise) restriction of $f$ on $^*\mathcal{O}$.

\end{enumerate}
\end{definition}
\begin{theorem}[Some Basic Results]\label{T: Some Basic Results} Let $\mathcal{M}$ be (as before) a convex
subring of ${^*\mathbb{C}}$. Then: 
\begin{description}
\item{\bf (i)} $\mathcal{M}(\Omega)$ is a differential subring of
$^*\mathcal{E}(\Omega)$ and $\mathcal{M}_0(\Omega)$ is a differential ideal in
$\mathcal{M}(\Omega)$. Consequently,
$\widehat{\mathcal{M}}(\Omega)$ is a \textbf{differential ring}.

\item{\bf (ii)} $\mathcal{E}(\Omega)$ is a \textbf{differential subring} of
$\widehat{\mathcal{M}}(\Omega)$ under the embedding
$f\to\widehat{^*f}$. We shall often write this simply as an inclusion
\[
\mathcal{E}(\Omega)\subset\widehat{\mathcal{M}}(\Omega).
\]

	\item{\bf (iii)} Let $\widehat{f}, \widehat{g}\in\widehat{\mathcal{M}}(\Omega)$. Then
$\widehat{f}=\widehat{g}\;\Rightarrow\; \widehat{f}\cong\widehat{g}\;\Rightarrow\;
\widehat{f}\approx\widehat{g}$.

\item{\bf (iv)} The embedding $f\to\widehat{^*f}$ {\bf preserves the pairing} between
$\mathcal{E}(\Omega)$ and $\mathcal{D}(\Omega)$ in the sense that for every
$f\in\mathcal{E}(\Omega)$ and every
$\tau\in\mathcal{D}(\Omega)$ we have
\[
\int_\Omega\, f(x)\, \tau(x)\, dx=\left<\widehat{^*f}, \tau\right>.
\]
Consequently, if $f, g\in\mathcal{E}(\Omega)$, then either of $\widehat{f}\cong\widehat{g}$ or 
$\widehat{f}\approx\widehat{g}$ implies $f=g$.

\item{\bf (v)} The embedding $\widehat{\mathcal{M}}(\Omega)\hookrightarrow
\widehat{\mathcal{M}}^{\; \mu_\mathcal{M}(\Omega)}$, defined by the pointwise values of
$\widehat{f}\in\widehat{\mathcal{M}}(\Omega)$, preserves the addition, multiplication
and partial differentiation in $\widehat{\mathcal{M}}(\Omega)$.

\item{\bf (vi)} For every \textbf{arcwise connected open set} $\Omega$ of $\mathbb{R}^d$ we have
\begin{equation}\label{E: Ring of Scalars}\notag
\widehat{\mathcal{M}}=
\left\{ \widehat{f}\in \widehat{\mathcal{E}}_\mathcal{M}(\Omega)\mid \nabla
\widehat{f}=0\right\}.
\end{equation}
 In particular,
\begin{equation}\label{E: Ring of Scalars}\notag
\widehat{\mathcal{M}}=
\left\{ \widehat{f}\in \widehat{\mathcal{E}}_\mathcal{M}(\mathbb{R}^d)\mid \nabla
\widehat{f}=0\right\}.
\end{equation}

\item{\bf (vii)} Let $c\in\mathcal{M}$ and $f_c\in{^*\mathcal{E}}(\Omega)$ denote the constant
function  $f_c(x)=c$ for all $x\in{^*\Omega}$. Then the  mapping
$\widehat{c}\to \widehat{f_c}$ from
$\widehat{\mathcal{M}}$ to $\widehat{\mathcal{M}}(\Omega)$ is a
differentail ring embedding. Consequently,
$\widehat{\mathcal{M}}(\Omega)$ is a \textbf{differential algebra over the field}
$\widehat{\mathcal{M}}$ under the ring operations in
$\widehat{\mathcal{M}}(\Omega)$. In particular the multiplication of functions in
$\widehat{\mathcal{M}}(\Omega)$ by scalars in $\widehat{\mathcal{M}}$ is defined by
$\widehat{c}\, \widehat{f}=\widehat{cf}$. Also $\mathcal{E}(\Omega)$ is a \textbf{differential
subalgebra of $\widehat{\mathcal{M}}(\Omega)$ over the field}\;  $\mathbb{C}$. We
shall often identify $\widehat{c}$ with its image $\widehat{f_c}$ and write simply
$\widehat{\mathcal{M}}\subset\widehat{\mathcal{M}}(\Omega)$ similarly to the
more conventional $\mathbb{C}\subset\mathcal{E}(\Omega)$.
\item{\bf (viii)} Let $\mathcal{T}_d$ stand for the usual topology on $\mathbb{R}^d$.
 The collection
$\mathcal{S}_\mathcal{M}=:
\{\widehat{\mathcal{M}}(\Omega)\}_{\Omega\in{\mathcal{T}_d}}$
is a {\bf sheaf} on the topological space
$(\mathbb{R}^d, \mathcal{T}_d)$ under the restriction $\restriction$. Consequently, every function
$\widehat{f}\in\widehat{\mathcal{M}}(\Omega)$ has a \textbf{support}\, 
$\supp(\widehat{f})$ which is a \textbf{closed set of} $\Omega$.
\end{description}
\end{theorem}

\Proof The properties (i)-(v) follow easily from the definition of 
$\widehat{\mathcal{M}}(\Omega)$ and we shall leave to the reader to check the detail. 
We shall proof (vi) and (vii) in Section~\ref{S: Pointwise Values and Fundamental Theorem} and we
shall prove (viii) in Section~\ref{S: Local Properties of Asymptotic Functions}.


	  Here are {\bf several examples} algebras of asymptotic functions. 

\begin{example}[$\mathcal{C}^\infty$-Functions]\label{Ex: Cinfinity-Functions} Let 
$\mathcal{M}=\mathcal{F}(^*\mathbb{C})$. In this case we have
$\mathcal{M}_0=\mathcal{I}(^*\mathbb{C})$ and $\widehat{\mathcal{M}}={\mathbb{C}}$
(Example~\ref{Ex: Finite Numbers}). For the $\mathcal{M}$-moderate and
$\mathcal{M}$-negligible functions we
have $\mathcal{M}(\Omega)=\mathcal{F}(^*\mathcal{E}(\Omega))$ and 
$\mathcal{M}_0(\Omega)=\mathcal{I}(^*\mathcal{E}(\Omega))$, where 
\begin{align}
&\mathcal{F}(^*\mathcal{E}(\Omega))=:\{f\in{^*\mathcal{E}(\mathbb{R}^{d})}\mid
(\forall\alpha\in\mathbb{N}_0^d)(\forall x\in\mu(\Omega)(\partial^\alpha
f(x)\in\mathcal{F}(^*\mathbb{C})),\notag\\
&\mathcal{I}(^*\mathcal{E}(\Omega))=:\{f\in{^*\mathcal{E}(\mathbb{R}^{d})}\mid
(\forall\alpha\in\mathbb{N}_0^d)(\forall x\in\mu(\Omega)(\partial^\alpha
f(x)\in\mathcal{I}(^*\mathbb{C}))\},\notag
\end{align}
respectively. The corresponding ring of $\mathcal{M}$-asymptotic functions
\[
\widehat{\mathcal{F}}(\Omega)=\mathcal{F}(^*\mathcal{E}(\Omega))/\mathcal{I}(^*\mathcal{E}(\Omega)),
\]
 is isomorphic to the ring $\mathcal{E}(\Omega)=\mathcal{C}^\infty(\Omega)$
of the usual $\mathcal{C}^\infty$-functions on $\Omega$.
\end{example}


\begin{example}\label{Ex: F=Frho} Let 
$\rho$ be (as before) a positive infinitesimal in
${^*\mathbb{R}}$ and let $\mathcal{M}=\mathcal{F}_\rho$ and
$\mathcal{M}_0=\mathcal{I}_\rho$ (Example~\ref{Ex: Logarithmic Rings}).
For the $\mathcal{M}$-moderate and
$\mathcal{M}$-negligible functions we have
$\mathcal{M}(\Omega)=\mathcal{F}_\rho(^*\mathcal{E}(\Omega))$ and 
$\mathcal{M}_0(\Omega)=\mathcal{I}_\rho(^*\mathcal{E}(\Omega))$, where
\begin{align}
&\mathcal{F}_\rho(^*\mathcal{E}(\Omega))=:\left\{f\in{^*\mathcal{E}(\Omega)}
\mid(\forall\alpha\in\mathbb{N}_0^d)(\forall
x\in\mu(\Omega))\left[\partial^\alpha
f(x)\in\mathcal{F}_\rho\right]\right\},\notag\\
&\mathcal{I}_\rho(^*\mathcal{E}(\Omega))=:\left\{f\in{^*\mathcal{E}(\Omega)}
\mid(\forall\alpha\in\mathbb{N}_0^d)(\forall
x\in\mu(\Omega))\left[\partial^\alpha
f(x)\in\mathcal{I}_\rho\right]\right\},\notag
\end{align}
respectively.
The corresponding ring of  $\mathcal{M}$-asymptotic functions
\[
\widehat{\mathcal{F}_\rho}(\Omega)=
\mathcal{F}_\rho(^*\mathcal{E}(\Omega))/\mathcal{I}_\rho(^*\mathcal{E}(\Omega)),
\]
is an algebra over the logarithmic field $\widehat{\mathcal{F}_\rho}$
(Example~\ref{Ex: Examples of Asymptotic Fields}).
\end{example}
\begin{example}[$\rho$-Asymptotic Functions]\label{Ex: rho-Asymptotic Functions} Let  
$\rho$ be a positive infinitesimal in ${^*\mathbb{R}}$ and let
$\mathcal{M}=\mathcal{M}_\rho$ and
$\mathcal{M}_0=\mathcal{N}_\rho$ (Example~\ref{Ex: Robinson Rings}). For the
$\mathcal{M}$-moderate and $\mathcal{M}$-negligible functions we have
$\mathcal{M}(\Omega)=\mathcal{M}_\rho(^*\mathcal{E}(\Omega))$ and
$\mathcal{M}_0(\Omega)=\mathcal{N}_\rho(^*\mathcal{E}(\Omega))$, where
\begin{align}
&\mathcal{M}_\rho(^*\mathcal{E}(\Omega))=:\left\{f\in{^*\mathcal{E}(\Omega)}
\mid(\forall\alpha\in\mathbb{N}_0^d)(\forall
x\in\mu(\Omega))\left[\partial^\alpha
f(x)\in\mathcal{M}_\rho\right]\right\},\notag\\
&\mathcal{N}_\rho(^*\mathcal{E}(\Omega))=:\left\{f\in{^*\mathcal{E}(\Omega)}
\mid(\forall\alpha\in\mathbb{N}_0^d)(\forall
x\in\mu(\Omega))\left[\partial^\alpha
f(x)\in\mathcal{N}_\rho\right]\right\},\notag
\end{align}
respectively.
The corresponding ring of $\mathcal{M}$-asymptotic functions 
\[
\widehat{\mathcal{M}_\rho}(\Omega)=\mathcal{M}_\rho(^*\mathcal{E}(\Omega))/\mathcal{N}_\rho(^*\mathcal{E}(\Omega)),
\]
denoted also by $^\rho\mathcal{E}(\Omega)$, is an algebra over A. Robinson field $^\rho\mathbb{C}$
(Example~\ref{Ex: Examples of Asymptotic Fields}). The algebra
$^\rho\mathcal{E}(\Omega)$ is introduced in (M.
Oberguggenberger and T. Todorov\cite{OberTod98}) under the name {\bf $\rho$-asymptotic
functions}.  We shall follow this terminology. The reader will find a more detail about
$^\rho\mathcal{E}(\Omega)$ in  Chapter~\ref{Ch: Asymptotic Functions}. The algebra
$^\rho\mathcal{E}(\Omega)$ is, in a sense, a non-standard counterpart of a {\bf special 
Colombeau's algebra} (J. F. Colombeau~\cite{jCol84a}) with the important {\bf improvement of
the properties of the scalars:} The ring of the scalars $^\rho\mathbb{C}$ of
$^\rho\mathcal{E}(\Omega)$ constitutes an algebraically closed Cantor-complete field (Theorem~\ref{T:
Cantor Completeness}). In contrast, the  ring of the scalars $\widetilde{\mathbb{C}}$ of Colombeau
algebra
$\mathcal{G}(\Omega)$ is a ring with zero divisors. 
\end{example}
\begin{example}[Exponential Asymptotic Functions]\label{Ex: Exponential Asymptotic Functions}Let 
$\rho$ be (as before) a positive infinitesimal in
${^*\mathbb{R}}$ and let $\mathcal{M}=\mathcal{E}_\rho$ and
$\mathcal{M}_0=\mathcal{E}_{\rho,0}$ (Example~\ref{Ex: Logarithmic-Exponential
Rings}).
 The corresponding ring of
asymptotic functions
$\widehat{\mathcal{E}_\rho}(\Omega)$ is an algebra over the exponential field
$\widehat{\mathcal{E}_\rho}$ (Example~\ref{Ex: Examples of Asymptotic Fields}).
\end{example}
\begin{example}[The case $\mathcal{M}={^*\mathbb{C}}$]\label{Ex: The case F=*C} Let
$\mathcal{M}={^*\mathbb{C}}$. In this case $\mathcal{M}_0=\{0\}$ and
$\widehat{\mathcal{M}}={^*\mathbb{C}}$ (Example~\ref{Ex: The case F=*C}). For the 
$\mathcal{M}$-moderate and $\mathcal{M}$-negligible functions we have
$\mathcal{M}(\Omega)={^*\mathcal{E}}(\Omega)$ and
$\mathcal{M}_0(\Omega)={^*\mathcal{E}_0}(\Omega)$, respectively, where
\begin{align}
&{^*\mathcal{E}_0}(\Omega)=\{f\in{^*\mathcal{E}(\mathbb{R}^{d})}:
(\forall\alpha\in\mathbb{N}_0^d)(\forall x\in\mu(\Omega)(\partial^\alpha f(x)=0)\},\notag
\end{align}
The ring of the corresponding $\mathcal{M}$-asymptotic functions 
\[
\widehat{^*\mathbb{C}}(\Omega)={^*\mathcal{E}}(\Omega)/{^*\mathcal{E}_0}(\Omega).
\]
is an algebra over the field $^*\mathbb{C}$. The algebra $\widehat{^*\mathbb{C}}(\Omega)$ is, in
a sense, a {\bf non-standard counterpart of Egorov algebra}
(Yu. V. Egorov~\cite{yEgorov90a}-\cite{yEgorov90b}) with the important improvement of
the properties of the scalars:  The ring of the scalars $^*\mathbb{C}$ of
$\widehat{^*\mathbb{C}}(\Omega)$ constitutes an algebraically closed saturated field. In contrast,
the the scalars of Egorov's algebra are a ring with zero divisors.  The algebra
$\widehat{^*\mathbb{C}}(\Omega)$ will be studied in detail in Chapter~\ref{Ch: Non-Standard Smooth
functions}.
\end{example}

	\section{$\mathcal{M}$-Moderate and $\mathcal{M}$-Negligible Functions}
	In this section we present \textbf{several characterizations} of the
$\mathcal{M}$-moderate and $\mathcal{M}$-negligible functions (Section~\ref{S:  F-Asymptotic
Functions}).

	Through out this section $\mathcal{M}$ stands for a convex subring of $^*\mathbb{C}$
(Section~\ref{S: F-Asymptotic Numbers: Definitions and Examples}) and
$\mathbb{M}\in\mathcal{M}ax(\mathcal{M})$ stands for a maximal field within 
$\mathcal{M}$ (Definition~\ref{D: Maximal Fields}).

\begin{theorem}\label{T: A Characterization} Let
$f\in{^*\mathcal{E}}(\Omega)$. Then the following are equivalent:
\begin{description}
\item{\bf (i)}\;\,  $(\forall x\in\mu(\Omega))(f(x)\in\mathcal{M})$. 
\item{\bf (ii)}\;  $(\forall x\in\mu(\Omega))(\exists M\in\mathbb{M}_+)(|f(x)|\leq M)$. 
\item{\bf (iii)} $(\forall K\subset\subset\Omega)(\exists
M\in\mathbb{M}_+)(\sup_{x\in{^*K}} |f(x)|\leq M)$.

\item{\bf (iv)}\,  $(\forall x\in\mu(\Omega))(\exists A\in\mathcal{M}\setminus\mathcal{M}_0)(|f(x)|\leq
A)$. 
\item{\bf (v)}\, $(\forall K\subset\subset\Omega)(\exists A\in\mathcal{M}\setminus\mathcal{M}_0)
(\sup_{x\in{^*K}} |f(x)|\leq A)$.
\item{\bf (vi)} $(\forall x\in\mu(\Omega))(\forall
B\in{^*\mathbb{R}_+}\setminus\mathcal{M})(|f(x)|<B)$.  
\item{\bf (vii)} $(\forall K\subset\subset\Omega)(\forall
B\in{^*\mathbb{R}_+}\setminus\mathcal{M}) (\sup_{x\in{^*K}} |f(x)|< B)$. 
\end{description}
\end{theorem}

\begin{remark}\label{R: Relaxation 1} {\em We should note that the above theorem remains true even if
the maximal field $\mathbb{M}$ is replaced by a set $S\subseteq\mathcal{M}\setminus\mathcal{M}_0$
such that $S$ contains arbitrarily large numbers.
}\end{remark}

\Proof (i)$\Leftrightarrow$(ii) follows immediately by part~(i) of Theorem~\ref{T:
Characterization}.

	(ii)$\Rightarrow$(iii): Let $K\subset\subset\Omega$ and recall that $^*K\subset\mu(\Omega)$ by
Theorem~\ref{T: The Usual Topology on Rd}.  We observe that $\sup_{\xi\in{^*K}}
|f(\xi)|\in\mathcal{M}$. Indeed, suppose  (on the contrary) that
$\gamma=:\sup_{\xi\in{^*K}} |f(\xi)|\notin\mathcal{M}$ which implies also
$\gamma/2\notin\mathcal{M}$. There exists $y\in{^*K}$ such that
$\gamma/2<|f(y)|<\gamma$ by the choice of
$\gamma$.  It follows $f(y)\notin\mathcal{M}$ which contradicts to (i) (hence it contradicts to (ii))
since
$y\in\mu(\Omega)$.  On the other hand, $\sup_{\xi\in{^*K}} |f(\xi)|\in\mathcal{M}$ implies that the
internal set
\[
\mathcal{A}=\{a\in{^*\mathbb{R}_+} : \sup_{\xi\in{^*K}} |f(\xi)|\leq a\},
\]
contains $^*\mathbb{R}_+\setminus\mathcal{M}$ by by part~(ii) of Theorem~\ref{T:
Characterization}. Thus
$\mathcal{A}$ contains  arbitrarily small numbers in ${^*\mathbb{C}}\setminus\mathcal{M}$. It
follows that
$\mathcal{A}\cap(\mathcal{M}\setminus\mathcal{M}_0)\not=\varnothing$ by the Underflow of
${^*\mathbb{C}}\setminus\mathcal{M}$ (Theorem~\ref{T: Spilling Principles}). Thus
$\sup_{x\in{^*K}}|f(x)|\leq A$ holds for any
$A\in\mathcal{A}\cap(\mathcal{M}\setminus\mathcal{M}_0)$. Also there exists
$M_1\in\mathbb{M}$ such that
$A-M_1\in\mathcal{M}_0$ by part~(i) of Theorem~\ref{T: Embeddings}. Let
$H\in\mathbb{M}_+$. Then  (iii) holds for $M=M_1+H$.

(iii)$\Rightarrow$(iv): Suppose that $x\in\mu(\Omega)$ and observe that $\st(x)\in\Omega$ by the
definition of $\mu(\Omega)$. Since $\Omega$ is an open set, there exists
$\varepsilon\in\mathbb{R}_+$ such that $K\subset\subset\Omega$, where $K=\{r\in\Omega\,
:\, |r-\st(x)|\leq\varepsilon\}$. There exists $M\in\mathbb{M}_+$ such that 
$\sup_{\xi\in{^*K}} |f(\xi)|\leq M$ by assumption which implies (iv) for $A=M$ since $x\in{^*K}$ and
$M\in\mathbb{M}_+\subset\mathcal{M}\setminus\mathcal{M}_0$.

The proof of (iv)$\Rightarrow$(v) is almost identical to the proof of (ii)$\Rightarrow$(iii) and we leave it
to the reader.

(v)$\Rightarrow$(vi) follows immediately by part~(ii) of Theorem~\ref{T: Characterization}.

 (vi)$\Rightarrow$(vii): Suppose (on the contrary) that  $\gamma=:\sup_{\xi\in{^*K}} |f(\xi)|\geq B$ for
some $K\subset\subset\Omega$ and some $B\in{^*\mathbb{R}_+}\setminus\mathcal{M}$. We
have $B/2\leq |f(y)|<\gamma$ for some $y\in{^*K}$ by the choice of $\gamma$. This contradicts (vi)
since  $y\in\mu(\Omega)$ and $B/2\in{^*\mathbb{R}_+}\setminus\mathcal{M}$.

 (vii)$\Rightarrow$(i): Suppose that $x\in\mu(\Omega)$ and observe that $\st(x)\in\Omega$ by the
definition of $\mu(\Omega)$. As before there exists $K\subset\subset\Omega$ such that $x\in{^*K}$. As
before the internal set $\mathcal{A}$ contains $^*\mathbb{R}_+\setminus\mathcal{M}$. Thus (as
before) $\mathcal{A}\cap(\mathcal{M}\setminus\mathcal{M}_0)\not=\varnothing$ by the Underflow
for $^*\mathbb{C}\setminus\mathcal{M}$ (Theorem~\ref{T: Spilling Principles}). Thus
$\sup_{\xi\in{^*K}} |f(\xi)|< A$ for any
$A\in\mathcal{A}\cap(\mathcal{M}\setminus\mathcal{M}_0)$. It follows that $|f(x)|< A$ since
$x\in{^*K}$ by the choice of $K$. Thus $f(x)\in\mathcal{M}$ (as required) by the convexity of
$\mathcal{M}$.

$\blacktriangle$

	Here is a {\bf list of characterizations of the $\mathcal{M}$-moderate functions}.

\begin{corollary}[$\mathcal{M}$-Moderate Functions]\label{C: F-Moderate Functions} Let
$f\in{^*\mathcal{E}}(\Omega)$. Then the following are equivalent:
\begin{description}

\item{\bf (i)}\quad  $f\in\mathcal{M}_\mathcal{M}(\Omega)$.
\item{\bf (ii)}\;  $(\forall
\alpha\in\mathbb{N}_0^d)(\forall x\in\mu(\Omega))(\exists M\in\mathbb{M}_+)(|\partial^\alpha f(x)|\leq
M)$. 
\item{\bf (iii)}\, $(\forall
\alpha\in\mathbb{N}_0^d)(\forall K\subset\subset\Omega)
(\exists
M\in\mathbb{M}_+) (\sup_{x\in{^*K}} |\partial^\alpha f(x)|\leq M)$. 
\item{\bf (iv)}\;  $(\forall
\alpha\in\mathbb{N}_0^d)(\forall x\in\mu(\Omega))(\exists
A\in\mathcal{M}\setminus\mathcal{M}_0)(|\partial^\alpha f(x)|\leq A)$. 
\item{\bf (v)}\;\,  $(\forall
\alpha\in\mathbb{N}_0^d)(\forall K\subset\subset\Omega)(\exists
A\in\mathcal{M}\setminus\mathcal{M}_0) (\sup_{x\in{^*K}} |\partial^\alpha f(x)|\leq A)$. 
\item{\bf (vi)}\; $(\forall
\alpha\in\mathbb{N}_0^d)(\forall x\in\mu(\Omega))(\forall
B\in{^*\mathbb{R}_+}\setminus\mathcal{M})(|\partial^\alpha f(x)|<B)$. 
\item{\bf (vii)}\, $(\forall
\alpha\in\mathbb{N}_0^d)(\forall K\subset\subset\Omega)(\forall
B\in{^*\mathbb{R}_+}\setminus\mathcal{M}) (\sup_{x\in{^*K}} |\partial^\alpha f(x)|< B)$. 
\end{description}
\end{corollary}
\begin{remark}\label{R: Relaxation 2} {\em We should note that the above corollary remains true
even if the maximal field $\mathbb{M}$ is replaced by a set
$S\subseteq\mathcal{M}\setminus\mathcal{M}_0$ such that $S$ contains arbitrarily large numbers.
}\end{remark}
\Proof An immediate after replacing $f$ by $\partial^\alpha f$ in Theorem~\ref{T: A
Characterization}. 

 $\blacktriangle$

	We turn to the $\mathcal{M}$-negligible functions.

\begin{theorem}\label{T: Another Characterization} Let
$f\in{^*\mathcal{E}}(\Omega)$. Then the following are equivalent:
\begin{description}
\item{\bf (i)}\quad $(\forall x\in\mu(\Omega))(f(x)\in\mathcal{M}_0)$. 

\item{\bf (ii)}\quad $(\forall x\in\mu(\Omega))(\forall M\in\mathbb{M}_+)(|f(x)|< M)$. 

\item{\bf (iii)}\; $(\forall K\subset\subset\Omega)(\forall
M\in\mathbb{M}_+) (\sup_{x\in{^*K}} |f(x)|< M)$. 

\item{\bf (iv)}\; $(\forall x\in\mu(\Omega))(\exists A\in\mathcal{M}_0)(|f(x)|\leq A)$. 

\item{\bf (v)}\quad $(\forall K\subset\subset\Omega)(\exists A\in\mathcal{M}_0)
(\sup_{x\in{^*K}} |f(x)|\leq A)$. 

\item{\bf (vi)}\, $(\forall x\in\mu(\Omega))(\forall
B\in\mathcal{M}\setminus\mathcal{M}_0)(|f(x)|<|B|)$. 

\item{\bf (vii)} $(\forall K\subset\subset\Omega)(\forall
B\in\mathcal{M}\setminus\mathcal{M}_0) (\sup_{x\in{^*K}} |f(x)|< |B|)$. 
\end{description}
\end{theorem}
\begin{remark}\label{R: Relaxation 3} {\em We should note that the above theorem remains true even if
the maximal field $\mathbb{M}$ is replaced by a set $S\subseteq\mathcal{M}\setminus\mathcal{M}_0$
such that $S$ contains arbitrarily small numbers.
}\end{remark}
\Proof We shall prove the equivalence of (i) and (v) only and leave the rest of the proof to
the reader (who might decide to adapt the arguments used in the proof of
the previous lemma). 

	(i)$\Rightarrow$(v) Suppose that
$K$ is a compact subset of
$\Omega$ and recall that $^*K\subset\mu(\Omega)$ by Theorem~\ref{T: The Usual Topology on
Rd}. Notice that $\sup_{x\in{^*K}}|f(x)|\in\mathcal{M}_0$. Indeed, suppose (on the contrary) that
$\gamma=:\sup_{x\in{^*K}}|f(x)|\notin\mathcal{M}_0$ which implies
$\gamma/2\notin\mathcal{M}_0$. Also  there exists $y\in{^*K}$ such that $\gamma/2<|f(y)|<\gamma$
by the choice of $\gamma$. Thus
$|f(y)|\notin\mathcal{M}_0$ contradicting to our assumption (i) since $y\in\mu(\Omega)$. On the
other hand,
$\sup_{x\in{^*K}}|f(x)|\in\mathcal{M}_0$ implies that the internal set
\[
\mathcal{A}=\{c\in{^*\mathbb{C}} : \sup_{x\in{^*K}}|f(x)|\leq |c|\, \},
\] 
contains $\mathcal{M}\setminus\mathcal{M}_0$ by by part~(ii) of Theorem~\ref{T: Characterization}.
It follows that
$\mathcal{A}\cap\mathcal{M}_0\not=\varnothing$ by the Underflow of
$\mathcal{M}\setminus\mathcal{M}_0$ (Theorem~\ref{T: Spilling Principles}). Thus 
$\sup_{x\in{^*K}} |f(x)|\leq A$ holds (as required) for any
$c\in\mathcal{A}\cap\mathcal{M}_0$ and $A=|c|$. 

	(i)$\Leftarrow$(v): Suppose that $x\in\mu(\Omega)$. As in the previous lemma, there exists
$\varepsilon\in\mathbb{R}_+$ such that 
$K=\{r\in\Omega\, :\, |r-\st(x)|\leq\varepsilon\}\subset\subset\Omega$.
Observe that there exists $A\in\mathcal{M}_0$ such that $\sup_{\xi\in{^*K}}|f(\xi)|\leq A$ by
assumption. Thus $f(\xi)\in\mathcal{M}_0$ for all $\xi\in{^*K}$ (as required) by the 
convexity of $\mathcal{M}_0$.

$\blacktriangle$

	Here is a {\bf list of  characterizations of the $\mathcal{M}$-negligible functions}.

\begin{corollary}[$\mathcal{M}$-Negligible Functions]\label{C: F-Negligible Functions} Let
$f\in{^*\mathcal{E}}(\Omega)$. Then the following are equivalent:
\begin{description}
\item{\bf (i)}\quad  $f\in\mathcal{N}_\mathcal{M}(\Omega)$.

\item{\bf (ii)}\quad $(\forall
\alpha\in\mathbb{N}_0^d)(\forall x\in\mu(\Omega))(\forall M\in\mathbb{M}_+)(|f(x)|< M)$. 

\item{\bf (iii)}\; $(\forall
\alpha\in\mathbb{N}_0^d)(\forall K\subset\subset\Omega)(\forall
M\in\mathbb{M}_+) (\sup_{x\in{^*K}} |f(x)|< M)$. 

\item{\bf (iv)}\; $(\forall
\alpha\in\mathbb{N}_0^d)(\forall x\in\mu(\Omega))(\exists A\in\mathcal{M}_0)(|f(x)|\leq A)$. 

\item{\bf (v)}\quad $(\forall
\alpha\in\mathbb{N}_0^d)(\forall K\subset\subset\Omega)(\exists A\in\mathcal{M}_0)
(\sup_{x\in{^*K}} |f(x)|\leq A)$. 

\item{\bf (vi)}\, $(\forall
\alpha\in\mathbb{N}_0^d)(\forall x\in\mu(\Omega))(\forall
B\in\mathcal{M}\setminus\mathcal{M}_0)(|f(x)|<|B|)$. 

\item{\bf (vii)} $(\forall
\alpha\in\mathbb{N}_0^d)(\forall K\subset\subset\Omega)(\forall
B\in\mathcal{M}\setminus\mathcal{M}_0) (\sup_{x\in{^*K}} |f(x)|< |B|)$. 
\end{description}
\end{corollary}
\begin{remark}\label{R: Relaxation 4} {\em We should note that the above corollary remains true even
if the maximal field $\mathbb{M}$ is replaced by a set
$S\subseteq\mathcal{M}\setminus\mathcal{M}_0$ such that $S$ contains 
arbitrarily small numbers. }\end{remark}
\Proof An immediate after replacing $f$ by $\partial^\alpha f$ in Theorem~\ref{T: Another
Characterization}.  

$\blacktriangle$

		In the next theorem we present several more characterizations of the
$\mathcal{M}$-negligible functions (in addition to the presented above), 
where the quantifier $\forall \alpha\in\mathbb{N}^d_0$ is replaced simply by $\alpha=0$.

\begin{theorem}[A Simplification]\label{T: A Simplification} Let
$f\in\mathcal{M}_\mathcal{M}(\Omega)$. Then $f\in\mathcal{N}_\mathcal{M}(\Omega)$
\ifff  $f(x)\in\mathcal{M}_0$ for all $x\in\mu(\Omega)$. Consequently, we have the following
several formulas for $\mathcal{N}_\mathcal{M}(\Omega)$:
\begin{align}
&\mathcal{N}_\mathcal{M}(\Omega)=\{f\in\mathcal{M}_\mathcal{M}(\Omega)\mid
(\forall x\in\mu(\Omega)(f(x)\in\mathcal{M}_0)\},\notag\\
&\mathcal{N}_\mathcal{M}(\Omega)=\{f\in\mathcal{M}_\mathcal{M}(\Omega)\mid
(\forall x\in\mu(\Omega))(\forall
M\in\mathbb{M}_+)(|f(x)|<M)\},\notag\\
&\mathcal{N}_\mathcal{M}(\Omega)=\{f\in\mathcal{M}_\mathcal{M}(\Omega)\mid
(\forall K\subset\subset\Omega)(\forall
M\in\mathbb{M}_+) (\sup_{x\in{^*K}} |f(x)|< M)\},\notag\\
&\mathcal{N}_\mathcal{M}(\Omega)=\{f\in\mathcal{M}_\mathcal{M}(\Omega)\mid
(\forall x\in\mu(\Omega))(\exists A\in\mathcal{M}_0)
( |f(x)|\leq A)\},\notag\\
&\mathcal{N}_\mathcal{M}(\Omega)=\{f\in\mathcal{M}_\mathcal{M}(\Omega)\mid
(\forall K\subset\subset\Omega)(\exists A\in\mathcal{M}_0)
(\sup_{x\in{^*K}} |f(x)|\leq A)\},\notag\\
&\mathcal{N}_\mathcal{M}(\Omega)=\{f\in\mathcal{M}_\mathcal{M}(\Omega)\mid
(\forall x\in\mu(\Omega))(\forall
B\in\mathcal{M}\setminus\mathcal{M}_0)(|f(x)|<|B|)\},\notag\\
&\mathcal{N}_\mathcal{M}(\Omega)=\{f\in\mathcal{M}_\mathcal{M}(\Omega)\mid
(\forall K\subset\subset\Omega)(\forall
B\in\mathcal{M}\setminus\mathcal{M}_0) (\sup_{x\in{^*K}} |f(x)|< |B|)\}.\notag
\end{align}
\end{theorem}
\Proof ($\Rightarrow$) follows immediately after letting $\alpha=0$. 

($\Leftarrow$) Suppose that $x\in\mu(\Omega)$.  We have to show that
$\partial^\alpha f(x)\in\mathcal{M}_0$ for all multi-indexes $\alpha\in\mathbb{N}_0^d,\;
|\alpha|\geq 1$. We start with $|\alpha|=1$. If $\nabla f(x)=0$, there is nothing to prove. 
Suppose that $\nabla f(x)\not=0$ and let $\varepsilon\in\mathbb{M}_+$. It suffices to show that 
$||\nabla f(x)||<\varepsilon$ in view of Theorem~\ref{T: Characterization}. Since $\Omega$ is an open
set, there exists an open relatively compact set $\mathcal{O}$ of $\Omega$ such that
$\st(x)\in\mathcal{O}\subset\subset\Omega$. Now
$f\in\mathcal{M}_\mathcal{M}(\Omega)$  implies 
$\left|\sum_{|\alpha|=2}\partial^\alpha f(\xi)\right|<\delta$ for some $\delta\in\mathbb{M}_+$ and
all $\xi\in{^*\mathcal{O}}$ by Corollary~\ref{C: F-Moderate Functions} since
$^*\mathcal{O}\subset\mu(\Omega)$. Let
$h\in\mathcal{I}(\mathbb{M}^d)$ be an infinitesimal vector with the direction of $\nabla f(x)$
and of length $||h||<\varepsilon/\delta$. Notice that $||h||\in\mathbb{M}_+$ thus
$||h||\in\mathcal{M}\setminus \mathcal{M}_0$ which is important for what follows. We have
$|f(x+h)-f(x)|<\delta||h||^2/2$ by part~(vi) of Theorem~\ref{T: Characterization} since
$f(x+h)-f(x)\in\mathcal{M}_0$ by assumption and
$x+h\in\mu(\Omega)$. Next we observe that the Taylor formula:
\[
\nabla f(x)\cdot h= f(x+h)-f(x)-\frac{1}{2}\sum_{|\alpha|=2}\partial^\alpha f(x+\theta h)\, h^\alpha.
\]
holds for some $\theta\in{^*\mathbb{R}},\; 0<\theta<1$, by Transfer Principle (Theorem~\ref{T:
Transfer Principle}). Thus $x+\theta h\approx x\approx \st(x)$  implying
$x+\theta h\in{^*\mathcal{O}}$. We have 
\[
\left|\nabla f(x)\cdot
h\right|<\delta||h||^2/2+\delta||h||^2/2<\delta ||h||^2.
\] 
Also we have $|\nabla f(x)\cdot h|= ||\nabla f(x)||\, ||h||$ by the choice of the direction of $h$. It follows 
$||\nabla f(x)||=\delta ||h||<\varepsilon$ as required. We generalize this result for $|\alpha|=2,
3,\dots$ by induction.  The different formulas for $\mathcal{N}_\mathcal{M}(\Omega)$ follow
immediately by Theorem~\ref{T: Another Characterization}.
$\blacktriangle$
\section{Pointwise Values and Fundamental Theorem}\label{S: Pointwise Values and 
Fundamental Theorem} 

	Recall that every non-standard smooth function $f\in{^*\mathcal{E}}(\Omega)$ can be 
characterized as a pointwise function of the form $f:{^*\Omega}\to{^*\mathbb{C}}$ in the sense that
there exists an embedding ${^*\mathcal{E}}(\Omega)\embed{^*\mathbb{C}}\, ^{^*\Omega}$ which
preserves the ring operations and the partial differentiation of any order (Section~\ref{S:
Non-Standard Smooth Functions}). Among other things the purpose of this section is to  show that
every asymptotic function $\widehat{f}\in\widehat{\mathcal{E}_\mathcal{M}}(\Omega)$  
(Section~\ref{S: F-Asymptotic Functions}) can be characterized as a pointwise function of the form
$\widehat{f}: \mu_\mathcal{M}(\Omega)\to\widehat{\mathcal{M}}$ in the sense that there exists an
embedding $\widehat{\mathcal{E}_\mathcal{M}}(\Omega)\embed
\widehat{\mathcal{M}}\,^{\mu_\mathcal{M}(\Omega)}$ which preserves the ring operations and the
partial differentiation of any order. We also prove a fundamental
theorem of calculus in $\widehat{\mathcal{E}_\mathcal{M}}(\Omega)$.

	We shall use the notation introduced in the first several pages in (Section~\ref{S: F-Asymptotic
Numbers: Definitions and Examples}) and (Section~\ref{S: F-Asymptotic Functions}). In particular, 
$\mathcal{M}$ stands for a convex subring of $^*\mathbb{C}$ (Section~\ref{S: F-Asymptotic
Numbers: Definitions and Examples}). If
$\Omega\subseteq\mathbb{R}^d$ is an open set of $\mathbb{R}^d$, then

\begin{equation}\label{E: F-Monad}
\mu_\mathcal{M}(\Omega)=\{r+dx \mid r\in \Omega,\,  dx\in\Re(\widehat{\mathcal{M}}^d),\; 
||dx||\approx 0
\},
\end{equation}
is the $\mathcal{M}$-monad of $\Omega$. Here $\Re(\widehat{\mathcal{M}}^d)$ stands for the 
real part of
the vector space $\widehat{\mathcal{M}}^d$. We denote by
$\widehat{\mathcal{M}}\,^{\mu_\mathcal{M}(\Omega)}$ the ring of the functions $F$ of the
form $F:\mu_\mathcal{M}(\Omega)\to\widehat{\mathcal{M}}$ (Section~\ref{S: F-Asymptotic
Functions}). 

	 In this section we generalize
some of the results in  Todor Todorov~\cite{tdTod99} where the particular case
$\mathcal{M}=\mathcal{M}_\rho(^*\mathbb{C})$ (Example~\ref{Ex: A. Robinson's
Asymptotic Numbers}) is discussed only. The closest counterpart in J.F. Colombeau's theory can be
found in M. Kunzinger and M. Oberguggenberger's article~\cite{KunzOber99}, where a
characterization of Colombeau's generalized functions in
$\mathcal{G}(\Omega)$ in the ring of generalized scalars $\widetilde{\mathbb{C}}$ is established.

	For convenience
of the reader we shall recall the definition pointwise values presented in 
(Section~\ref{S: F-Asymptotic Functions}).

\begin{definition}[Pointwise Values]\label{D: Pointwise Values}
 Let
$\widehat{f}\in\widehat{\mathcal{E}_\mathcal{M}}(\Omega)$ be a $\mathcal{M}$-asymptotic function
(Section~\ref{S: F-Asymptotic Functions}) and
$\widehat{x}\in\mu_\mathcal{M}(\Omega)$ be a $\mathcal{M}$-asymptotic point. We define the {\bf value
of
$\widehat{f}$ at}
$\widehat{x}$ by the formula 
\[
\widehat{f}(\widehat{x})=\widehat{f(x)}.
\]
We shall use the same notation,
$\widehat{f}$, for the asymptotic function $\widehat{f}\in\widehat{\mathcal{E}_\mathcal{M}}(\Omega)$
and its graph $\widehat{f}\in{\widehat{\mathcal{M}}}^{\mu_\mathcal{M}(\Omega)}$ given by the mapping 
$\widehat{f}:\mu_\mathcal{M}(\Omega)\to\widehat{\mathcal{M}}$.
\end{definition}

	The correctness of the above definition is justified by the following result.
\begin{lemma}[Correctness]\label{L: Correctness} Let $x, y\in\mu(\Omega)$ and 
$f, g\in\mathcal{M}_\mathcal{M}(\Omega)$. Then $x-y\in\mathcal{M}_0$ and
$f-g\in\mathcal{N}_\mathcal{M}(\Omega)$ implies $f(x)-g(y)\in\mathcal{M}_0$.
\end{lemma}
\Proof  We have $f(x)-f(y)=\nabla f(t)\cdot (x-y)$ by
Transfer Principle (Theorem~\ref{T: Transfer Principle}) for some $t\in{^*\mathbb{R}^d}$
between 
$x$ and $y$ (in the sense that $t=x+\theta(y-x)$ for some $\theta\in{^*\mathbb{R}},\;  0<\theta<1$). 
 Also 
\begin{align}
|f(x)-g(y)|=&|f(x)-f(y)+f(y)-g(y)|\leq |f(x)-f(y)|+|f(y)-g(y)|\leq\notag\\
&\leq ||\nabla f(t)||\,||x-y||+|f(y)-g(y)|. \notag
\end{align}
Observe that $x-y\in\mathcal{M}_0$ implies $x-y\approx 0$ by part~(iii) of Theorem~\ref{T:
Characterization}) implying
$\st(x)=\st(y)=\st(t)$. It follows  $t\in\mu(\Omega)$ since $x, y\in\mu(\Omega)$ by assumption. Thus
$f\in\mathcal{M}_\mathcal{M}(\Omega)$ implies $||\nabla f(t)||\in\mathcal{M}$.  For the first term 
we have $||\nabla f(t)||\,||x-y||\in\mathcal{M}_0$ since 
$||x-y||\in\mathcal{M}_0$ by assumption  and $\mathcal{M}_0$ is an ideal in $\mathcal{M}$. Also
$f-g\in\mathcal{N}_\mathcal{M}(\Omega)$  implies 
$|f(y)-g(y)|\in\mathcal{M}_0$ since $y\in\mu(\Omega)$ by assumption. Thus
$|f(x)-g(y)|\in\mathcal{M}_0$ as required. 
$\blacktriangle$

	Here is another similar result which plays some role in what follows.
\begin{lemma}\label{L: F0} Let $x\in\mu(\Omega)$ and $f\in\mathcal{M}_\mathcal{M}(\Omega)$.
Then:
\begin{description}
\item{\bf (i)} $h\in\mathcal{M}_0$ implies $f(x+h)-f(x)\in\mathcal{M}_0$.
\item{\bf (ii)} $h\in\mathcal{M}_0$ and $h\not=0$ implies 
$\frac{|f(x+h)-f(x)-\nabla f(x)\cdot h|}{||h||}\in\mathcal{M}_0$.
\end{description}
\end{lemma}
\Proof (i) follows directly from the previous lemma for $y=x+h$ and $f=g$.

	(ii)  By the Mean Value Theorem applied by Transfer Principle (Theorem~\ref{T:
Transfer Principle}), we have $\nabla f(x)\cdot
h=f(x+h)-f(x)-\frac{1}{2}\sum_{|\alpha|=2}\partial^\alpha f(x+\theta h)\, h^\alpha$
for some $\theta\in{^*\mathbb{R}},\; 0<\theta<1$. Thus we have
\[
\frac{|f(x+h)-f(x)-\nabla f(x)\cdot h|}{||h||}\leq
\frac{1}{2}\sum_{|\alpha|=2}|\partial^\alpha f(x+\theta h)|\, ||h||\in\mathcal{M}_0,
\]
as required, because $\mathcal{M}_0$ is an ideal in $\mathcal{M}$ and
$\partial^\alpha f(x+\theta h)\in\mathcal{M}$ by assumption since $x+\theta h\in\mu(\Omega)$.

$\blacktriangle$

	 Recall that we have the embedding $\mathcal{E}(\Omega)\embed\widehat{\mathcal{E}}(\Omega)$
under the mapping
$f\to\widehat{^*f}$ (Section~\ref{S: F-Asymptotic Functions}). The next result shows that
the evaluation in
$\widehat{\mathcal{E}_\mathcal{M}}(\Omega)$ reduces to the usual evaluation in $\mathcal{E}(\Omega)$.
Recall that \label{Item: Embeddings}

\begin{proposition}[The Usual Evaluation]\label{P: The Usual Evaluation} Let 
$f\in\mathcal{E}(\Omega)$. Then $\widehat{^*f}$ is an extension
of $f$, i.e. $\widehat{^*f}\,|\, \Omega=f$.
\end{proposition}
\Proof  $\widehat{^*f}(\widehat{x})=\widehat{^*f(x)}=\widehat{f(x)}=f(x)$ since $^*f$ is an
extension of $f$. We also have $x=\widehat{x}$ for all $x\in\Omega$ by the 
identification $\Omega$ with
its image in $\Re(\mathcal{M}^d)$ (\#~\ref{Item: Embeddings}, 
Section~\ref{S: F-Asymptotic Numbers: Definitions and Examples}).
Thus $\widehat{^*f}(x)=f(x)$ as required.
 $\blacktriangle$

	 In what follows the cardinal number $\kappa$ stands for the saturation of
$^*\mathbb{C}$ (Section~\ref{S: kappa-Good Two Valued
Measures}). Recall that $\kappa=\card(\mathcal{I})$, where $\mathcal{I}$ is the index set 
used in the construction $^*\mathbb{C}$ (Section~\ref{S: A Non-Standard Analysis: The General
Theory}).
\begin{theorem}[Differential Ring Embedding]\label{T: Differential Ring Embedding} The mapping
\[
\widehat{\mathcal{E}_\mathcal{M}}(\Omega)\ni \widehat{f}\to
\widehat{f}\in\widehat{\mathcal{M}}\,^{\mu_\mathcal{M}(\Omega)},
\]
from\,  $\widehat{\mathcal{E}_\mathcal{M}}(\Omega)$ into
$\widehat{\mathcal{M}}\,^{\mu_\mathcal{M}(\Omega)}$ is a \textbf{differential ring embedding}
in the sense that it is injective and preserves the ring operations and partial differentiation 
of any order.
\end{theorem}
\begin{remark}[Interpretation] {\em Recall that $\mu_\mathcal{M}(\Omega)\subset
\Re(\widehat{\mathcal{M}}^d)$ and thus  $(\mu_\mathcal{M}(\Omega), T_<)$ is a topological space. 
Similarly,  $(\widehat{\mathcal{M}}, T_<)$ is a topological space (Section~\ref{S: F-Asymptotic
Numbers: Definitions and Examples}). With this in mind, let 
$\mathcal{C}^\infty(\mu_\mathcal{M}(\Omega), \widehat{\mathcal{M}})$ denote the space
of the $\mathcal{C}^\infty$-functions from $\mu_\mathcal{M}(\Omega)$ into
$\widehat{\mathcal{M}}$. The above theorem shows that
$\widehat{\mathcal{E}_\mathcal{M}}(\Omega)$ is isomorphic to
$\mathcal{C}^\infty(\mu_\mathcal{M}(\Omega), \widehat{\mathcal{M}})$. Based on this result
we shall sometimes identify a given asymptotic function with its graph and write simply
$\widehat{\mathcal{E}_\mathcal{M}}(\Omega)=\mathcal{C}^\infty(\mu_\mathcal{M}(\Omega),
\widehat{\mathcal{M}})$ or
$\widehat{\mathcal{E}_\mathcal{M}}(\Omega)\subset
\widehat{\mathcal{M}}\,^{\mu_\mathcal{M}(\Omega)}$ instead of the more precise
$\widehat{\mathcal{E}_\mathcal{M}}(\Omega)\embed
\widehat{\mathcal{M}}\,^{\mu_\mathcal{M}(\Omega)}$. We should note that
$\widehat{\mathcal{M}}\,^{\mu_\mathcal{M}(\Omega)}\setminus
\widehat{\mathcal{E}_\mathcal{M}}(\Omega)\not=\varnothing$.
}\end{remark}

\Proof  To show that the mapping is injective, observe that 
$\widehat{f}(\widehat{x})=0$ for all $\widehat{x}\in\mu_\mathcal{M}(\Omega)$ is equivalent to
$f(x)\in\mathcal{M}_0$ for all $\forall x\in\mu(\Omega)$. The latter implies
$f\in\mathcal{N}_\mathcal{M}(\Omega))$ by Theorem~\ref{T: A Simplification}. Thus
$\widehat{f}=0$ as required. The mapping preserves the addition because
$(\widehat{f}+\widehat{g})(\widehat{x})=\widehat{f}(\widehat{x})+\widehat{g}(\widehat{x})=
\widehat{f(x)+g(x)}$ and similarly for the multiplication. We turn to the preserving of the partial
differentiation. Let $x\in\mu(\Omega)$ and $f\in\mathcal{M}_\mathcal{M}(\Omega)$. In view of the
fact that every maximal field
$\mathbb{M}\in\mathcal{M}ax(\mathcal{M})$ (Definition~\ref{D: Maximal Fields}) is isomorphic to
$\widehat{\mathcal{M}}$ (Lemma~\ref{L: Isomorphic Fields}), it suffices
to show that for every $\varepsilon\in\mathbb{M}_+$ there exists $\delta\in\mathbb{M}_+$ such
that for every $h\in{^*\mathbb{R}^d}$ we have:
\begin{description}
\item{\bf (a)} $||h||<\delta$ implies $|f(x+h)-f(x)|<\varepsilon$.
\item{\bf (b)} $0< ||h||<\delta$ implies $\frac{|f(x+h)-f(x)-\nabla f(x)\cdot h|}{||h||}<\varepsilon$
\end{description}
We have to consider separately two differently cases: Suppose first, that
$\widehat{\mathcal{M}}$ has a base for the open neighborhoods of the zero of cardinality less than
$\kappa$. Since  $\widehat{\mathcal{M}}$ and $\mathbb{M}$ are isomorphic, it follows that there
exists a set  $\Gamma\subseteq\mathbb{M}_+$ of cardinality less than $\kappa$ such that the
collection of open intervals $(0, \gamma),
\gamma\in\Gamma$, is a base for the open neighborhoods of the zero in $\mathbb{M}_+$. 
Now, suppose (on the contrary) that (a) and (b) fail, i.e. there exists $\varepsilon\in\mathbb{M}_+$ such
that for every
$\delta\in \Gamma$ we have $X_\delta\not=\varnothing$ and $Y_\delta\not=\varnothing$, where 
\begin{align}
&X_\delta=\left\{h\in{^*\mathbb{R}^d} : ||h||<\delta \; \text{and}\;
|f(x+h)-f(x)|>\varepsilon\right\},\notag\\ 
&Y_\delta=\left\{h\in{^*\mathbb{R}^d} : 0< ||h||<\delta \;
\text{and}\;\frac{|f(x+h)-f(x)-\nabla f(x)\cdot h|}{||h||}>\varepsilon\right\}.\notag
\end{align}
We observe that the families $\{X_\delta\}_{\delta\in \Gamma}$ and $\{Y_\delta\}_{\delta\in
\Gamma}$ have the finite intersection properties. Thus there exist $h_1, h_2\in{^*\mathbb{R}^d}$ such
that
$h_1\in  X_\delta$ and $h_2\in  Y_\delta$ for all $\delta\in \Gamma$ by the Saturation Principle
(Theorem~\ref{T: Saturation Principle in *C}). It follows that
$||h_1||, ||h_1||\in\mathcal{M}_0$ and
$f(x+h_1)-f(x)\notin\mathcal{M}_0$ and $\frac{|f(x+h)-f(x)-\nabla f(x)\cdot
h|}{||h||}\notin\mathcal{M}_0$ by Theorem~\ref{T: Characterization} contradicting the result of
Lemma~\ref{L: F0}. This proves the preservation of the partial derivatives
$\partial^\alpha$ for
$|\alpha|\leq 1$. The generalization of the result to all multi-indices $\alpha$ follow by induction.
Suppose  now that $\widehat{\mathcal{M}}$ does not have a base for the open neighborhoods of the
zero of cardinality less than $\kappa$. In this case we have
$\widehat{\mathcal{E}_\mathcal{M}}(\Omega)={^*\mathcal{E}}(\Omega)$ (Example~\ref{Ex: The
Case F=*C}). Thus the preservation of the partial differentiation follows by default since
${^*\mathcal{E}}(\Omega)$ consists exactly of the $\mathcal{C}^\infty$-functions from $^*\Omega$
into $^*\mathbb{C}$.
$\blacktriangle$
\begin{theorem}[Fundamental Theorem]\label{T: Fundamental Theorem} Let $\Omega$ be an arcwise
connected open set of $\mathbb{R}^d$ and let $f\in\mathcal{M}_\mathcal{M}(\Omega)$. Then the
following are equivalent:
\begin{description}
\item{\bf (i)}\quad $(\exists\,
\widehat{c}\in\widehat{\mathcal{M}})(\forall\,\widehat{x}\in\mu_\mathcal{M}(\Omega))(\widehat{f}(\widehat{x})=
\widehat{c})$.
\item{\bf (ii)}\; $(\exists\,
c\in\mathcal{M})(\forall\,x\in\mu(\Omega))(f(x)-c\in\mathcal{M}_0)$.
\item{\bf (iii)} $(\forall\,x\in\mu(\Omega))(||\nabla f(x)||\in\mathcal{M}_0)$.
\item{\bf (iv)} $(\forall\,\widehat{x}\in\mu_\mathcal{M}(\Omega))(\nabla\widehat{f}(\widehat{x})=0)$.
\item{\bf (v)}\, $\nabla\widehat{f}=0$ in $\widehat{\mathcal{E}_\mathcal{M}}(\Omega)$.
\end{description}
\end{theorem}
\Proof  (i)$\Leftrightarrow$(ii), (iii)$\Leftrightarrow$(iv) and (iv)$\Leftrightarrow$(v) follow directly from
Theorem~\ref{T: Differential Ring Embedding}.

	(ii)$\Rightarrow$(iii): Suppose that $x\in\mu(\Omega)$. If $\nabla f(x)=0$, there is nothing to prove.
Suppose that $\nabla f(x)\not=0$ and let $h\in\mathcal{I}(\mathbb{M}^d)$ be an infinitesimal vector in the
direction of $\nabla f(x)$. By the Mean Value Theorem applied by Transfer Principle (Theorem~\ref{T:
Transfer Principle}), we have
\[
\nabla f(x)\cdot h=f(x+h)-f(x)-\frac{1}{2}\sum_{|\alpha|=2}\partial^\alpha f(x+\theta h)\, h^\alpha,
\]
for some $\theta\in{^*\mathbb{R}},\; 0<\theta<1$. We have
$\left|\frac{1}{2}\sum_{|\alpha|=2}\partial^\alpha f(x+\theta h)\right|\leq\delta$ for some
$\delta\in\mathbb{M}_+$ by Theorem~\ref{T: Characterization} since $x+\theta h\in\mu(\Omega)$ and
$f\in\mathcal{M}_\mathcal{M}(\Omega)$ by assumption. Also $|\nabla f(x)\cdot h|=||\nabla f(x)||\, ||h||$ by
the choice of the direction of $h$. Thus
\[
||\nabla f(x)||\leq \left(\frac{f(x+h)-f(x)}{||h||^2}+\delta\right)\, ||h||,
\]
Observe that $f(x+h)-f(x)\in\mathcal{M}_0$ by assumption since $x+h\in\mu(\Omega)$. Thus
$\frac{f(x+h)-f(x)}{||h||^2}+\delta\in\mathbb{M}_+$. Consequently, there exists
$M\in\mathbb{M}_+$ such that the internal set 
\[
\mathcal{A}=\left\{\, ||h||\; : \; h\in{^*\mathbb{R}^d},\;  \frac{\nabla f(x)}{||\nabla
f(x)||}=\frac{h}{||h||},\;  ||\nabla f (x)||\leq M\, ||h||\; \right\},
\]
contains $\mathcal{I}(\mathbb{M}_+)$. Thus $\mathcal{A}$ contains arbitrarily small numbers in
$\mathcal{M}\setminus\mathcal{M}_0$ since $\mathbb{M}_+\subset\mathcal{M}\setminus\mathcal{M}_0$. It
follows that $\mathcal{A}$ contains arbitrarily large numbers $\mathcal{M}_0$ by the Underflow of
$\mathcal{M}\setminus\mathcal{M}_0$ (Theorem~\ref{T: Spilling Principles}). Thus there
exists
$h\in{^*\mathbb{R}^d}$ such that $ ||\nabla f (x)||\leq M\, ||h||$ and
$||h||\in\mathcal{M}_0$. It follows that $ ||\nabla f (x)||\in\mathcal{M}_0$ (as required) since
$\mathcal{M}_0$ is an ideal in 
$\mathcal{M}$.

	(ii)$\Leftarrow$(iii): Suppose that $x, y\in\mu(\Omega)$. Since $\Omega$ is arcwise connected by
assumption, it follows that
$^*\Omega$ is $*$-arcwise connected by Transfer Principle (Theorem~\ref{T: Transfer Principle}). Thus there
exists a
$*$-continuous curve $L\subset\mu(\Omega)$ which connects $x$ and $y$. We have 
\[
f(x)-f(y)=\int_{L}\nabla f(t)\cdot dl,
\]
(again, by Transfer Principle).  It follows that 
\[
f(x)-f(y)= \nabla f(t)\cdot (x-y),
\]
for some $t\in {L}$ by the Mean Value Theorem (and Transfer Principle).  Thus $|f(x)-f(y)|\leq ||\nabla
f(t)||\, ||x-y||\in\mathcal{M}_0,$ since (as before) $\mathcal{M}_0$ is an ideal in $\mathcal{M}$ and we
have $||\nabla f(t)||\in\mathcal{M}_0$ by assumption and
$||x-y||\in\mathcal{M}(^*\mathbb{R})\subset\mathcal{M}$. Let
$c=f(y)$ for some (any) $y\in\mu(\Omega)$. The result is
$f(x)-c\in\mathcal{M}_0$ for all $x\in\mu(\Omega)$ as required.
$\blacktriangle$
\begin{corollary}[Constant Functions]\label{C: Constant Functions} Let $\Omega$ be an arcwise
connected open set of $\mathbb{R}^d$. Then
\begin{equation}\label{E: Ring of Scalars}
\widehat{\mathcal{M}}=
\left\{ \widehat{f}\in \widehat{\mathcal{E}}_\mathcal{M}(\Omega)\mid \nabla
\widehat{f}=0\right\},
\end{equation}
In particular,
\begin{equation}\label{E: Ring of Scalars}
\widehat{\mathcal{M}}=
\left\{ \widehat{f}\in \widehat{\mathcal{E}}_\mathcal{M}(\mathbb{R}^d)\mid \nabla
\widehat{f}=0\right\}.
\end{equation}
\end{corollary}
\Proof The inclusion $\widehat{\mathcal{M}}\subseteq
\left\{ \widehat{f}\in \widehat{\mathcal{E}}_\mathcal{M}(\Omega)\mid \nabla
\widehat{f}=0\right\}$ follows directly from
Theorem~\ref{T: Fundamental Theorem} in view of the embedding
$\widehat{\mathcal{M}}\embed\widehat{\mathcal{E}}_\mathcal{M}(\Omega)$ (through constant
functions) descussed in part~(v) of Theorem~\ref{T: Some Basic Results}. The
inclusion $\left\{ \widehat{f}\in
\widehat{\mathcal{E}}_\mathcal{M}(\Omega)\mid \nabla
\widehat{f}=0\right\}\newline
\subseteq\widehat{\mathcal{M}}$ follows also from
Theorem~\ref{T: Fundamental Theorem} and the identification of the constant functions in 
$\widehat{\mathcal{E}}_\mathcal{M}(\Omega)$ with their values. $\blacktriangle$


\section{Local Properties of Asymptotic Functions}\label{S: Local Properties of Asymptotic Functions}

\begin{definition}[Restriction]\label{D: Restriction}  Let $\Omega,
\mathcal{O}$ be two open sets of $\mathbb{R}^d$ such that
$\mathcal{O}\subseteq\Omega$. Let
$\widehat{f}\in\widehat{\mathcal{E}}_\mathcal{M}(\Omega)$. We define the {\bf restriction}
$\widehat{f}\upharpoonright\mathcal{O}$ of
$\widehat{f}$ on $\mathcal{O}$ by the formula
\[
\widehat{f}\upharpoonright\mathcal{O}=\widehat{f\, |{^*\mathcal{O}}},
\]
where $^*\mathcal{O}$ is the non-standard extension of $\mathcal{O}$ and $f\, |{^*\mathcal{O}}$ is
the usual (pointwise) restriction of
$f$ on $^*\mathcal{O}$ (Section~\ref{S: A Non-Standard Analysis: The General Theory}).
\end{definition}

	The above definition is justified by the following result.
\begin{lemma}[Justification]\label{L: Justification of Restriction} Let $f,
g\in\mathcal{M}_\mathcal{M}(\Omega)$ and $f-g\in\mathcal{N}_\mathcal{M}(\Omega)$. Then
$f\, |{^*\mathcal{O}}-g\, |{^*\mathcal{O}}\in\mathcal{N}_\mathcal{M}(\mathcal{O})$.
\end{lemma}
\Proof For every  $x\in\mu(\mathcal{O})$ we have $f(x)-g(x)\in\mathcal{M}_0$ by assumption since 
$\mu(\mathcal{O})\subseteq\mu(\Omega)$. It follows that
$f-g\in\mathcal{N}_\mathcal{M}(\mathcal{O})$ as required by Theorem~\ref{T: A Simplification}. 
$\blacktriangle$

 If $S$ is a set,
then $\mathcal{P}_\omega(S)$ denotes the set of all finite subsets of $S$. In particular,
$\mathcal{P}_\omega(\mathbb{N})$ denotes the set of the finite sets of natural numbers. The
elements of the non-standard extension $^*\mathcal{P}_\omega(\mathbb{N})$ are called
\textbf{hyperfinite sets}. They are, in general, infinite sets which are in one-to-one correspondence
with sets of the form $\{1, 2,\dots, \nu\}$ for some $\nu\in{^*\mathbb{N}}$. Let
$\mathcal{P}_\omega(\mathbb{N})^\Omega$ be the set of the functions of the form $F:
\Omega\to\mathcal{P}_\omega(\mathbb{N})$ and, similarly, let
$^*\mathcal{P}_\omega(\mathbb{N})^{^*\Omega}$ denote the set of the functions of the form $F:
{^*\Omega}\to{^*\mathcal{P}_\omega}(\mathbb{N})$. For the non-standard extension
$^*\!\left(\mathcal{P}_\omega(\mathbb{N})^\Omega\right)$ we have a strict inclusion
$^*\!\left(\mathcal{P}_\omega(\mathbb{N})^\Omega\right)
\subsetneqq{^*\mathcal{P}_\omega(\mathbb{N})^{^*\Omega}}$. We say that
$^*\!\left(\mathcal{P}_\omega(\mathbb{N})^\Omega\right)$ consists exactly of the {\em internal}
functions in ${^*\mathcal{P}_\omega(\mathbb{N})^{^*\Omega}}$. We should note that a fluent
knowledge on internal hyperfinite functions is not necessary for the understanding of what follows.

	We denote by  $\mathcal{T}_d$ the usual topology on $\mathbb{R}^d$ and by
$(\mathbb{R}^d,\mathcal{T}_d)$ the corresponding topological space. 
\begin{theorem} The collection
$\{\widehat{\mathcal{E}}_\mathcal{M}(\Omega)\}_{\Omega\in{\mathcal{T}_d}}$ 
is a {\bf sheaf of differential rings} on $(\mathbb{R}^d, \mathcal{T}_d)$ under the restriction
$\rest$ in the sense that:
\begin{description}
\item{\bf (i)}\quad
$(\forall\Omega\in\mathcal{T}_d)(\forall\widehat{f}\in\widehat{\mathcal{E}}_\mathcal{M}(\Omega))(\widehat{f}
\rest\Omega=\widehat{f})$.
\item{\bf (ii)}\;
$(\forall\Omega_1,
\Omega_2,
\Omega\in\mathcal{T}_d)(\forall\widehat{f}\in
\widehat{\mathcal{E}}_\mathcal{M}(\Omega))(\Omega_1\subseteq\Omega_2\subseteq\Omega$
implies $(\widehat{f}\rest\Omega_2)\rest\Omega_1=\widehat{f}\rest\Omega_1$.
\end{description}

	Let $\Omega=\bigcup_{\lambda\in\Lambda}\Omega_\lambda$ be an \textbf{open covering}
of $\Omega\in\mathcal{T}_d$ (for some  index set
$\Lambda$ and some open sets $\Omega_\lambda\in\mathcal{T}_d$).  Then:
\begin{description}

\item{\bf (iii)} $(\forall\lambda\in\Lambda)(\widehat{f}\rest\Omega_\lambda=0)$ implies
$\widehat{f}=0$.

\item{\bf (iv)} Let $\{\widehat{f}_\lambda\}_{\lambda\in\Lambda}$, 
$\widehat{f}_\lambda\in\widehat{\mathcal{E}}_\mathcal{M}(\Omega_\lambda)$, be a
\textbf{coherent family} of asymptotic functions in the sense that it satisfies the compatibility
condition
\[
(\forall
\lambda_1,
\lambda_2\in\Lambda)\left[\Omega_{\lambda_1}\cap\Omega_{\lambda_2}\not=
\varnothing\Rightarrow 
\widehat{f}_{\lambda_1}\rest(\Omega_{\lambda_1}\cap\Omega_{\lambda_2})=
\widehat{f}_{\lambda_2}\rest(\Omega_{\lambda_1}\cap\Omega_{\lambda_2})\right].
\]
Then there exists $\widehat{f}\in\widehat{\mathcal{E}}_\mathcal{M}(\Omega)$ such that
$\widehat{f}\rest\Omega_\lambda=\widehat{f}_\lambda$ for all $\lambda\in\Lambda$.

	\item{\bf (v)} The restriction $\rest$ agrees with the {\bf differential ring operations} in the sense
that
\item $(\forall\Omega, \mathcal{O}\in\mathcal{T}_d)(\forall
\widehat{f}\in{\widehat{\mathcal{E}_\mathcal{M}}}(\Omega))(\forall\alpha\in\mathbb{N}_0^d)\left(\mathcal{O}\subseteq\Omega
\Rightarrow (\partial^\alpha \widehat{f})
\rest\mathcal{O}=\partial^\alpha (\widehat{f}\rest\mathcal{O})\right)$.
\item $(\forall\Omega, \mathcal{O}\in\mathcal{T}_d)(\forall
\widehat{f},\,
\widehat{g}\in{\widehat{\mathcal{E}_\mathcal{M}}}(\Omega))\left(\mathcal{O}\subseteq\Omega
\Rightarrow (\widehat{f}+\widehat{g})
\rest\mathcal{O}=\widehat{f}\rest\mathcal{O}+\widehat{g}\rest\mathcal{O}\right)$.

\item $(\forall\Omega, \mathcal{O}\in\mathcal{T}_d)(\forall
\widehat{f},\,
\widehat{g}\in{\widehat{\mathcal{E}_\mathcal{M}}}(\Omega))\left(\mathcal{O}\subseteq\Omega
\Rightarrow (\widehat{f}\; \widehat{g})
\rest\mathcal{O}=(\widehat{f}\rest\mathcal{O})(\widehat{g}\rest\mathcal{O})\right)$.
\end{description}
\end{theorem}
\begin{remark}[Sheaf Terminology]\label{R: Sheaf Terminology} {\em A collection
$\{\widehat{\mathcal{E}}_\mathcal{M}(\Omega)\}_{\Omega\in{\mathcal{T}_d}}$ which satisfies the
properties (i) and (ii) is called {\bf presheaf} on $(\mathbb{R}^d, \mathcal{T}_d)$. A presheaf
$\{\widehat{\mathcal{E}}_\mathcal{M}(\Omega)\}_{\Omega\in{\mathcal{T}_d}}$ which satisfies (iii)
and (iv) is called a {\bf sheaf} on $(\mathbb{R}^d, \mathcal{T}_d)$. A sheaf is called a {\bf differential
ring sheaf} it it satisfies (v). For the relevant
terminology we refer to A. Kaneko~\cite{aKan88}.  }\end{remark}
\Proof (i)  $\widehat{f}\rest\Omega=\widehat{f\, |\, {^*\Omega}}=\widehat{f}$\,  because
$^*\Omega$ is the domain of $f$.

	(ii) $(\widehat{f}\rest\Omega_2)\rest\Omega_1=(\widehat{f\,
|{^*\Omega_2}})\rest\Omega_1=
\widehat{(f\, |{^*\Omega_2})\, | {^*\Omega_1}}=\widehat{f\, |\, {^*\Omega_1}}=f\rest\Omega_1$
(as required)  since $^*\Omega_1\subseteq{^*\Omega_2}\subseteq{^*\Omega}$ and $\Omega$ is the
domain of $f$. 

	(iii) Suppose that $\lambda\in\Lambda$. We have $\widehat{f}\rest
\Omega_\lambda=0$ \ifff $\widehat{f\, |\, {^*\Omega_\lambda}}=0$ \ifff $ f\, |\,
{^*\Omega_\lambda}\in\mathcal{N}(\Omega_\lambda)$ \ifff $(\forall
x\in\mu(\Omega_\lambda)(f(x)\in\mathcal{M}_0)$ by Theorem~\ref{T: A Simplification}. On the
other hand $(\forall \lambda\in\Lambda)(\forall x\in\mu(\Omega_\lambda)(f(x)\in\mathcal{M}_0)$
\ifff $(\forall x\in\mu(\Omega)(f(x)\in\mathcal{M}_0)$ since
$\bigcup_{\lambda\in\Lambda}\, \mu(\Omega_\lambda)=\mu(\Omega)$. Thus it follows $ f\, |\,
{^*\Omega}\in\mathcal{N}(\Omega)$ by Theorem~\ref{T: A Simplification} implying
$\widehat{f}\rest\Omega=0$ (as required).

	(iv) Let $\Omega=\bigcup_{n=1}^\infty\mathcal{O}_n$ be a {\em locally finite countable
covering} of $\Omega$ which is a {\em refinement} of
$\Omega=\bigcup_{\lambda\in\Lambda}\Omega_\lambda$ in the sense that
\begin{description}
\item{\bf (a)} $\mathcal{O}_n\in\mathcal{T}_d$ and
$\overline{\mathcal{O}}_n\subset\subset\Omega$.
\item{\bf (b)} For every $K\subset\subset\Omega$ the set $\{n\in\mathbb{N}\mid
K\cap\mathcal{O}_n\not=\varnothing\}$ is finite.
\item{\bf (c)} There exists a sequence $\lambda\in\Lambda^\mathbb{N}$ such that 
$\overline{\mathcal{O}}_n\subset\Omega_{\lambda(n)}$ for all $n\in\mathbb{N}$.
\end{description}
 Let  $\{\varphi_n\}_{n\in\mathbb{N}}$ be a {\em smooth partition of unity}
subordinate to $\{\mathcal{O}_n\}_{n\in\mathbb{N}}$ in the sense that: 
\begin{description}
\item{\bf  (d)} $\varphi_n\in\mathcal{D}(\mathcal{O}_n)$ for all $n\in\mathbb{N}$.
\item{\bf  (e)} $0\leq\varphi_n(x)\leq 1$ for all $x\in\mathcal{O}_n$.
 \item{\bf  (f)} $1=\sum_{n=1}^\infty\, \varphi_n(x)$ for all $x\in\Omega$.
\end{description}
We recall that {\em every open
covering has a locally finite countable covering refinement} and that {\em every locally finite
countable covering has a smooth partition of unity} (A. Kaneko~\cite{aKan88}). Notice that
there exists
$F\in{\mathcal{P}_\omega(\mathbb{N})^\Omega}$ such that 
\[
\sum_{n\in F(x)}\, \varphi_n(x)=1,
\]
for all  $x\in\Omega$ and the function $f:\Omega\to\mathbb{C}$, defined by the formula 
\[
f(x)=\sum_{n\in F(x)}\, {\varphi_n}(x)\, f_{\lambda(n)}(x),
\]
is in $\mathcal{E}(\Omega)$.  Thus there exists $
H\in{^*\!\left(\mathcal{P}_\omega(\mathbb{N})^\Omega\right)}$ such that
\[
\sum_{n\in H(x)}\, {^*\varphi_n(x)}=1,
\]
for all $x\in{^*\Omega}$, by Transfer Principle (Theorem~\ref{T: Transfer Principle}) which
implies (trivially)
\begin{equation}\label{E: flambda}
f_\lambda(x)=f_{\lambda}(x)\sum_{n\in H(x)}\, {^*\varphi_n}(x),
\end{equation}
for all $x\in{^*\Omega}$. We define the function $f: {^*\Omega}\to{^*\mathbb{C}}$
by the formula
\begin{equation}\label{E: f}
f(x)=\sum_{n\in H(x)}\, {^*\varphi_n}(x)\, f_{\lambda(n)}(x).
\end{equation}
This is the function  we are looking for. Indeed, we have $f\in{^*\mathcal{E}}(\Omega)$ by Transfer
Principle because $f$ is a hyperfinite sum (\ref{E: f}) of functions in
${^*\mathcal{E}}(\Omega)$. Also
$f\in\mathcal{M}_\mathcal{M}(\Omega)$ (by Transfer Principle again) because $f$ is a
hyperfinite sum (\ref{E: f}) of functions in
$\mathcal{M}_\mathcal{M}(\Omega)$. Suppose that
$x\in\mu(\Omega_\lambda)$. After subtracting (\ref{E: flambda}) from (\ref{E: f}) we obtain:
\[
f(x)-f_\lambda(x)=\sum_{n\in H(x)}\, {^*\varphi_n}(x)\,
\left(f_{\lambda(n)}(x)-f_\lambda(x)\right).
\]
This formula implies $f(x)-f_\lambda(x)\in\mathcal{M}_0$ since $^*\varphi_n(x)$ is a finite
number and 
$f_{\lambda(n)}(x)-f_\lambda(x)\in\mathcal{M}_0$ by the compatibility condition. On the
other hand, $f(x)-f_\lambda(x)\in\mathcal{M}_0$ implies 
$f - f_\lambda\in\mathcal{N}_\mathcal{M}(\Omega_\lambda)$ by Theorem~\ref{T: A
Simplification}.  Thus $\widehat{f}\rest\Omega_\lambda$$=\widehat{f\, |\,
{^*\Omega_\lambda}}=\widehat{f_\lambda\, |\,
{^*\Omega_\lambda}}=\widehat{f}_\lambda$ (as required) by Lemma~\ref{L:
Justification of Restriction} since
$^*\Omega_\lambda\subseteq{^*\Omega}$ and $^*\Omega_\lambda$ is the domain of
$f_\lambda$.

	(v) We have $(\partial^\alpha \widehat{f})\rest\mathcal{O}=\widehat{\partial^\alpha
f}\rest\mathcal{O}
=\widehat{\partial^\alpha f\, |\, {^*\mathcal{O}}}=
\partial^\alpha (\widehat{f\, |\, {^*\mathcal{O}}})$=
$\partial^\alpha (\widehat{f}\rest\mathcal{O})$ as required. The verification of the sum and
multiplication is similar and we leave it to the reader.

$\blacktriangle$

	The above result justifies the following definition. 

\begin{definition}[Standard Support]\label{D: Standard Support}{\em  Let $\Omega$ be two
open sets of
$\mathbb{R}^d$ and let $\widehat{f}\in\widehat{\mathcal{E}}_\mathcal{M}(\Omega)$.
Let $\mathcal{O}$ be the maximal open subset of $\mathbb{R}^d$  such that
$\widehat{f}\restriction\mathcal{O}=0$ in
$\widehat{\mathcal{E}}_\mathcal{M}(\mathcal{O})$. Then the set
$\supp(\widehat{f})={\mathbb{R}^d}\setminus\mathcal{O}$ is called {\bf standard support}
(or simply {\bf support} if no confusion can arise) of $\widehat{f}$.
}\end{definition}

	Te next result follows immediately from the above definition. 

\begin{proposition}\label{P: Standard Support} Every asymptotic function
$\widehat{f}\in\widehat{\mathcal{E}_\mathcal{M}}(\Omega)$ has a (standard) support\, 
$\supp(\widehat{f})$ which is a \textbf{closed set of}\;  $\Omega$ in the usual topology on
$\mathbb{R}^d$. 
\end{proposition}
\begin{theorem}[Usual Support]\label{T: Usual Support} The embedding $f\to\widehat{^*f}$
from $\mathcal{E}(\Omega)$ into $\widehat{\mathcal{E}_\mathcal{M}}(\Omega)$ is a
sheaf homomorphism in the sense that $\widehat{{^*(f\, |\, \mathcal{O})}}=
\widehat{^*f}\restriction\mathcal{O}$. Consequently, $\supp (f)=\supp(\widehat{^*f})$, where
$\supp(f)$ stands for the usual support of $f$ in $\mathcal{E}(\Omega)$.
\end{theorem}
\Proof We have $^*f\, |{^*\mathcal{O}}={^*(f\, |\, \mathcal{O})}$ by Transfer Principle
(Theorem~\ref{T: Transfer Principle}). Thus
$\widehat{{^*(f\, |\, \mathcal{O})}}=\widehat{^*f\, |{^*\mathcal{O}}}=
\widehat{^*f}\restriction\mathcal{O}$
as required. 
\section{A Canonical form of the Algebras $\widehat{\mathcal{E}}_\mathcal{M}(\Omega)$}

	So far we constructed the algebra $\widehat{\mathcal{E}}_\mathcal{M}(\Omega)$ of asymptotic
functions by the {\bf following scheme:}
\begin{enumerate}

\item We choose a convex subring $\mathcal{M}$ of
$^*\mathbb{C}$ (Section~\ref{S: F-Asymptotic Functions}). 
\item  We construct the ideal
$\mathcal{M}_0$ and the algebraically closed field $\widehat{\mathcal{M}}$ (Section~\ref{S:
F-Asymptotic Numbers: A Basic Theory}). Notice that
$\widehat{\mathcal{M}}$ can be embedded as a subfield of $^*\mathbb{C}$ (which is important
for what follows) and the image of $\widehat{\mathcal{M}}$ into $^*\mathbb{C}$ under this
embedding is a maximal field
$\mathbb{M}\in\mathcal{M}ax(\mathcal{M})$ (Definition~\ref{D: Maximal Fields}).
\item  We define $\mathcal{M}_\mathcal{M}(\Omega)$ and
$\mathcal{N}_\mathcal{M}(\Omega)$ and the algebra
$\widehat{\mathcal{E}}_\mathcal{M}(\Omega)=\mathcal{M}_\mathcal{M}(\Omega)/\mathcal{N}_\mathcal{M}(\Omega)$
over the field of scalars $\widehat{\mathcal{M}}$ (Section~\ref{S: F-Asymptotic Functions})
\end{enumerate}

	We shall presented an \textbf{alternative construction} of
$\widehat{\mathcal{E}}_\mathcal{M}(\Omega)$: We start from a given (already
chosen or constructed) an algebraically closed subfield $\mathbb{M}$ of $^*\mathbb{C}$ and
then we define the algebra $\widehat{\mathcal{E}}_\mathcal{M}(\Omega)$ directly from
$\mathbb{M}$. The connection between the two construction is given by the formula:
\[
\mathcal{M}=\{z\in{^*\mathbb{C}} \mid (\exists \zeta\in\mathbb{M})(|z|\leq |\zeta|)\}.
\]

\begin{definition}[Asymptotic Functions Generated by a Field]\label{D: Asymptotic Functions
Generated by a Field}{\em  Let $\mathbb{M}$ be an algebraically closed subfield of
$^*\mathbb{C}$. We let 
\begin{align}
&\mathbb{M}(\Omega)=\{f\in{^*\mathcal{E}(\Omega)}\mid
(\forall\alpha\in\mathbb{N}_0^d)(\forall x\in\mu(\Omega)(\partial^\alpha
f(x)\in\mathbb{M})\},\notag\\
&\mathbb{M}_0(\Omega)=\{f\in{^*\mathcal{E}(\Omega)}\mid
(\forall\alpha\in\mathbb{N}_0^d)(\forall x\in\mu(\Omega)(\partial^\alpha
f(x)=0)\},\notag
\end{align}
and let  $\widehat{\mathbb{M}}(\Omega)=\mathbb{M}(\Omega)/\mathbb{M}_0(\Omega)$ be
the corresponding factor ring. We say that $\widehat{\mathbb{M}}(\Omega)$ {\bf is generated
by the field} $\mathbb{M}$.
}\end{definition}
\begin{theorem} Let $\mathbb{M}$ be an algebraically closed subfield of
$^*\mathbb{C}$. Then $\mathbb{M}(\Omega)$ is a differential ring and
$\mathbb{M}_0(\Omega)$ is a differential ideal in
$\mathbb{M}(\Omega)$ and we also have
\[
\mathbb{M}_0(\Omega)=\{f\in{\mathbb{M}(\Omega)}\mid
(\forall x\in\mu(\Omega)(f(x)=0)\}.
\]
Consequently, $\widehat{\mathbb{M}}(\Omega)$ is both a differential
ring and a differential  algebra over the field $\mathbb{M}$.
\end{theorem}  
\Proof The statement about $\mathbb{M}(\Omega)$, $\mathbb{M}_0(\Omega)$ and
$\widehat{\mathbb{M}}(\Omega)$ follows directly from the above definition. The proof of the
formula for $\mathbb{M}_0(\Omega)$ is almost identical to the proof of Theorem~\ref{T: A
Simplification} and leave it to the reader. $\blacktriangle$

\begin{theorem}[An Isomorphism]\label{T: An Isomorphism} Let $\mathcal{M}$ be a convex 
subring of $^*\mathbb{C}$ and $\mathbb{M}\in\mathcal{M}ax(\mathcal{M})$ be a maximal field within
$\mathcal{M}$ (Definition~\ref{D: Maximal Fields}). Then $\widehat{\mathcal{E}_\mathcal{M}}(\Omega)$
and $\widehat{\mathbb{M}}(\Omega)$ are isomorphic differential algebras over the field
${\mathbb{M}}$ under pointwise characterization of 
$\widehat{\mathcal{E}_\mathcal{M}}(\Omega)$:
\[
\widehat{f}\in\widehat{\mathcal{E}_\mathcal{M}}(\Omega)\to
f\in\widehat{\mathcal{M}}^{\mu_\mathcal{M}(\Omega)},
\]
(Section~\ref{S: Pointwise Values and Fundamental Theorem}). 
\end{theorem}

\Proof We have  $\widehat{\mathcal{M}}=\widehat{\mathbb{M}}$ by part~(i) of Theorem~\ref{T: Field of
Representatives} and also we have 
\begin{equation}\label{E: M-Monad}
\mu_\mathcal{M}(\Omega)=\{r+dx \mid r\in \Omega,\,  dx\in\Re(\widehat{\mathbb{M}}^d),\; 
||dx||\approx 0 \}.
\end{equation}
(compare with (\ref{E: F-Monad}) in Section~\ref{S: Pointwise Values and
Fundamental Theorem}). Thus
$\widehat{\mathcal{M}}^{\mu_\mathcal{M}(\Omega)}=
\widehat{\mathbb{M}}^{\mu_\mathbb{M}(\Omega)}$. On the other hand $\widehat{\mathbb{M}}$ and
$\mathbb{M}$ are field isomorphic by part~(ii) of Theorem~\ref{T: Field of Representatives}. Thus
$\widehat{\mathbb{M}}^{\mu_\mathbb{M}(\Omega)}$ and
$\mathbb{M}^{\mu_\mathbb{M}(\Omega)}$ are ring isomorphic. The theorem is complete.
$\blacktriangle$
\begin{corollary} Let $\mathcal{M}$ be a convex 
subring of $^*\mathbb{C}$ and $\mathbb{M}_1,\, \mathbb{M}_2 \in\mathcal{M}ax(\mathcal{M})$ be
two maximal fields. Then $\widehat{\mathbb{M}_1}(\Omega)$ and $\widehat{\mathbb{M}_2}(\Omega)$
are isomorphic differential algebras over the field $\widehat{\mathcal{M}}$.
\end{corollary}

\begin{remark}[A Canonical Form]\label{R: A Canonical Form}{\em  Based on the above results
we shall sometimes identify notationally the algebras of asymptotic functions 
$\widehat{\mathcal{E}_\mathcal{M}}(\Omega)$ and $\widehat{\mathbb{M}}(\Omega)$ writing
simply
\[
\widehat{\mathcal{E}_\mathcal{M}}(\Omega)=\widehat{\mathbb{M}}(\Omega). 
\]
We say that $\widehat{\mathbb{M}}(\Omega)$ is a \textbf{canonical form of the algebra}
$\widehat{\mathcal{E}_\mathcal{M}}(\Omega)$. We should mention that the theory of
$\widehat{\mathbb{M}}(\Omega)$ is somewhat simpler  and more elegant than the theory of
$\widehat{\mathcal{E}_\mathcal{M}}(\Omega)$. However,
$\widehat{\mathcal{E}_\mathcal{M}}(\Omega)$ is more easily supported by examples
because one can more easily produce examples of convex subrings $\mathcal{M}$ of
$^*\mathbb{C}$ rather then to produce examples of algebraically closed subfileds
$\mathbb{M}$ of
$^*\mathbb{C}$ (see the end of Section~\ref{S: F-Asymptotic Functions}).
}\end{remark}
\begin{example}[Levi-Civita Field]\label{Ex: Levi-Civita Field}{\em Let $\rho$ be a positive
infinitesimal in
${^*\mathbb{R}_+}$ and let 
$\mathbb{M}=\mathbb{C}\bra\rho\ket$ denote the field of the T. Levi-Civita~\cite{tLC} power
series with complex coefficients, i.e. series of the form 
\[
\sum_{n=0}^\infty\, a_n\, \rho^{r_n},
\]
where $(a_n)$ is a sequence in $\mathbb{C}$ and 
$(r_n)$ is a strictly increasing sequence in $\mathbb{R}$ such that
$\lim_{n\to\infty}a_n=\infty$ (we shall abbreviate all these as $r_0<r_1<r_2<\dots\to\infty$).
We should mention that $\mathbb{C}\bra\rho\ket$ is an algebraically closed field (as
required in the above definition). Also the Levi-Civita series are convergent in
$^*\mathbb{C}$ under the valuation norm $||\cdot||_\rho$ . Let
$\widehat{\mathbb{M}}(\Omega)=\widehat{\mathbb{C}\bra\rho\ket}(\Omega)$ be the
algebra of generalized functions generated by the field $\mathbb{C}\bra\rho\ket$. Then every
$\widehat{f}\in\widehat{\mathbb{C}\bra\rho\ket}(\Omega)$ can be presented in the form
\[
\widehat{f}(x)=\sum_{n=0}^\infty\, \widehat{a_n}(x)\, \rho^{r_n}
\]
for every $x\in\mu_\mathcal{M}(\Omega)$, where  $\widehat{a_n}:
\mu_\mathcal{M}(\Omega)\to\mathbb{C}$. }\end{example}

\newpage


\section{Convolution in Non-Standard Setting}\label{S: Convolution in Non-Standard Setting}
\begin{definition}[Convolution]\label{D:Convolution} 
\begin{description}
\item{\bf (i)} Let
$T\in{\mathcal{D}}^\prime(\Omega)$ and let
$T:
\mathcal{D}(\Omega)\to {\mathbb{C}}$ be the corresponding mapping. We define the {\bf
non-standard extension} 
$^*T: {^*\mathcal{D}}(\Omega)\to{^*\mathbb{C}}$  of $T$ by the formula 
\[
\left< {^*T}, \bra\varphi_i\ket\right>=\left< \bra T, \varphi_i\ket\right>,
\]
where $\bra\varphi_i\ket\in{^*\mathcal{D}}(\Omega)$.

\item{\bf (ii)} Let $T\in\mathcal{E}^\prime(\Omega)$ and
$\bra D_i\ket\in{^*\mathcal{D}}(\mathbb{R}^d)$. We define the {\bf convolution} between $^*T$ and
$\bra D_i\ket$ by the formula 
\[
^*T\star \bra D_i\ket=\bra T\star D_i\ket,
\]
where $T\star D_i$ is the usual convolution between $T$ and $D_i$ in the sense of distribution
theory (i.e. $\bra T(\xi), D_i(x-\xi)\ket$ for every $x\in\Omega$ and every $i\in\mathcal{I}$).
\end{description}
\end{definition}

\begin{lemma} For every $T\in\mathcal{E}^\prime(\Omega)$ and every
$D\in{^*\mathcal{D}}(\mathbb{R}^d)$ we have $^*T\star D\in{^*\mathcal{E}}(\Omega)$.  
\end{lemma}


\section{Schwartz Distributions in $^\rho\mathcal{E}(\Omega)$}

	If $f\in\mathcal{L}_{loc}^1(\Omega)$, we denote by $T_f\in\mathcal{D}^\prime(\Omega)$ the
Schwartz distribution with kernel $f$, i.e.
\[
\bra T_f, \varphi\ket=\int_\Omega f(x)\varphi(x)\, dx,
\]
for all $\varphi\in\mathcal{D}(\Omega)$. Recall that
$\mathcal{E}(\Omega)$ is a differential subring of ${^\rho\mathcal{E}}(\Omega)$ under the embedding
\[
\mathcal{E}(\Omega)\hookrightarrow{^\rho\mathcal{E}}(\Omega),
\]
defined by the mapping $f\to \widehat{^*f}$, where $^*f$ is the non-standard extension of $f$
(i.e. $^*f=\bra f_i\ket$, $f_i=f$ for all $i\in\mathcal{I}$) and $\widehat{^*f}$ stands for the
corresponding equivalence class (see the beginning of Section~\ref{S: Algebras of Generalized
Functions: A General Theory}).
\begin{theorem}[Existence of an Embedding]\label{T: Existence of an Embedding} There exists an
embedding $\Sigma_\Omega:
\mathcal{D}^\prime(\Omega)\to{^\rho\mathcal{E}}(\Omega)$ which preserves the sheaf-properties and
the linear operations in 
$\mathcal{D}^\prime(\Omega)$ (including partial differentiation) and such that
$\Sigma_\Omega(T_f)=\Sigma_\Omega(^*f)$ for every $f\in\mathcal{E}(\Omega)$. Consequently, the
multiplication in ${^\rho\mathcal{E}}(\Omega)$ reduces to the usual pointwise multiplication on
$\mathcal{E}(\Omega)$. We summarize this in: 
\[
\mathcal{E}(\Omega)\hookrightarrow\mathcal{D}^\prime(\Omega)\hookrightarrow{^\rho\mathcal{E}}(\Omega)
\]
\end{theorem}


\Proof We shall separate the proof in numerous definitions and lemmas: 

\begin{definition}[$\rho$-Delta Function] $D\in{^*\mathcal{E}}(\mathbb{R}^d)$  is called a
{\bf $\rho$-delta function} if:
\begin{enumerate}
\item $||x||\not\approx 0$ implies $D(x)=0$. ({\bf Lemma:} There exists a {\bf positive
infinitesimal}, say $\rho$, such that $||x||\leq \rho$ implies $D(x)=0$).

	The next conditions on $D$ depend on the choice of $\rho$:

\item $\int_{||x||\leq\rho}D(x)\, dx-1\in\mathcal{N}_\rho(^*\mathbb{C})$.

\item $\int_{||x||\leq\rho} D(x)\, x^\alpha\, dx\in\mathcal{N}_\rho(^*\mathbb{C})$ for all
$|\alpha|\not= 0$.

\item $D\in\mathcal{M}_\rho(^*\mathcal{E}(\mathbb{R}^d))$, i.e.
\[
(\forall \alpha\in\mathbb{N}_0^d)(\forall x\in\mu(\mathbb{R}^d))\left(\partial^\alpha
D(x)\in\mathcal{M}_\rho(^*\mathbb{C})\right).
\]
\end{enumerate}
\end{definition}
\begin{theorem} There exists a $\rho$-delta function $D$.
\end{theorem}
  
\Proof: For the original proof we refer to (M. Oberguggenberger and T.
Todorov~\cite{OberTod98}). Here is a {\bf summary of this result}:

{\bf Step 1)}  For every
$n\in\mathbb{N}$, we define the set of test-functions:
\begin{align}\label{E: A) Bn}
	 \mathcal{B}_n  =  \{&\varphi\in\mathcal{D}(\mathbb{R}^d) :\,  \\\notag															
 																					&\int_{\mathbb{R}^d}\varphi(x)\, dx=1,\\\notag
																						&\int_{\mathbb{R}^d}x^\alpha\varphi(x)\, dx=0 \text{\; for all\; }
\alpha\in\mathbb{N}_0^d,\; 1\leq|\alpha|\leq n,\\\notag
 &||x||\geq 1/n \Rightarrow\varphi(x)=0,\\\notag
																						&1\leq\int_{\mathbb{R}^d}|\varphi(x)|\, dx<  1+\frac{1}{n}\, \}.
\end{align}

\begin{lemma}[Properties of $\mathcal{B}_n$]\label{L: Properties of Bn}  

	{\bf ($B_1$)} $\mathcal{B}_n\not=\varnothing$ for all $n$.

	{\bf ($B_2$)} $\mathcal{D}(\mathbb{R}^d)=\mathcal{B}_0\supset\mathcal{B}_1\supset\mathcal{B}_2\supset\mathcal{B}_3\supset\dots$.
(Thus $\mathcal{B}_n\cap\mathcal{B}_n=\mathcal{B}_{\max{(m, n)}}$).

	{\bf ($B_3$)} $\cap_n\; \mathcal{B}_n=\varnothing$. 
\end{lemma}

{\bf Step 2)} Find the {\bf non-standard extension} of $\mathcal{B}_n$:

\begin{align}\label{E: A) Bn}
	 ^*\mathcal{B}_n  =  \{&\varphi\in{^*\mathcal{D}}(\mathbb{R}^d) :\,  \\\notag
 																					&\int_{^*\mathbb{R}^d}\varphi(x)\, dx=1,\\\notag
																						&\int_{^*\mathbb{R}^d}x^\alpha\varphi(x)\, dx=0 \text{\; for all\; }
\alpha\in\mathbb{N}_0^d,\; 1\leq|\alpha|\leq n,\\\notag
 &||x||\geq 1/n \Rightarrow\varphi(x)=0,\\\notag
																						&1\leq\int_{^*\mathbb{R}^d}|\varphi(x)|\, dx<  1+\frac{1}{n}\, \}.
\end{align}

	{\bf Step 3)} Let $M$ be an infinitely large positive number in
$\mathcal{M}_\rho(^*\mathbb{R})$. For example, $M=|\ln{\rho}|$ will do. Define the
internal sets:
\[
\mathcal{A}_n=\{\varphi\in{^*\mathcal{B}_n} :\quad {^*\sup}_{||x||\leq
1/n}|\partial^\alpha\varphi(x)|<\frac{M}{n} \text{\, for all\, } |\alpha|\leq n \},
\]
We observe that (trivially)
$^*\mathcal{D}(\mathbb{R}^d)\supset\mathcal{A}_1\supset\mathcal{A}_2\supset\dots$.  Also,
$\mathcal{A}_n\not=\varnothing$ for all $n$. Indeed, $\varphi\in\mathcal{B}_n$ implies
$^*\varphi\in\mathcal{A}_n$ since 
\[
{^*\sup}_{||x||\leq 1/n}|\partial^\alpha(^*\varphi(x))|={\sup}_{||x||\leq
1/n}|\partial^\alpha\varphi(x)|< \frac{M}{n},
\]
and ${\sup}_{||x||\leq
1/n}|\partial^\alpha\varphi(x)|$ is a real number and $M/n$ is an infinitely large positive number for
any
$n\in\mathbb{N}$. Thus there exists
\[
\Theta\in\bigcap_{n=1}^\infty\, \mathcal{A}_n \not=\varnothing,
\]
by Saturation Principle (Theorem~\ref{T: Sequential Saturation}). Notice that $\Theta$ {\bf
satisfies all properties (1)-(4)} of the definition of $\rho$-delta function {\bf except (possibly)
the property (5)}.

{\bf Step 3)} The non-standard function $D\in{^*\mathcal{D}}(\mathbb{R}^d)$, defined by the formula
\[
D(x)=\rho^{-d}\Theta(x/\rho),
\]
is the $\rho$-delta function we are looking for.
\begin{definition} The mapping $T\to Q_\Omega\left(^*T\star D\right)$ from
$\mathcal{E}^\prime(\Omega)$ to $^\rho\mathcal{E}(\Omega)$ is the embedding of the space of
distributions with compact support in $\Omega$.
\end{definition}

{\bf Step 4)} 
\begin{definition}[$\rho$-Cut-Off Function] $\Pi_\Omega\in{^*\mathcal{D}}(\Omega)$ is called a {\bf
$\rho$-cut-off function} for the open set $\Omega\subseteq\mathbb{R}^d$ if 
\begin{description}
\item{\bf (a)} $\Pi_\Omega(x)=0$ for all $x\in\mu(\Omega)$.
\item{\bf (b)}  $\supp(\Pi_\Omega)\subseteq\{x\in{^*\Omega}\mid {^*d(x, \partial\Omega)}\geq\rho\}$
\end{description}
\end{definition}
\begin{lemma} There exists a $\rho$-cut-off-function.
\end{lemma}
\Proof Let $\Omega_\rho =\{x\in{^*\Omega}\mid {^*d(x, \partial\Omega)}\geq 2\rho,
||x||<1/\rho\, \}$ and let $\chi$ be the characteristic function of $\Omega_\rho$. The function
$\Pi_\Omega= \chi\star D$ is the $\rho$-cut-off function we are looking for. $\blacktriangle$
\begin{definition} The mapping $T\to Q_\Omega\left(^*T\Pi_\Omega)\star D\right)$ from
$\mathcal{D}^\prime(\Omega)$ to $^\rho\mathcal{E}(\Omega)$ is the embedding the existence of which was
stated in Theorem~\ref{T: Existence of an Embedding}. 
\end{definition}
	
	The proof of Theorem~\ref{T: Existence of an Embedding} is complete. $\blacktriangle$

\newpage


\begin{thebibliography}{99}

\bibitem{ArigBiag91} J.~ Arigona and H.~A.~ Biagioni, {\em Intrinsic Definition of the Colombeau Algebra of 	Generalized Functions},
Analysis Mathematica, 17 (1991), p.~75--132.

\bibitem{nAron73} N.~ Aronszajn, {\em Traces of analytic solutions of the heat equation}, Colloq. Internat. Ast\'{e}rique 2 et 3, C.N.R.S.,
Soc. Math. France, 1973,  p.~35--68 (\S7, Chapter VI). 

\bibitem{bBan} B. Banaschewski, {\em On Proving the Existence of Complete Ordered Fields}, Math. Monthly, Vol. 105, No. 6, 1998, p.~548-551.

\bibitem{msBaouendi75} M.~S.~ Baouendi, {\em Solvability of partial differential equations in the traces of analytic 	
      		solutions of the heat equation}, Amer. J. Math. {\bf 97}, 1975, p.~983-1005 (\S7, Chapter VI). 

\bibitem{aBiag90} Hebe A. Biagioni, {\em A Nonlinear Theory of Generalized Functions}, Lecture Notes in	Mathematics, Springer Verlag, Vol.
1421, XII, 1990.

\bibitem{Bourbaki} N. Bourbaki, {\em Algebra II}, Springer-Verlag, Berlin-Heidelberg-New York, 1990.

\bibitem{hBremermann65} H. Bremermann, {\em Distributions, Complex Variables, and Fourier Transforms},
 Addison-Wesley Publ. Co., Inc., Palo Alto, 1965.

\bibitem{CKeis} C. C. Chang and H. Jerome Keisler, {\em Model Theory}, Studies in Logic and the Foundations of Mathematics, Vol. 73,
Elsevier, Amsterdam, 1998. 

\bibitem{ChrTod84}  Chr. Ya. Christov, T. D. Todorov, {\em Asymptotic Numbers: Algebraic 	Operations with Them}, In Serdica, Bulgaricae
Mathematicae  Publicationes, Vol.2, 1974, p.~87--102. 

\bibitem{jCol84a} J. F. Colombeau,  {\em New Generalized Functions and Multiplication of Distributions}, 
		North-Holland Math. Studies 84, 1984.

\bibitem{jCol8484b} J. F. Colombeau, {\em New General Existence Result for Partial Differential Equations in
	the $C^\infty$\!~- Case}, Preprint, Universite de Bordeaux, 1984.

\bibitem{jCol85} J. F. Colombeau,  {\em Elementary Introduction to New Generalized Functions}, 
North-Holland, Math. Sdudies 113, Amsterdam, 1985.

 
\bibitem{jCol90} J. F. Colombeau, {\em Multiplication of Distributions}, Bull.A.M.S. 23, 2, 1990, 	p.~251--268.

 \bibitem{jCol92} J. F. Colombeau, {\em Multiplication of Distributions;  A Tool in Mathematics, Numerical 	Engineering and Theoretical
Physics}, Lecture Notes in Mathematics, 1532, Springer-	Verlag, Berlin, 1992.

\bibitem{jCol91} J. F. Colombeau, A. Heibig, M. Oberguggenberger, {\em Generalized solutions to
	partial differential equations of evolution type}, Preprint of Ecole Normale Superieure de Lyon, France, 1991.

\bibitem{jConway}John B. Conway, {\em A Course in Functional Analysis}, Second Edition, 
Springer-Verlag, New York, Berlin, Heidelberg, 1990.

\bibitem{DalWoodin} H. Garth Dales and W. Hugh Woodin, {\em Super-Real Fields: Totally Ordered Fields with Additional Structure}, Oxford Science
Publications, London Mathematical Monographs-New Series {\bf 14}, Clarendon Press, Oxford, 1996.

\bibitem{mDavis} Martin Davis, {\em Applied Nonstandard Analsysis}, Dover Publications, 
Inc., Mineola, New York, 2005.

\bibitem{bDi} B. Diarra, {\em Ultraproduits ultram\'{e}triques de corpsvalu\'{e}s}, Ann. Sci. Univ. Clermont II, S\'{e}r. Math., Fasc. {\bf 22}
(1984), pp. 1-37.

\bibitem{DriesMacintyreMarker} L. Van den Dries, A. MacIntyre, D. Marker, {\em Logaritmic-Exponential Power
Sereis}, J. London Math. Soc. 56 (1997), pp. 417-434.

\bibitem{DunSchwartz58} N. Dunford and J. T. Schwartz, {\em Linear Operators, Part I: General Theory},	Interscience Publishers, Inc. New
York, 1958.

\bibitem{yEgorov90a} Yu. V. Egorov, {\em A contribution to the theory of generalized functions} (Russian). 	Uspekhi Mat Nauk (Moskva)
45, No 5 (1990), 3 - 40. English transl. in Russian Math. Surveys {\bf 45}  (1990), p.~1--40. 

\bibitem{yEgorov90b} Yu. V. Egorov, {\em On generalized functions and linear differential equations} Vestnik Moskov. Univ. Ser. I, (1990),
No 2, p.~92--95. (Russian)

\bibitem{lEhren56}  L. Ehrenpreis, {\em Solutions of some problems of division III}, Amer. J. Math., {\bf 78} (1956), p.~685.	

\bibitem{ErdGillHenr55} P. Erd\"{o}s, L. Gillman and M. Henriksen, {\em An Isomorphism Theorem for Real-Closed Fields}, Annals of
Mathematics, Vol. 61, No. 3, May 1955, p.~542--560.

\bibitem{EstradaKanwal} R. Estrada and R. P. Kanwal, {\em Asymptotic Analysis: A Distributional Approach}, Birkh\"{a}user, Boston, Basel, Berlin,
1994.

\bibitem{FishTod91} B. Fisher and T. Todorov, {\em A Commutative Product with Distribution Vectors}, Journal of Mathematical and Physical 	
	Sciences, Vol. 25, No 2, April 1991, p.~138--151. 

\bibitem{FishTod87} B. Fisher and T. D. Todorov, {\em Operations with Distribution Vectors},	In Demonstratio Mathematica, Vol. XX, No 3-4,
1987, p.~401--412. 

\bibitem{lFux63} L. Fuchs, {\em Partially Ordered Abelean Groups}, Pergamon Press, Oxford, 1963.

\bibitem{fqGouv} Fernando Q. Gouv, {\em $p$-Adic Numbers. An Introduction}, Springer, 1997.

\bibitem{mGrosser at al 1} M. Grosser, Eva Farkas, M. Kunzinger and R. Steinbauer, {\em On the Foundations of
Nonlinear Generalized Functions I and II},  Memoirs of the AMS, Vol. 153, Number 729, American
Mathematical Society, 2001.

\bibitem{mGrosser at al 2} M. Grosser, M. Kunzinger, M. Oberguggenberger and R. Steinbauer, {\em Geometric
Theory of Generalized Functions withy Applications to General Relativity}, Vol. 537 of Mathematics and Its
Applications, Kluwer Academic Publishers, Dordrecht, 2001.

\bibitem{hHahn} H. Hahn, {\em \"{U}ber die nichtarchimedischen Gr\"{o}ssensysteme}, Sitz. K. Akad. Wiss. {\bf 116} (1907), p.~601-655.
	
\bibitem{lHor63} L. H\"{o}rmander, {\em Linear partial differential operators}, Grund.  Math. Wis., 116 (1963), 	Springer Verlag.

\bibitem{HosPinto93} R. F. Hoskins, J. Sousa Pinto, {\em Nonstandard Treatments of New Generalised 	Functions}, In: R.S. Pathak (Ed.),
``Generalized Functions and Their Applications'', Plenum Press, New York, 1993, p. 95-108.  

\bibitem{HurdLoeb85} A.E. Hurd and P.A. Loeb, {\em An Introduction to Nonstandard Real Analysis}, Academic
																					Press, 1985.

\bibitem{wIngleton}W. Ingleton, {\em The hahn-banach theorem for non-archimedean valued fields},
Proc. Cambridge Phil. Soc. 48 (1952), pp 41-45.

\bibitem{rJin} R. Jin, {\em Better nonstandard universes with applications},
in: ``Nonstandard Analysis: Theory and Applications'' (L. O. Arkeryd, N. J. Cutland, and C. W. Henson, eds.), NATO ASI Series, Series
C: {\em Mathematical and Physical Sciences}, v. 493, Kluwer Acad. Publishers, 1997.

\bibitem{aKan88} A. Kaneko, {\em Introduction to Hyperfunctions}, Kluwer Acad. Publ., Dordrecht 1988. 

\bibitem{iKa} I. Kaplansky, {\em Maximal fields with valuations}, Duke
Math. J. {\bf 9} (1942), p~. 303-21.

\bibitem{jKeisE} H. J. Keisler, {\em Elementary Calculus},  Prindle, Weber \& Schmidt, Boston, 1976.

\bibitem{jKeisF} H. J. Keisler,  {\em Foundations of Infinitesimal Calculus}, Prindle, Weber \& Schmidt,
Boston, 1976.

\bibitem{hKomatsu} Hikosaburo Komatsu, {\em Ultradistributions. I: Structure theorems 
and characterization}, J. Fac. Sci., Univ. Tokyo, Sect. I A 20, 25-105 (1973).

\bibitem{wKrull} W. Krull, {\em Allgemeine Bewertungstheorie}, Journal
f\"{u}r die reine und angewandte Mathematik {\bf 167} (1932), p. 160-196.

\bibitem{KuhlKuhlShel} F. V. Kuhlmann, S. Kuhlmann and S. Shelah, {\em Exponentiation in Power Series Fields}, Proc. Amer.
Math. Soc., Vol. 125, N.11, 1997, p.~3177--3183.

\bibitem{KunzOber99} M. Kunzinger and M. Oberguggenberger, {\em Characterization of
Colombeau generalized functions by their  pointvalues}, Math. Nachr., 203, 147-157, 1999.

\bibitem{sLang} S. Lang, {\em Algebra}, Addison-Wesley, 1992.

\bibitem{dLaug59} D. Laugwitz, {\em Eine EinfŸhrung der $\delta$-Funktionen}, Sitzungsber. Bayerische
Akad. der 	Wissenschaften, 4, 41-59 (1959).

\bibitem{dLaug61a} D. Laugwitz, {\em Anwendung unendlich kleiner Zahlen. I. Zur Theorie der Distributionen}, 	J. reine angew.  Math. 207,
53-60  (1961).

\bibitem{dLaug61b} D. Laugwitz, {\em Anwendung unendlich kleiner Zahlen. II. Ein Zugang zur Operatorenrechnung von Mikusinski}, J. reine
angew.  Math. 208, 22--34 (1961).

\bibitem{dLaug68} D. Laugwitz, {\em Eine Nichtarchimedische Erweiterung Angeordneter K\"{o}rper}, Math. Nachrichte {\bf 37} (1968), 
p.~225--236.

\bibitem{hLewy57} H. Lewy, {\em An example of a smooth linear partial differential equation without solution},	Annals of Mathematics,
Vol. 66,  No. 1,  July,  1957.

\bibitem{leBangHe78} Li Bang-He, {\em Non-Standard Analysis and multiplication of disributions}, Sci. Sinica, 21 	(1978), 561-585.

\bibitem{tLC} T. Levi-Civita, {\em Sugli Infiniti ed Infinitesimi Attuali
Quali Elementi Analitici} (1892-1893), Opere Mathematiche, vol. 1,
Bologna (1954), p.~1--39.

\bibitem{LiBang} Li Bang-He, {\em Non-standard analysis and multiplication of
distributions}, Sci. Sinica {\bf 21} (1978), p.~561--585

\bibitem{tLin} T. Lindstr\o m, {\em An invitation to nonstandard analysis},
in: Nonstandard Analysis and its Applications, N. Cutland (Ed), Cambridge U. Press, 1988, p.~1--105.

\bibitem{LiRob} A. H. Lightstone and A. Robinson, {\em Nonarchimedean Fields and
Asymptotic Expansions}, North-Holland, Amsterdam, 1975. 

\bibitem{wLux} W. A. J. Luxemburg, {\em On a class of valuation fields
introduced by A. Robinson}, Israel J. Math. {\bf 25} (1976), p.~189--201.

\bibitem{sMacLane}  S. MacLane, {\em The Universality of Formal Power Series Fields}, Bull of 	AMS, 45, p.~888--890, Theorem 1, p.~889. 

\bibitem{sMacL} S. MacLane, {\em The uniqueness of the power series
representation of certain fields with valuation}, Annals of Math. {\bf 39} (1938), p.~370-382.

\bibitem{bMal59}  B. Malgrange, {\em Sur la propagation de la regularite des solutions des equations a coefficients constants}, Bull.
Math Soc. Sci. Math. Phys. R. P. Roumuanie {\bf 3} (1959), p.~433 - 440.

\bibitem{MarkerMessPill} D. Marker, M. Messmer, A. Pillay, {\em Model Theory of Fields}, 
Lecture Notes in Logic 5, Springer-Verlag, Berlin Heidelberg, 1996.

\bibitem{eMayerhofer} Eberhard Mayerhofer, {\em The wave equation on singular space-times}, A Ph.D.
Thesis, Faculty of Mathematics, University of Vienna, 1090 Vienna, Austria.

\bibitem{eNelson77} E. Nelson, {\em Internal Set Theory},  Bull. Amer. Math. Soc. 83 (1977), p.~1165-1193.

\bibitem{bNeumann49} B. H. Neumann, {\em On Ordered Division Rings}, Trans. Am. Math. Soc. 66 	(1949), p.~202-252.

\bibitem{mOber88} M. Oberguggenberger, {\em Products of distributions:  nonstandard
methods}, Z.  Anal. Anwendungen {\bf 7}(1988),  347 - 365.  Corrections:
ibid. {\bf 10}(1991), 263-264.

\bibitem{mOber92} M. Oberguggenberger, {\em Multiplication of Distributions and Applications to Partial 	Differential Equations}, Pitman
Research Notes Math., 259, Longman, Harlow, 1992. 

\bibitem{OberTod98} M. Oberguggenberger and T. Todorov, {\em An Embedding of  Schwartz Distributions in the
Algebra of Asymptotic Functions},
      International J. Math. \& Math. Sci., Vol. 21, No. 3 (1998), p.~417--428. 

\bibitem{mOber} M. Oberguggenberger, {\em Contributions of nonstandard analysis to partial differential equations}, in: Developments in
Nonstandard Mathematics (Eds. N.J. Cutland, V. Neves, F. Oliveira and J. Sousa-Pinto), Longman Press, Harlow, 1995, pp. 130-150.

\bibitem{ggPestov80} G. G. Pestov, {\em Structure of Ordered Fields}, Toms University Press, 	Tomsk, 1980 (in Russian).

\bibitem{vPes} V. Pestov, {\em On a valuation field invented by A.
Robinson and certain structures connected with it}, Israel J. Math. {\bf
74} (1991), p.~65--79.

\bibitem{mRadynaICGF2000} Mikalai Radyna, {\em New
Model of Generalized Functions and Its Applcations to Hopf
Equation}, in A. Delcroix, M. Hasler, J.-A. Marti, V. Valmorin
(Eds.), {\em Nonlinear Algebraic Analysis and Applications,
Proceedings of the International Conference on Generalized
Functions (ICGF 2000)}, Cambridge Scientific Publ., Cottenham
2004, p. 269-288.

\bibitem{pRibenboim} P. Ribenboim, {\em The Theory of Classical Valuation},
Springer Monographs in Mathematics, Springer-Verlag, 1999.

\bibitem{amRobert} A. M. Robert, {\em A course in $p$-adic analysis}, vol. 198 of Graduate Texts in
Mathematics, Springer-Verlag, New York, 2000.

\bibitem{aRob66} A. Robinson, {\em Nonstandard Analysis}, North Holland, Amsterdam, 1966.

\bibitem{aRob73} A. Robinson, {\em Function theory on some nonarchimedean fields}, Amer. Math. Monthly 80 (6),
Part~II: Papers in the Foundations of Mathematics (1973), p.~87--109.

\bibitem{aRob72} A. Robinson, {\em On the Real Closure of a Hardy Field (1972)}, In Selected Papers of 
Abraham Robinson, Volume~1: Model Theory and Algebra (Edited by H. J. Keisler), 
New Haven and London, Yale University Press, 1979.

\bibitem{aRob79} A. Robinson, {\em Germs}, In Selected Papers of Abraham Robinson, Volume~1: 
Nonstandard Analysis and Philosophy (Edited by W.A.J. Luxemburg and S. K\"{o}rner), 
New Haven and London, Yale University Press, 1979.

\bibitem{eRosinger80} E. E. Rosinger, {\em Nonlinear Partial Differential Equations: Sequential and Weak
Solutions}, North-Holland, Math. Studies 44, Amsterdam, 1980.

\bibitem{eRosinger87} E. E. Rosinger, {\em Generalized Solutions of Nonlinear Partial Differential Equations}, 
	North-Holland, Math. Studies 146, Amsterdam, 1987.

\bibitem{eRosinger90} E. E Rosinger, {\em Non-linear Partial Differential Equations. An Algebraic View of 
	Generalized Solutions}, North-Holland, Math. Studies Vol. 164, Amsterdam 1990.

\bibitem{SalbTod98} S. Salbany and T. Todorov, {\em Nonstandard Analysis in Point-Set Topology}, Lecture Notes No. 666, 1998 (52 pages) of Erwing
			Schr\"{o}dinger Institute for Mathematical Physics, Vienna (ftp at  ftp.esi.ac.at,  URL: 	http://www.esi.ac.at).

\bibitem{pSchapira67} P. Schapira, {\em Une equation aux derivees partielles sans solutions dans l'espace des 
	hyperfonctions}, C.R. Acad. Sci. Paris Ser. A-B 265 (1967), A665-A667. MR 36, \#4112.

\bibitem{SchmiedenLaug66} C. Schmieden and D. Laugwitz, {\em Eine Erweiterung der Infinitesimalrechnung}, Math. 	Zeitschr., 69, p.~1-39
(1958).

\bibitem{lSchwartz66} L. Schwartz, {\em ThŽorie des Distributions}, Hermann, 1966.

\bibitem{l.Schwartz54} L. Schwartz, {\em Sur l'impossibilitŽ de la multiplication des distributions}, C.R.Acad.Sci.,	Paris 239 (1954), 
p.~847-848.

\bibitem{StroLux76} K. D. Stroyan and W. A. J. Luxemburg, {\em Introduction to the Theory of
Infinitesimals}, Academic Press, New York, 1976.

\bibitem{tdTod81e}  T. D.~Todorov, {\em The Products  $\delta^2(x),\, \delta(x)\, x^{-n},\ H(x)\, x^{-n}$ in the 	Class of the Asymptotic
      Functions}, In Bulg. J. Phys., Vol.12 (1981), 5, p.~465--480.

\bibitem{tdTod86} T.~D.~Todorov, {\em Application of Non-Standard Hilbert Space to	Quantum Mechanics}, 
In Proceedings of the International Conference on Complex Analysis and Applications, Varna, May 5-11, 1985,
Bulgarian Academy of Sciences Publ., 1986, p.~689--704.

\bibitem{tdTod88}  T.D. Todorov,  {\em Colombeau's new generalized functions and non-standard analysis},  In: B.
Stankovic, E. Pap, S. Pilipovic, V.S. Vladimirov (editors), ``Generalized Functions, Convergence 	Structures and
their Applications'', Plenum Press, New York (1988), p.~327--339. 


\bibitem{tdTod90}  Todor Todorov, {\em A Nonstandard Delta Function}, In Proc. Amer. 	Math. 
Soc., Vol. 110,  Number 4, 1990, p.~1143--1144.

\bibitem{tdTod92} Todor Todorov, {\em Kernels of Schwartz Distributions}, In Proc. Amer. Math. Soc., 
Vol. 114,  No 3,  March 1992, p.~817--819.

\bibitem{tdTod95}  Todor Todorov, {\em An Existence Result for a Class of Partial 	Differential Equations with Smooth Coefficients},  In S.
Albeverio, W.A.J. Luxemburg, M.P.H. Wolff (Eds.), ``Advances in 	Analysis, Probability and Mathematical Physics; Contributions to 
Nonstandard Analysis'',  Kluwer Acad. Publ., Dordrecht, Vol. 314, 1995, p.~ 107--121. 

\bibitem{tdTod96} Todor Todorov, {\em An Existence of Solutions for Linear PDE with $C^{\infty}$-Coefficients in an Algebra of Generalized
Functions}, in	Transactions of the American Mathematical Society, Vol. 348, 2, 	Feb. 1996, p.~ 673--689.

\bibitem{tdTod99} T. Todorov, {\em Pointwise Values and Fundamental Theorem in the Algebra of Asymptotic Functions}, 
in Non-Linear Theory of Generalized Functions (Eds: M. Grosser, G\"{u}nther H\"{o}rmann, M. Kunzinger and M.
Oberguggenberger), Chapman \& Hall/CRC Research Notes in Mathematics, 401, 1999, p.~369-383,
arXiv:math.FA/0601723.

\bibitem{tdTod2000a} Todor~D. Todorov, {\em Back to Classics: Teaching Limits through 	Infinitesimals},
International Journal of Mathematical Education in Science and Technology, 2001, 
vol. 32, no. 1, p. 1-20.

\bibitem{TodWolf} Todor Todorov and Robert Wolf, {\em Hahn Field Representation of A. Robinson's
Asymptotic Numbers}, in Nonlinear Algebraic Analysis and Applications, Proceedings of the ICGF 2000 
(Edited by A. Delcroix, M. Hasler, J.-A. Marti, V. Valmorin), 2004 Cambridge Scientific Publishers, 
p. 357-374, ArXiv:math.AC/0601722. 

\bibitem{AsyCol} Todor D. Todorov, {\em Existence and uniqueness of $v$-asymptotic expansions and
Colombeau's generalized numbers}, Journal of Mathematical Analysis and Applications, Volume 312,
Issue 1, 1 December, 2005, p. 261-279, arXiv:math.CA/0601720.

\bibitem{fTreves70} Francois Treves, {\em On Local Solvability of Linear Partial Differential Equations}, 
	Bulletin of the American Mathematical Society, Vol.76, No 3, May, 1970, p.~552-5.

\bibitem{hVernaevePhD} Hans Vernaeve, {\em Nonstandard Contributions to the Theory of
Generalized Functions}, Ph.D. Thesis, Universy of Gent, Belgium, March, 2002.

\bibitem{hVernaeveEdin} Hans Vernaeve, {\em Optimal embeddings of distributions into algebras},
Proc. Edinburgh Math. Soc. (2003) 46: 373-378.

\bibitem{vVladimirov} V. Vladimirov, {\em Generalized Functions in Mathematical Physics},
Mir-Publisher, Moskow, 1979.

\bibitem{VanDerWaerden} B. L. Van Der Waerden, {\em Modern Algebra}, Ungar Publishing, New York, third
printing, 1964.

\end{thebibliography}
\end{document}